\documentclass[a4paper, reqno]{amsart}

\usepackage{amsfonts, amsmath,  amssymb, color,colortbl, enumerate, esint, fancyvrb, mathrsfs, tikz}

\usepackage{pgfplots, subcaption}

\pgfplotsset{compat=1.10}
\usepgfplotslibrary{fillbetween}
\usetikzlibrary{patterns}

\pgfdeclarelayer{ft}
\pgfdeclarelayer{bg}
\pgfsetlayers{bg,main,ft}

\usetikzlibrary{positioning}
\VerbatimFootnotes

\usepackage[framemethod=TikZ]{mdframed}

\mdfdefinestyle{MyFrame}{%
    linecolor=blue,
    outerlinewidth=1.5pt,
    roundcorner=20pt,
    innertopmargin=5pt,
    innerbottommargin=\baselineskip,
    innerrightmargin=20pt,
    innerleftmargin=20pt,
    backgroundcolor=gray!10!white}

\title
[Fourier restriction in higher dimensions]{Improved Fourier restriction estimates\\ in higher dimensions}

\author{Jonathan Hickman}
\address{Mathematical Institute, University of St Andrews, North Haugh,
St Andrews, Fife, KY16 9SS, UK.}
\email{jeh25@st-andrews.ac.uk}

\author{Keith M. Rogers}
\address{Instituto de Ciencias Matem\'aticas CSIC-UAM-UC3M-UCM, Cantoblanco, Madrid, 28049, Spain.}
\email{keith.rogers@icmat.es}



\thanks{Partially supported by the EPSRC standard grant EP/R015104/1, the NSF grant DMS-1440140, and the MINECO grants SEV-2015-0554 and MTM2017-85934-C3-1-P}

\theoremstyle{plain}
\newtheorem{theorem}{Theorem}[section]
\newtheorem{lemma}[theorem]{Lemma}
\newtheorem{proposition}[theorem]{Proposition}

\newtheorem*{key estimate}{Key estimate}

\theoremstyle{definition}
\newtheorem{definition}[theorem]{Definition}

\newtheorem{remark}[theorem]{Remark}
\newtheorem*{acknowledgement}{Acknowledgement}
\newtheorem*{notation}{Notation}

 \renewcommand{\le}{\leqslant}
\renewcommand{\ge}{\geqslant}
\renewcommand{\leq}{\leqslant}
\renewcommand{\geq}{\geqslant}

\newcommand{\ta}{\texttt{a}}
\newcommand{\bta}{\mbox{\small$\ta$}}
\newcommand{\sta}{\mbox{\scriptsize$\ta$}}
\newcommand{\tc}{\texttt{c}}
\newcommand{\btc}{\mbox{\small$\tc$}}
\newcommand{\stc}{\mbox{\scriptsize$\tc$}}

\renewcommand{\O}{\O}

\newcommand{\N}{\mathbb{N}}
\newcommand{\Z}{\mathbb{Z}}
\newcommand{\R}{\mathbb{R}}
\newcommand{\T}{\mathbb{T}}
\newcommand{\C}{\mathbb{C}}

\newcommand{\ud}{\mathrm{d}}

\renewcommand{\O}{\mathcal{O}}
\newcommand{\Sc}{\mathcal{S}}

\newcommand{\fh}{\mathfrak{h}}
\newcommand{\fB}{\mathfrak{B}}

\newcommand{\bZ}{\mathbf{Z}}
\newcommand{\bY}{\mathbf{Y}}

\newcommand{\cB}{\mathcal{B}}

\newcommand{\sE}{\mathscr{E}}

\newcommand{\BL}[1]{\mathrm{BL}_{#1}}
\newcommand{\tang}{\mathrm{tang}}
\newcommand{\trans}{\mathrm{trans}}
\newcommand{\cell}{\mathrm{cell}}
\newcommand{\alg}{\mathrm{alg}}
\newcommand{\Deg}{\overline{\deg}\,}

\newcommand{\Dec}{\mathrm{RapDec}(r)\|f\|_{2}}

\begin{document}
\begin{abstract} We consider Guth's approach to the Fourier restriction problem via polynomial partitioning. By writing out his induction argument as a recursive algorithm and introducing new geometric information, known as the polynomial Wolff axioms, we obtain improved bounds for the restriction conjecture, particularly in high dimensions. Consequences for the Kakeya conjecture are also considered.
\end{abstract}

\maketitle

\section{Introduction}

We consider the Fourier transform defined, initially on integrable functions, by 
$$
\widehat{f}(\xi):=\int_{\R^n} f(x)\, e^{-i \langle x,\xi\rangle} \ud x.
$$
Letting $\sigma$ denote the surface measure on a truncated piece of the paraboloid, Stein's restriction conjecture~\cite{Stein1979} asserts that the {\it a priori} estimate
\begin{equation}\label{restriction}\tag{$\mathrm{R}_p$}
\|\widehat{f}\,\|_{L^{p'}\!(d\sigma)} \le C_{n,p}\|f\|_{L^{p'}\!(\R^n)}
\end{equation}
holds for all $p>\frac{2n}{n-1}$, where $1/p+1/p'=1$. This was proved by Fefferman and Stein in two dimensions~\cite{Fefferman1970}, but remains open in higher dimensions despite extensive study; see, for example,~\cite{BCT2006, Bourgain1991, BG2011, CS, Demeter,   Guth2016,  Guth,  
Moyua1996, 
Ramos,
Strichartz1977,
Tao2003,  TV2000, TVV1998, Temur2014, Tomas1975, 
 Wang, Wolff2001} and the references therein.

The strongest partial results  are based on the polynomial partitioning
method, introduced to the problem by Guth~\cite{Guth2016, Guth}. In this article further progress is
obtained  by augmenting the method
with additional geometric inequalities recently established in work of Katz and the second
author~\cite{KR}.

Our results are most easily compared with the previous literature in the high dimensional context. If the restriction conjecture were true,  \eqref{restriction} would hold for
$$p>2+2n^{-1}+O(n^{-2}),$$ and so we consider  $\lambda\ge 2$ for which we can confirm that \eqref{restriction} holds in the range
\begin{equation}\label{range}
 p>2+\lambda n^{-1}+O(n^{-2}).
\end{equation}
A consequence of the work of Tomas~\cite{Tomas1975}  is that
$\lambda$ can be taken to be~4. Although many refinements were made since (including the work of Tao~\cite{Tao2003} which removed the $O(n^{-2})$-term with $\lambda=4$),  the linear coefficient  was not improved for some thirty-five years when Bourgain and Guth~\cite{BG2011} showed that it can be lowered to~$3$. Most recently, Guth~\cite{Guth} proved that~$\lambda$ can be taken to be  $8/3$. We improve this as follows:

\begin{theorem}\label{asymptotic theorem} \eqref{restriction} holds in the range \eqref{range} with  $\lambda=4/(5-2\sqrt{3})$.
\end{theorem}

We also obtain concrete improvements on the range of exponents for \eqref{restriction} in all dimensions $n \geq 3$ except $n=3,6,8,10$ or $12$. In these exceptional cases the current best results are due to Wang~\cite{Wang} when $n=3$ and Guth~\cite{Guth} when $n= 6,8,10$ or $12$. The current state-of-the-art for the restriction problem in various low dimensions is tabulated below in Figure~\ref{exponent table}. 

The proof of Theorem~\ref{asymptotic theorem} relies on geometric information coming from a recent result in~\cite{KR}.  This geometric information, which we will refer to as the \emph{polynomial Wolff axioms},  bounds the number of direction-separated line segments that can be contained in the neighbourhood of a real algebraic variety.

The present analysis extends that previously performed by Guth~\cite{Guth2016} in $\R^3$, who proved and applied the polynomial Wolff axiom for a two dimensional variety. Guth's induction argument ~\cite{Guth2016} can be combined with later developments from~\cite{Guth} and thereby directly extended to higher dimensions, using a single application of the $(n-1)$-dimensional polynomial Wolff axioms (see, for example,~\cite{Demeter} or~\cite{KR}), however this yields weaker results than those obtained here. We will take advantage of the polynomial Wolff axioms more often.

By combining the arguments of this article with results from~\cite{GHI}, one may also establish a version of Theorem~\ref{asymptotic theorem} for general positively-curved surfaces, including the unit sphere. It is also possible that the methods could be applied to study other oscillatory integral operators, such as those arising in the study of Bochner--Riesz multipliers, but this has not  been fully explored. Finally, by a standard argument relating the restriction and Kakeya conjectures, Theorem~\ref{asymptotic theorem} implies estimates for the Kakeya maximal function. This bound is new with $n=9$, however it does not improve the dimension estimate for Kakeya sets due to Katz--Tao~\cite{Katz2002}. Perhaps of more interest is the fact that these estimates provide an asymptotic improvement over the classical Wolff bound~\cite{Wolff1995} via a very different approach to that used in~\cite{Katz2002}.

The article is organised as follows:
\begin{itemize}
    \item A number of reductions are performed in the sequel. Following~\cite{BG2011, Guth2016, Guth}, the problem is reduced to establishing the so-called \emph{$k$-broad estimates} for the extension operator. 
    \item After setting up some notational conventions in Section~\ref{notation section}, a sketch of the proof of the main theorem is provided in Section~\ref{Wolff axioms section}.
    \item In Sections~\ref{k broad section}-\ref{wave packet decomposition section}, the basics of broad norms, polynomial partitioning and the wave packet decomposition are recalled.  
   \item In Section~\ref{tangential}, we show how the polynomial Wolff axiom theorem  can be used to improve certain estimates for averaged norms at different scales.
   \item In Section~\ref{structure lemma section}, Guth's polynomial partitioning argument from~\cite{Guth2016, Guth} is reformulated as a recursive algorithm. 
    \item In Section~\ref{proof section}, the new estimates are combined with the recursive algorithm to improve the range of estimates for the restriction problem.
    \item The final section contains a discussion of restriction to other hypersurfaces, some remarks on the numerology, and possible directions in which the argument could be strengthened. Finally, the application to the Kakeya problem is described. 
\end{itemize}

\begin{acknowledgement} The first author thanks Larry Guth and Hong Wang for some interesting discussions which greatly helped the development of this project. The authors also thank an anonymous referee for a thorough and helpful report.
\end{acknowledgement}




\section{Reduction to $k$-broad estimates}\label{reduction section}

Restriction estimates are typically proven via duality, with the adjoint operator~$E$ defined by
\begin{equation*}
E g(x) := \int_{|\xi|\le 1}
g(\xi)\,e^{i(x_1\xi_1+\ldots+x_{n-1}\xi_{n-1} +x_n|\xi|^2)}  \ud\xi.
\end{equation*}
Noting that now $\xi \in \mathbb{R}^{n-1}$ (and $x\in \mathbb{R}^n$ as before), this is often referred to as the \emph{extension operator}. It follows that the estimate \eqref{restriction} for a given value of $p$ is equivalent to the inequality 
\begin{equation*}
\|Eg\|_{L^p(\R^n)}\leq C_{n,p} \|g\|_{L^p(\mathbb{R}^{n-1})}.
\end{equation*}
Moreover, by a now standard \emph{$\varepsilon$-removal argument} (see~\cite{Tao1999}) and factorisation theory (see ~\cite{Bourgain1991} or ~\cite[Lemma 1]{Carbery1992}), this holds for all $p$ in an open range if and only if for all $\varepsilon>0$ and all $R \gg 1$ the local estimates
\begin{equation}\label{local extension}\tag{$\mathrm{R}^*_p$}
\|Eg\|_{L^p(B_{\>\!\!R})}\leq C_{n,p,\varepsilon} R^{\varepsilon}\|g\|_{L^\infty(\R^{n-1})}
\end{equation}
hold in the same range. Here $B_{\>\!\!R}$ denotes an arbitrary ball of radius $R$ in $\R^n$.

Rather than attempt to prove \eqref{local extension} directly, it is useful to work with a class of weaker inequalities known as \emph{$k$-broad estimates}. These inequalities were introduced by Guth~\cite{Guth2016, Guth} and were inspired by the earlier multilinear restriction theory developed in~\cite{BCT2006}. The $k$-broad estimates take the form
\begin{equation}\label{broad}\tag{$\BL{k}^p$}    
\|Eg\|_{\BL{k}^{p}(B_{\>\!\!R})} \leq C_{n,p,\varepsilon} R^{\varepsilon} \|g\|_{L^\infty(\R^{n-1})},
\end{equation}
where the expression on the left-hand side is known as a \emph{$k$-broad norm}. The precise definition of the $k$-broad norm is a little complicated and is deferred until Section~\ref{k broad section}. We remark, however, that the key advantage of working with $\|Eg\|_{\BL{k}^{p}(B_{\>\!\!R})}$ rather than $\|Eg\|_{L^p(B_{\>\!\!R})}$ is that the former expression is very small whenever the mass of $Eg$ is concentrated near a $(k-1)$-dimensional set (see Lemma~\ref{key k-broad lemma} below for a precise statement of this property). 

The main result of this article is the following theorem.

\begin{theorem}\label{main theorem} Let $2 \leq k \leq n-1$ and  
\begin{align}\label{with}
 p\ge p_n(k):= 2+\frac{8(2n-1)}{n(5n+2k-9)+k(k-3) +4}.
\end{align}
Then \eqref{broad} holds for all $\varepsilon>0$ and $R \gg 1$. 
\end{theorem}

When $n=3$ and $k=2$ this corresponds to the main result from~\cite{Guth2016} and stronger estimates are now known in this case~\cite{Wang}. In all other dimensions $n \geq 4$, Theorem~\ref{main theorem} offers an improvement over what was previously known. When $n=4$ and $k =3$, the range \eqref{with} extends that given by~\cite[Theorem 3.2]{Demeter}.\footnote{In~\cite{Demeter} it is shown that the $n=4$ and $k = 3$ case of Theorem~\ref{main theorem} would follow from a strengthened version of the polynomial Wolff axiom theorem from~\cite{KR} involving a polynomial dependence on the degree. For the purposes of this article, no such explicit dependence on the degree is required, and therefore the $3$-broad inequality in $\R^4$ is established in a larger range than that stated in~\cite{Demeter}.} When $n \geq 5$, Theorem~\ref{main theorem} strengthens a (corollary of a) theorem of Guth~\cite{Guth} which showed that the inequality \eqref{broad} holds whenever $p \geq 2 + \tfrac{4}{n+k-2}$;\footnote{In~\cite{Guth}, strengthened versions of \eqref{broad} are established with $L^2$ rather than $L^{\infty}$ norms appearing on the right-hand side, and so our estimates are stronger in one sense and weaker in another.} observe that the range \eqref{with} in Theorem~\ref{main theorem} is strictly larger than this. 

Unfortunately, since the $k$-broad estimates are weaker than the corresponding linear estimates, it is difficult to pass directly from an inequality of the form \eqref{broad} to one of the form \eqref{local extension}. Nevertheless, a mechanism developed by Bourgain and Guth~\cite{BG2011} allows this passage under certain constraints on the exponent $p$. In particular, the following proposition is a consequence of the method developed in~\cite{BG2011}, as observed in \cite[Proposition~9.1]{Guth}.

\begin{proposition}[Bourgain--Guth~\cite{BG2011}, Guth~\cite{Guth}]\label{9.1} Let $n\ge 3$ and 
\begin{equation*}
2+\frac{4}{2n-k} \le p\le 2+\frac{2}{k-2}.
\end{equation*}
Then \eqref{broad} implies \eqref{local extension}.
\end{proposition}

The original method of Bourgain--Guth~\cite{BG2011} was developed to convert certain multilinear inequalities of Bennett--Carbery--Tao~\cite{BCT2006} into linear estimates. It was later observed by Guth~\cite{Guth} that the method of~\cite{BG2011} does not require the full strength of the $k$-linear theory, but may instead take $k$-broad estimates as its input (which appear to be somewhat easier to prove\footnote{See \cite[Section 6.2]{GHI} for a detailed discussion of the relationship between $k$-broad and $k$-linear inequalities.}).

\renewcommand{\arraystretch}{1.2}
 \begin{figure}
\begin{tabular}{ ||c|c|c||c|c|c|| } 
 \hline
  $n=$ & $p >$  & \textbf{} & $n=$ & $p > $  & \textbf{} \\ 
   \hline 
  2 & \cellcolor{white} 4 & \cellcolor{white} Fefferman--Stein~\cite{Fefferman1970}  & 11 &  \cellcolor{yellow!50} $2+\frac{14}{55}$ &  \cellcolor{yellow!50} Theorem~\ref{main theorem}\\ 
   3 & \cellcolor{white} $3+\frac{3}{13}$ & \cellcolor{white} Wang~\cite{Wang}  &  12 & \cellcolor{white} $2+\frac{4}{17}$ &  \cellcolor{white} Guth~\cite{Guth}  \\
  4 &  \cellcolor{yellow!50} $2+\frac{1407}{1759}$ &  \cellcolor{yellow!50} Theorem~\ref{main theorem}\footnotemark[5]  &    13 &  \cellcolor{yellow!50} $2+\frac{100}{471}$ &  \cellcolor{yellow!50} Theorem~\ref{main theorem}  \\  
   5 &  \cellcolor{yellow!50} $2+\frac{12}{19}$ &  \cellcolor{yellow!50} Theorem~\ref{main theorem}  &  14 & \cellcolor{yellow!50}  $2+\frac{108}{541}$ &   \cellcolor{yellow!50} Theorem~\ref{main theorem}  \\ 
 6 &\cellcolor{white} $2+\frac{1}{2}$ & \cellcolor{white} Guth~\cite{Guth}  &  15 &  \cellcolor{yellow!50} $2+\frac{116}{637}$ &  \cellcolor{yellow!50} Theorem~\ref{main theorem}  \\ 
   7 &  \cellcolor{yellow!50} $2+ \frac{52}{123}$ &  \cellcolor{yellow!50} Theorem~\ref{main theorem}  &   16 &  \cellcolor{yellow!50} $2+\frac{62}{359}$ &    \cellcolor{yellow!50} Theorem~\ref{main theorem}   \\ 
  8 &\cellcolor{white} $2+\frac{4}{11}$ & \cellcolor{white} Guth~\cite{Guth}  &  17 & \cellcolor{yellow!50} $2+\frac{4}{25}$ & \cellcolor{yellow!50} Theorem~\ref{main theorem}  \\ 
  9 &  \cellcolor{yellow!50} $2+\frac{34}{107}$ & \cellcolor{yellow!50}  Theorem~\ref{main theorem}  &  18 &  \cellcolor{yellow!50} $2+\frac{7}{46}$ &    \cellcolor{yellow!50} Theorem~\ref{main theorem}   \\ 
  10 &\cellcolor{white} $2+\frac{2}{7}$ & \cellcolor{white} Guth~\cite{Guth}  &    19 &  \cellcolor{yellow!50} $2+\frac{1}{7}$ & \cellcolor{yellow!50} Theorem~\ref{main theorem} \\
 \hline
\end{tabular}
    \caption{The current state-of-the-art for the restriction problem in low dimensions. New results are \colorbox{yellow!50}{highlighted} and in most cases are deduced by combining Theorem~\ref{main theorem} with Proposition~\ref{9.1}.$^5$
    }
    \label{exponent table}
\end{figure}
\footnotetext[4]{This is deduced by combining Theorem~\ref{main theorem} with a more sophisticated version of Proposition~\ref{9.1}: see Remark~\ref{Demeter remark} below.}
\footnotetext[5]{These computations were carried out using the following Maple~\cite{Maple} code:
 \begin{verbatim}
n := [insert dimension]; 
p_broad := 2+8*(2n-1)/(n*(5n+2k-9)+k*(k-3) +4):   p_limit :=2+ 4/(2*n-k):
p_seq := [seq(max(eval(p_broad, k = i), eval(p_limit, k = i)), i = 2 .. n)]: 
new_exponent := min(p_seq);
\end{verbatim}}

Theorem~\ref{asymptotic theorem} follows as a direct consequence of Theorem~\ref{main theorem} and Proposition~\ref{9.1}. When applying Proposition~\ref{9.1}, the upper bound on $p$ is unimportant. However the lower bound, 
\begin{equation}\label{BG2011}
p\ge 2+\frac{4}{2n-k},
\end{equation}
is a limiting factor in the arguments, along with the condition \eqref{with} on the exponents in the $k$-broad inequality. In order to improve the state-of-the-art for the restriction conjecture one must choose an optimal $k$ so that neither of these two conditions is overly restrictive. For instance, if $n=5$ and $k=3$, then
$$
2+\frac{4}{10-3}\le p_5(3)=2+\frac{12}{19};
$$
 Theorem~\ref{main theorem} and Proposition~\ref{9.1} therefore imply that the restriction inequality holds for $p > 2+\frac{12}{19}$ when $n=5$. Other low dimensional cases can also be analysed directly and some examples can be found in Figure~\ref{exponent table} above.

In high dimensions, to derive the $\lambda$ coefficient featured in Theorem~\ref{asymptotic theorem}, we write $k=\nu n + O(1)$ for some $0<\nu<1$, so that, asymptotically,
\begin{equation}\label{boundbel0}
p_n(k)=2+\frac{16}{5+2\nu+\nu^2} n^{-1}+O(n^{-2}).
\end{equation}
On the other hand, with $k=\nu n + O(1)$, the condition \eqref{BG2011} can be rewritten as
\begin{equation}\label{boundbel}
p\ge 2+\frac{4}{2-\nu}n^{-1} +O(n^{-2}).
\end{equation}
The linear coefficients in \eqref{boundbel0} and \eqref{boundbel} are then equal when $\nu$ is the positive solution of the quadratic equation
$$
x^{2}+6x-3=0.
$$
Plugging this solution back into \eqref{boundbel} yields  \eqref{restriction} in the range
  $$
 p>2+\lambda n^{-1}+O(n^{-2})
 $$
with $\lambda = \frac{4}{2-\nu}=\frac{4}{5-2\sqrt{3}}$. 

It remains to prove Theorem~\ref{main theorem}, which will be the focus of the remainder of the article. 




\section{Notational conventions}\label{notation section}

From now on, we work with smooth, bounded functions $f$, $g$ or $h$ that map from the unit ball $B^{n-1}$ of $\mathbb{R}^{n-1}$ to the complex numbers, and we sometimes write $x \in \R^n$ as $x= (x', x_n)$ where $x'\in \mathbb{R}^{n-1}$. We call an $n$-dimensional ball $B_r$ of radius~$r$ an {\it $r$-ball} and an $(n-1)$-dimensional ball $\theta$ of radius $r^{-1/2}$ an {\it $r^{-1/2}$-cap}.  We call a cylinder of length~$r$ and radius $r^{1/2}$ an {\it $r$-tube}. The $\delta$-neighbourhood of a set $E$ will be denoted by $N_{\delta}E$.

The arguments will involve the {\it admissible parameters} $n$, $p$ and $\varepsilon$ and the constants in the estimates will be allowed to depend on these quantities. Given positive numbers $A, B \geq 0$ and a list of objects $L$, the notation $A \lesssim_L B$, $B \gtrsim_L A$ or $A = O_L(B)$ signifies that $A \leq C_L B$ where $C_L$ is a constant which depends only on the objects in the list and the admissible parameters. We write $A \sim_L B$ when both $A \lesssim_L B$ and $B \lesssim_L A$. We will also write $A \ll B$ or $B \gg A$ to denote that $A \leq C^{-1}B$ for some choice of $C \geq 1$ which can be taken to be as large as desired provided it is admissible.

 The cardinality of a finite set $A$ is denoted by $\#A$. A set $A'$ is said to be a \emph{refinement} of $A$ if $A' \subseteq A$ and $\#A' \gtrsim \#A$. In many cases it will be convenient to \emph{pass to a refinement} of a set $A$, by which we mean that the original set $A$ is replaced with some refinement.




\section{Overview}\label{Wolff axioms section}

\subsection{The polynomial Wolff axioms}  The key new geometric ingredient is the following theorem, which amounts to a confirmation of the Kakeya conjecture in a very specialised `algebraic' situation. It follows by combining \cite[Theorem 1.1]{KR} with Wongkew's lemma~\cite{Wongkew1993}, the latter of which bounds the measure of  a neighbourhood of a real algebraic variety over a ball.

\begin{theorem}[Polynomial Wolff axioms~\cite{KR}]\label{Katz--Rogers theorem} Let $\delta > 0$ and $c, r\ge1$. Let $\bZ \subseteq \R^n$ denote an $m$-dimensional algebraic variety
and let  $\mathbf{T}$ denote a collection of $r$-tubes contained in a ball of radius $2r$.  If the central axes of the tubes point in  $r^{-1/2}$-separated directions, then
 \begin{equation*}
\# \Big\{\, T \in \mathbf{T}\, :\, T \subseteq N_{cr^{1/2}}\bZ\, \Big\} \lesssim_{\Deg \bZ,\delta} c^{n-m}r^{\frac{m-1}{2} + \delta}.
\end{equation*}
\end{theorem}

This theorem was proven for $n=3$ by Guth~\cite{Guth2016} who later conjectured the general statement in~\cite{Guth} (see also~\cite{GZ}). The $n=4$ case was solved by Zahl before a proof in general dimensions was given in~\cite{KR}. Theorem~\ref{Katz--Rogers theorem} is referred to as the \emph{polynomial Wolff axioms} since the result can be interpreted as a verification that families of direction-separated tubes satisfy a natural polynomial generalisation of the classical (linear) \emph{Wolff axiom} introduced in~\cite{Wolff1995} (see~\cite{GZ} for further details).

\subsection{A brief description of the proof} The proof of Theorem~\ref{main theorem} extends an argument of Guth~\cite{Guth2016} in $\R^3$, by combining it with the later developments in higher dimensions from~\cite{Guth}. Both the articles~\cite{Guth2016} and~\cite{Guth} give comprehensive and highly readable introductory overviews of the core arguments; readers unfamiliar with these topics are encouraged to consult these sources for a detailed description of the main ideas. In high dimensions some complications arise which are not present in $\R^3$. For this reason, the proof given in Sections~\ref{k broad section} -~\ref{proof section} is structured somewhat differently from the proofs presented in~\cite{Guth2016, Guth}. These differences are highlighted and explained in the following subsection. 

The key ingredients of the proof of Theorem~\ref{main theorem} are as follows:

\subsubsection*{Wave packet decomposition} The first step of the argument is to employ the standard technique of decomposing the input function $f$ as a sum of localised pieces called \emph{wave packets}. In particular, fixing a large scale $R \gg 1$, one decomposes the domain $B^{n-1}$ as a union of $R^{-1/2}$-balls denoted by $\theta$ and referred to as $R^{-1/2}$-\emph{caps}. The function is then written as a sum of pieces $f = \sum_{(\theta,v)} f_{\theta,v}$ where each $f_{\theta,v}$ has support in the cap $\theta$ and the inverse Fourier transform of $f_{\theta,v}$ is concentrated in an $R^{1/2}$-ball centred at $v \in \R^{n-1}$. There are two key properties of this decomposition:
\begin{itemize}
    \item \textbf{Orthogonality:} Given any collection of wave packets $\mathbb{W}$ one has
    \begin{equation}\label{introduction: orthogonality}
    \Big\|\sum_{(\theta,v) \in \mathbb{W}} f_{\theta,v}\Big\|_2^2 \sim \sum_{(\theta,v) \in \mathbb{W}} \|f_{\theta,v}\|_2^2.
    \end{equation}
    \item \textbf{Spatial concentration:} On the ball $B(0,R)$, the function $Ef_{\theta,v}$ is essentially supported on an $R$-tube $T_{\theta,v}$ with direction governed by $\theta$ and position governed by $v$.
\end{itemize}
More precisely, the direction of $T_{\theta,v}$ is given by the normal direction to the paraboloid at the point $(\xi_{\theta}, |\xi_{\theta}|^2)$, where $\xi_{\theta}$ is the centre of $\theta$. Thus, $Ef = \sum_{(\theta,v)} Ef_{\theta,v}$ can be thought of as a sum of oscillating, normalised characteristic functions of tubes, which point in many different directions. Understanding the incidence geometry of these tubes is a key consideration in the restriction problem. 

\subsubsection*{Polynomial partitioning} A useful tool for studying the incidence-geometric problems arising from the wave packet decomposition is the \emph{polynomial partitioning method}. This method was introduced by Guth and Katz~\cite{GK2015} in their resolution of the Erd\H{o}s distance conjecture and was first applied to the restriction problem by Guth in~\cite{Guth} (the latter work also incorporated a refinement to the original partitioning method of \cite{GK2015} due to Solymosi and Tao \cite{ST}). The basic idea is a divide-and-conquer-style argument: one begins by finding a polynomial $P$ of low degree which partitions the mass of $\|Ef\|_{\BL{k}^p(B_{\>\!\!R})}$ into equal size pieces. More precisely, let $Z(P) := \{ z \in \R^n : P(z) = 0\}$ denote the zero set of $P$ and $\cell(P)$  the set of connected components of $\R^n \setminus Z(P)$. These connected components are referred to as \emph{cells}. The polynomial $P$ can then be chosen so that the $\|Ef\|_{\BL{k}^p(O')}$ are (essentially) equal as $O'$ varies over $\cell(P)$. Due to geometric (and underlying uncertainty principle) considerations, one actually works with a `blurred out' version of the variety $Z(P)$ given by the $R^{1/2}$-neighbourhood $W := N_{R^{1/2}}Z(P)$ and referred to as the \emph{wall}. Defining the collection of slightly shrunken cells $\O := \{O'\setminus W : O' \in \cell(P)\}$, the following simple, yet vital, geometric property holds:
\begin{equation}\label{shrunken property}
\parbox{\dimexpr\linewidth-4em}{%
    \strut
    Whenever $T_{\theta,v}$ enters a shrunken cell $O=O' \setminus W$, the core line necessarily enters the original cell $O'$.
    \strut
  }
\end{equation}

Unlike the original cells, collectively the $O \in \O$  may only account for a small proportion of the mass of $\|Ef\|_{\BL{k}^p(B_{\>\!\!R})}$. There are two cases to consider:\\

{\textbf{Cellular case:}} The mass of $\|Ef\|_{\BL{k}^p(B_{\>\!\!R})}$ concentrates on the $O \in \O$ in the sense that
    \begin{equation*}
        \|Ef\|_{\BL{k}^p(B_{\>\!\!R})}^p \lesssim \sum_{O \in \O}\|Ef\|_{\BL{k}^p(O)}^p.
    \end{equation*}
    In this situation, one defines $f_O := \sum_{(\theta,v) \in \T_O} f_{\theta,v}$ where $\T_O$ denotes the collection of wave packets $(\theta,v)$ for which $T_{\theta,v} \cap O \neq \emptyset$. By the spatial concentration property of the wave packets 
       \begin{equation*}
        \|Ef\|_{\BL{k}^p(B_{\>\!\!R})}^p \lesssim \sum_{O \in \O}\|Ef_O\|_{\BL{k}^p(O)}^p.
    \end{equation*}
    The key observation here is that each $(\theta,v)$ can only belong to a small number (in particular, $\deg P + 1$) of the sets $\T_O$. This is due to \eqref{shrunken property} and the fact that, by the fundamental theorem of algebra (or B\'ezout's theorem), the core line of a tube $T_{\theta,v}$ can only enter $\deg P + 1$ cells from $\cell(P)$. This observation can be interpreted as saying the sets $\T_O$ are `almost disjoint' which implies, via \eqref{introduction: orthogonality}, that the $f_{O}$ are `almost orthogonal'. Consequently, one can pass to the cells and analyse them individually. This forms the basis of a recursive procedure.\setcounter{footnote}{5}\footnote{This part of the argument is fairly delicate and the almost orthogonal property needs to be precisely quantified in terms of $\deg P$. For the purposes of this sketch, the full details are omitted.}\\
    
{\textbf{Algebraic case:}} The mass of $\|Ef\|_{\BL{k}^p(B_{\>\!\!R})}$ concentrates on the wall in the sense that
    \begin{equation*}
        \|Ef\|_{\BL{k}^p(B_{\>\!\!R})}^p \lesssim \|Ef\|_{\BL{k}^p(W)}^p.
    \end{equation*} 
    Here it suffices to consider only those wave packets $(\theta,v)$ for which $T_{\theta,v} \cap W \neq \emptyset$. A tube $T_{\theta,v}$ can intersect $W$ in one of two ways: either \emph{tangentially} or \emph{transversally}. The analysis is further divided into two subcases depending on whether the main contribution to $\|Ef\|_{\BL{k}^p(B_{\>\!\!R})}^p$ arises from tangential or transverse wave packets. 

In the transversal subcase the tubes can be thought of as passing directly through the wall. This situation can be treated in a manner similar to the cellular case, this time using a continuum version of B\'ezout's theorem to show that any given tube can intersect~$W$ transversally in relatively few places. 

It remains to study the tangential subcase. Here the $T_{\theta,v}$ can be thought of as being contained in $W$ and making a small angle with tangent spaces at nearby points of the variety $Z(P)$. 

\subsubsection*{Dimensional reduction} The polynomial partitioning argument sketched above can be interpreted as a dimensional reduction. If either the cellular or the transverse algebraic case holds, then one can obtain acceptable estimates for $\|Ef\|_{\BL{k}^p(B_{\>\!\!R})}$. Thus, it suffices to consider the situation where the wave packets of $f$ are all tangent to some variety of dimension $n-1$. By iterating this dimensional reduction procedure,\footnote{A number of serious complications arise in implementing this iteration scheme and, in particular, in dealing with the transverse algebraic case. This part of the argument requires what are known as \emph{transverse equidistribution estimates}: these inequalities are briefly mentioned below, see the introductory discussion in~\cite{Guth} for further details.} it becomes important to understand what can be said when the wave packets of $f$ are all tangent to some variety of dimension $m$ for any value $0 \leq m \leq n-1$. 

\subsubsection*{Key estimates in the tangential case} The reduction to tangential situations, as outlined above, can be exploited in a number of ways: 
\begin{itemize}
     \item \textbf{Vanishing property of the $k$-broad norms:} The definition of the $k$-broad norms implies that if the wave packets of $f$ are all tangential to a variety of dimension $m < k$, then $\|Ef\|_{\BL{k}^p(B_{\>\!\!R})}$ essentially vanishes. Thus, one need only consider tangency properties with respect to varieties of dimension at least $k$. Using this fact alone, one may prove $k$-broad estimates in the range $p > \frac{2k}{k-1}$ corresponding to the Bennett--Carbery--Tao multilinear restriction theorem~\cite{BCT2006}.\footnote{Indeed, this follows by applying the argument of \cite{Guth} but ignoring gains coming from \emph{transverse equidistribution} (see the following bulletpoint).}   
     \item \textbf{Transverse equidistribution estimates:} These inequalities were introduced by Guth~\cite{Guth} and heavily exploit the curvature properties of the paraboloid, allowing for $k$-broad estimates beyond the $p > \frac{2k}{k-1}$ range. The basic idea behind the transverse equidistribution estimates is recalled below in Section~\ref{transverse equidistribution section}.
     \item \textbf{The polynomial Wolff axioms:} Given a family $\mathbf{T}$ of $R$-tubes lying in the $R^{1/2}$-neighbourhood of a variety, the polynomial Wolff axioms limit the number of different directions in which the $T \in \mathbf{T}$ can lie. Thus, if the wave packets of $f$ are all tangent to some low dimensional variety, then $f$ must be supported on very few caps $\theta$ (since the caps $\theta$ correspond to the directions of the tubes $T_{\theta,v}$). The small support of $f$ can be exploited via H\"older's inequality to obtain favourable $k$-broad estimates.  
\end{itemize}

\subsection{Induction versus recursion}  When applying the polynomial Wolff axioms to the restriction problem in high dimensions, a number of complications arise which are not present in the $\R^3$ case treated in~\cite{Guth2016}. The root of these complications lies in the fact that, in contrast with $\R^3$ where one only need consider tangency conditions with respect to 2-surfaces, in higher dimensions one must consider tangency conditions with respect to surfaces of many different dimensions.

 The core argument sketched in the previous subsection can be implemented as either an induction or a recursion argument. 
The original articles~\cite{Guth2016} and~\cite{Guth} make heavy use of mathematical induction (inducting on a number of quantities including the choice of scale $R$); this has the advantage of yielding a clean and concise argument, but unfortunately useful structural properties are potentially hidden. From the perspective of a recursive algorithm one may gain a more detailed understanding of the argument at each stage of the iterative process; this is the approach taken in the present article. 

There is certainly a precedent for the recursive approach: for instance, in the fourth section of~\cite{BG2011},  Bourgain and Guth reformulate their key induction-on-scale argument as a recursive procedure to allow for the use of additional information coming from $X$-ray transform estimates (see also~\cite{Luca, Temur2014} for an elaboration of this argument). Similarly, in a recent article of Wang~\cite{Wang}, the induction-on-scale procedure of~\cite{Guth2016} was rewritten as a recursion; this permitted a more detailed analysis of the underlying geometry of the extension operator and led to the current best known bounds for the restriction conjecture in $\R^3$.
 
 When written in the form of a recursive algorithm, the polynomial partitioning argument of~\cite{Guth} can be interpreted as a structural statement. Following the discussion in the previous subsection, one may think of the input function $f$ as being broken into many different pieces where, roughly, each piece is made up of wave packets tangential to a low dimensional variety at some scale (there may be other pieces which do not have this property, but they arise from the cellular or transverse algebraic cases and satisfy favourable estimates). Thus, the structural statement allows one to focus on estimating the `tangential' pieces $\{f_{\tang}\}$ of the function. This `tangential reduction' is then exploited via the key estimates described in the previous subsection. 
 
 In high dimensional cases, however, the $f_{\tang}$ tend to enjoy further structural properties which one could potentially utilise in order to improve the range of estimates guaranteed by Theorem~\ref{main theorem}. Indeed, typically a given $f_{\tang}$ is not only tangent to a single variety $\bZ$ at a single scale $r$, but it satisfies certain tangency conditions with respect to a whole sequence of scales $r_m < \dots < r_n$ and a corresponding sequence of varieties $\bZ_m, \dots, \bZ_n$ with $\dim \bZ_{i} = i$ for $m \leq i \leq n$. These `nested' conditions could  potentially lead to further gains for the restriction exponent. To carry out such a programme, however, one would have to effectively analyse properties of the $f_{\tang}$ across many distinct scales $r_m < \dots < r_n$; this situation  lends itself more naturally to a recursive algorithm, rather than an inductive argument.




\section{Broad norms}\label{k broad section}

Here we recall the definition and basic properties of the $k$-broad norms from~\cite{Guth2016} and \cite{Guth}. For a detailed motivation of this definition and its relation to the multilinear restriction theory of~\cite{BCT2006} the reader is referred to~\cite{Guth} and \cite[Section 6.2]{GHI}. 

Fix some large $R \gg 1$ and a ball $B_{\>\!\!R} \subset \R^n$. Decompose the unit ball $B^{n-1}$ into finitely-overlapping balls $\tau$ of radius $K^{-1}$, where $K$ is a large constant satisfying $1 \ll K \ll R$. These $(n-1)$-dimensional balls are referred to as \emph{$K^{-1}$-caps}. Given a function $f$, supported on $B^{n-1}$, we write $f = \sum_{\tau} f_{\tau}$ where $f_{\tau} := f\psi_{\tau}$ for $(\psi_{\tau})_{\tau}$ a partition of unity subordinate to the caps $\tau$. Let $G \colon B^{n-1} \to S^{n-1}$ denote the Gauss map associated to the paraboloid, given explicitly by
\begin{equation}\label{Gauss map}
G(\xi) := \frac{1}{(1 + 4|\xi|^2)^{1/2}}\big(-2\xi,1\big).
\end{equation}
Given a pair of non-zero vectors $v, v' \in \R^n$, let $\angle(v,v')$ denote the (unsigned) angle between them. If $V \subseteq \R^n$ is a linear subspace, then let $\angle(G(\tau), V)$ denote the minimum of $\angle(v,v')$  over all pairs of non-zero vectors $v \in V$ and $v' \in G(\tau)$. 

The spatial ball $B_{\>\!\!R}$ is also decomposed into relatively small balls $B_{\>\!\!K^2}$ of radius~$K^2$. In particular, fix $\mathcal{B}_{\>\!\!K^2}$ a collection of finitely-overlapping $K^2$-balls which are centred in and cover $B_{\>\!\!R}$. Then, for $B_{\>\!\!K^2} \in \mathcal{B}_{\>\!\!K^2}$, define
\begin{equation}\label{mu definition}
\mu_{Ef}(B_{\>\!\!K^2}) := \min_{V_1,\dots,V_A \in \mathrm{Gr}(k-1, n)} \Bigg(\max_{\tau : \angle(G(\tau), V_a) > K^{-1}  \textrm{ for } 1 \leq a \leq A} \|Ef_{\tau}\|_{L^p(B_{\>\!\!K^2})}^p \Bigg);
\end{equation}
here $\mathrm{Gr}(k-1, n)$ is the Grassmannian manifold of all $(k-1)$-dimensional subspaces in $\R^n$. For $U\subseteq \R^n$ the \emph{$k$-broad norm over $U$} can then be defined as
\begin{equation}\label{k-broad norm definition}
\|Ef\|_{\mathrm{BL}^p_{k,A}(U)} := \Bigg(\sum_{B_{\>\!\!K^2} \in \mathcal{B}_{\>\!\!K^2}} \frac{|B_{\>\!\!K^2} \cap U|}{|B_{\>\!\!K^2}|}\mu_{Ef}(B_{\>\!\!K^2}) \Bigg)^{1/p}.
\end{equation}
With this definition, the inequality \eqref{broad} from Section~\ref{reduction section} is understood to hold for ${\|Ef\|_{\mathrm{BL}^p_{k}(U)} := \|Ef\|_{\mathrm{BL}^p_{k,A}(U)}}$ for some choice of $A \sim 1$.

Before continuing it is perhaps useful to clarify the relative sizes of the parameters. Given any $p$ and $\varepsilon$, when proving a broad norm estimate \eqref{broad} it is always assumed that $K$ and $A$ are large but admissible (that is, they depend only on $n$, $p$ and $\varepsilon$). The parameter $K$ must be chosen large in order for Proposition~\ref{9.1} to hold (see \cite{Guth} and \cite{BG2011}) whilst the parameter $A$ must be chosen large in order to facilitate multiple applications of Lemma \ref{triangle inequality lemma} and Lemma \ref{logarithmic convexity inequality lemma}, as described below. Nevertheless, it is always possible to make admissible choices of $K$ and $A$. The parameter $R$, on the other hand, is an arbitrarily large number which will be, in general, non-admissible.  

As mentioned in Section~\ref{reduction section}, the key advantage of working with $k$-broad norms rather than the classical $L^p$-norms is that, roughly, they vanish whenever the mass of $Ef$ is concentrated around a set of dimension less than $k$. This property is fundamental to the proof of Theorem~\ref{main theorem}, but to make it precise requires a number of preliminary definitions and therefore the details are postponed until Lemma~\ref{key k-broad lemma} below.

\subsection{Basic properties} It is easy to see that $\|Ef\|_{\mathrm{BL}^p_{k,A}(U)}$ is not a norm in any traditional sense. Nevertheless, as noted in~\cite{Guth}, it does satisfy weak variants of certain key properties of $L^p$-norms.

\begin{lemma}[Finite subadditivity]\label{subadditivity lemma} Let $U_1, U_2 \subseteq \R^n$, $1 \leq p < \infty$, and $A\in\mathbb{N}$. Then
\begin{equation*}
\|Ef\|_{\mathrm{BL}^p_{k,A}(U_1 \cup U_2)}^p \leq \|Ef\|_{\mathrm{BL}^{p}_{k,A}(U_1)}^p + \|E f\|_{\mathrm{BL}^{p}_{k,A}(U_2)}^p
\end{equation*} 
holds for all integrable $f \colon B^{n-1} \to \C$.
\end{lemma}

This is an immediate consequence of the definition of the $k$-broad norms. A slightly less trivial observation is that $\|Ef\|_{\mathrm{BL}^p_{k,A}(U)}$ also satisfies weak versions of the triangle and logarithmic convexity inequalities. 

\begin{lemma}[Triangle inequality]\label{triangle inequality lemma} Let $U \subseteq \R^n$, $1 \leq p < \infty$ and $A \in \mathbb{N}$. Then
\begin{equation*}
\|E(f_1 + f_2)\|_{\mathrm{BL}^p_{k,2A}(U)} \lesssim \|Ef_1\|_{\mathrm{BL}^p_{k,A}(U)} + \|E f_2\|_{\mathrm{BL}^p_{k,A}(U)}
\end{equation*} 
holds for all integrable $f_1, f_2 \colon B^{n-1} \to \C$.
\end{lemma}

\begin{lemma}[Logarithmic convexity]\label{logarithmic convexity inequality lemma} Let $U \subseteq \R^n$, $1 \leq p, p_0, p_1 < \infty$ and $A \in \mathbb{N}$. Suppose that  $0 \leq \alpha \leq 1$ satisfies
\begin{equation*}
\frac{1}{p} = \frac{1-\alpha}{p_0} + \frac{\alpha}{p_1}.
\end{equation*}
Then
\begin{equation*}
\|Ef\|_{\mathrm{BL}^p_{k,2A}(U)} \lesssim \|Ef\|_{\mathrm{BL}^{p_0}_{k,A}(U)}^{1-\alpha} \|E f\|_{\mathrm{BL}^{p_1}_{k,A}(U)}^{\alpha}
\end{equation*} 
holds for all integrable $f \colon B^{n-1} \to \C$.
\end{lemma}

The proofs of these estimates are entirely elementary and can be found in~\cite{Guth}. The parameter $A$ appears in the definition of the $k$-broad norm to allow for these weak triangle and logarithmic convexity inequalities.

\subsection{Linear versus $k$-broad estimates} Any $k$-broad estimate is weaker than the corresponding linear estimate. For instance, assuming that the local extension estimate \eqref{local extension} holds, given $\varepsilon > 0$ and $1 \le r \leq R$, it follows that
\begin{align*}
    \|Ef\|_{\BL{k,A}^p(B_{r})} &\leq \Bigg( \sum_{\tau : K^{-1}\mathrm{-cap}} \sum_{\substack{B_{\>\!\!K^2} \in \cB_{\>\!\!K^2}\\ B_{\>\!\!K^2} \cap B_{r} \neq \emptyset}} \| Ef_{\tau} \|_{L^p(B_{\>\!\!K^2})}^p\Bigg)^{1/p} \lesssim K^{O(1)} r^{\varepsilon} \|f \|_{\infty};
\end{align*}
since $K$ is just a constant (in particular, it is chosen independently of $R \gg 1$), this implies \eqref{broad}. 

From the preceding observation, $L^p$ estimates for the extension operator translate into $k$-broad inequalities. In view of this, it is useful to briefly recall some standard $L^2$ estimates for the extension operator. Plancherel's theorem implies the familiar conservation of energy identity 
\begin{equation}\label{energy identity}
\int_{\R^{n-1}} |Ef(x',x_n)|^2 \,\ud x' = (2\pi)^{n-1} \|f\|_{2}^2
\end{equation}
and one may integrate in the $x_n$ variable and take square roots to conclude that
\begin{equation*}
    \|Ef\|_{L^2(B_{r})} \lesssim r^{1/2} \|f\|_{2}
\end{equation*}
for any $r$-ball $B_{r}$. Arguing as above, one immediately arrives at the $k$-broad variant
\begin{equation}\label{broad energy identity}
    \|Ef\|_{\BL{k,A}^2(B_{r})} \lesssim r^{1/2} \|f\|_{2},
\end{equation}
valid for all $r \geq 1$.







\section{Polynomial partitioning}

\subsection{Basic partitioning}\label{basic partitioning section} In this section the relevant algebraic and topological ingredients for the proof of Theorem~\ref{main theorem} are reviewed. In particular, the key polynomial partitioning theorem is stated, which is adapted from previous works of Guth~\cite{Guth2016, Guth} on the restriction conjecture. 

\begin{definition} Given any collection of polynomials $P_1, \dots, P_{n-m} \colon \R^n \to \R$ the common zero set
\begin{equation*}
    Z(P_1, \dots, P_{n-m}) := \Big\{\,x \in \R^n \,:\, P_1(x) = \cdots = P_{n-m}(x) = 0\,\Big\}
\end{equation*}
will be referred to as a \emph{variety}.\footnote{The ideal generated by the $P_j$ is \emph{not} required to be irreducible.} Given a variety $\bZ = Z(P_1, \dots, P_{n-m})$, define its \emph{(maximum) degree} to be the number 
\begin{equation*}
\overline{\deg}\,\bZ := \max \{\deg P_1, \dots, \deg P_{n-m} \}.
\end{equation*}
\end{definition}
It will often be convenient to work with varieties which satisfy the additional property that
\begin{equation}\label{non singular variety}
\bigwedge_{j=1}^{n-m}\nabla P_j(z) \neq 0 \qquad \textrm{for all $z \in \bZ = Z(P_1, \dots, P_{n-m})$.}
\end{equation}
In this case the zero set forms a smooth $m$-dimensional submanifold of $\R^n$ with a (classical) tangent space $T_z\bZ$ at every point $z \in \bZ$. A variety $\bZ$ which satisfies \eqref{non singular variety} is said to be an \emph{$m$-dimensional transverse complete intersection}.

Of particular interest is the case of hypersurfaces, where $m = n-1$. Given a polynomial $P \colon \R^n \to \R$ consider the collection  $\cell(P)$  
 of connected components of $\R^n \setminus Z(P)$. As in Section~\ref{Wolff axioms section}, each $O \in \cell(P)$ is referred to as a \emph{cell} cut out by the variety $Z(P)$ and the cells are thought of as partitioning the ambient euclidean space into a finite collection of disjoint regions. 
 
\begin{theorem}[Guth~\cite{Guth2016}]\label{simple partitioning theorem} Fix $d \in \N$ and suppose $F \in L^1(\R^n)$ is non-negative. Then there exists a polynomial $P \colon \R^n \to \R$ of degree at most $d$ such that:
\begin{enumerate}[i)]
\item $\#\cell(P) \sim d^n$;
\item The integrals $\int_{O} F$ for $O \in \cell(P)$ are all equal.
\end{enumerate}
\end{theorem}

This theorem is based on an earlier discrete partitioning result which played a central role in the resolution of the Erd\H{o}s distance conjecture~\cite{GK2015}. The proof is essentially topological, involving the polynomial ham sandwich theorem of Stone--Tukey~\cite{Stone1942}, which is itself a consequence of the Borsuk--Ulam theorem.

Under the hypotheses of Theorem~\ref{simple partitioning theorem}, it trivially follows that
\begin{equation}\label{partitioning 1}
    \int_{\R^n} F = \sum_{O \in \cell(P)} \int_{O} F = \#  \cell(P)  \int_{O_*} F \qquad \textrm{for any  $O_* \in \cell(P)$}.
\end{equation}
In view of the forthcoming applications of the polynomial partitioning theorem, precise equality is not required in \eqref{partitioning 1}, but merely comparability. By relaxing the inequality, one may, for instance, ensure that $Z(P)$ is given by a finite union of transverse complete intersections: see Theorem 5.5 of~\cite{Guth}. Furthermore, often one may freely pass to some refinement of the collection of cells which satisfy additional properties. This observation naturally lends itself to pigeonholing arguments, and two examples along these lines are discussed presently.

\subsubsection*{Passing to shrunken cells} It will be necessary to work with a `blurred out' version of the variety $Z(P)$ given by the $r^{1/2+\delta_{\;\!\!\circ}}$-neighbourhood $N_{r^{1/2 + \delta_{\;\!\!\circ}}}Z(P)$ for different choices of $r>0$ and small parameter $\delta_{\;\!\!\circ}>0$. The set $N_{r^{1/2+\delta_{\;\!\!\circ}}}Z(P)$ is referred to as the \emph{wall}. A simple pigeonholing argument shows that at least one of two cases hold:\\

\paragraph{\underline{Cellular case}} One may pass to a refinement of $\cell(P)$ such that if $\O$ denotes the collection of \emph{$r^{1/2 + \delta_{\;\!\!\circ}}$-shrunken cells} 
\begin{equation}\label{shrunken cells}
    \O := \Big\{\, O' \setminus N_{r^{1/2+\delta_{\;\!\!\circ}}}Z(P)\, :\, O' \in \cell(P)\, \Big\},
\end{equation}
then the mass of $F$ is essentially evenly distributed across these shrunken cells:
    \begin{equation*}
        \int_{O} F \sim d^{-n} \int_{\R^n} F \qquad \textrm{for all $O \in \O$}.
    \end{equation*}
    
\paragraph{\underline{Algebraic case}} The contribution to the integral from the wall dominates:
    \begin{equation*}
        \int_{\R^n} F \lesssim \int_{N_{\!r^{1/2+\delta_{\;\!\!\circ}}}Z(P)} F.
    \end{equation*}

\subsubsection*{Controlling the size of the cells} A simple but useful observation, appearing in~\cite{Wang}, is that one may also apply a pigeonholing argument to yield some natural control on the size of the cells. Here the analysis is localised to a fixed $r$-ball $B_{r}$ and, in particular, it is assumed that $\mathrm{supp}\,F \subset B_{r}$. In this situation one may, after passing to various refinements and relaxing the equalities in \eqref{partitioning 1}, assume that each $O \in \cell(P)$ has diameter at most $r/d$. In the present article, this reduction is made more for convenience rather than out of necessity and only a bound of $r/2$ is needed on the diameter of the cells; the precise details of this argument are therefore omitted (see~\cite{Wang} for further information).

\subsection{Partitioning over lower dimensional sets}  Theorem~\ref{simple partitioning theorem} alone is insufficient for the purposes of this article and a more involved partitioning result, which is implicit in~\cite{Guth}, will be used.

\begin{theorem}[Guth~\cite{Guth}]\label{partitioning theorem} Fix $r\gg 1$, $d \in \N$ and suppose $F \in L^1(\R^n)$ is non-negative and supported on $B_{r} \cap N_{r^{1/2 +\delta_{\;\!\!\circ}}}\bZ$ for some $0 < \delta_{\;\!\!\circ} \ll 1$, where  $\bZ$ is an $m$-dimensional transverse complete intersection of degree at most $d$. At least one of the following cases holds:\\

\paragraph{\underline{Cellular case}} There exists a polynomial $P \colon \R^n \to \R$ of degree $O(d)$ with the following properties:
\begin{enumerate}[i)]
    \item $\#\cell(P) \sim d^m$ and each $O \in \cell(P)$ has diameter at most $r/2$.
    \item One may pass to a refinement of $\cell(P)$ such that if $\O$ is defined as in \eqref{shrunken cells}, then 
    \begin{equation*}
        \int_{O} F \sim d^{-m}\int_{\R^n} F \qquad \textrm{for all $O \in \O$.}
    \end{equation*}
\end{enumerate}  
\paragraph{\underline{Algebraic case}} There exists an $(m-1)$-dimensional  transverse complete intersection $\bY$ of  degree at most $O(d)$ such that
    \begin{equation*}
        \int_{B_{r} \cap N_{\!r^{1/2+\delta_{\;\!\!\circ}}}\bZ} F \lesssim \int_{B_{r} \cap N_{\!r^{1/2 + \delta_{\;\!\!\circ}}}\bY} F.
    \end{equation*}
\end{theorem}

The choice of scales $r$ and $r^{1/2 + \delta_{\;\!\!\circ}}$ is not particularly special in the sense that the theorem holds true in greater generality: the result is presented in this specific case only in anticipation of later applications.

The statement of this theorem does not explicitly appear in~\cite{Guth}, but it can be easily deduced from the argument described in Section 8.1 of that article together with the simple pigeonholing arguments discussed earlier in this subsection. The key difference between Theorem~\ref{partitioning theorem} and Theorem~\ref{simple partitioning theorem} is that in the latter one has the additional hypothesis that $F$ is supported in a $r^{1/2 + \delta_{\;\!\!\circ}}$-neighbourhood of the \emph{$m$-dimensional} variety $\bZ$. This allows one to construct a partitioning polynomial which cuts out only $O(d^m)$ cells rather than the $O(d^n)$ guaranteed by Theorem~\ref{simple partitioning theorem}. 

Theorem~\ref{partitioning theorem} is then applied to the relevant broad norm by taking 
\begin{equation*}
 F = \sum_{B_{\>\!\!K^2} \in \cB_{\>\!\!K^2}}  \mu_{Ef}(B_{\>\!\!K^2})\frac{1}{|B_{\>\!\!K^2}|} \mathbf{1}_{B_{\>\!\!K^2} \cap B_{r} \cap N_{\!r^{1/2+\delta_{\;\!\!\circ}}}\bZ}
\end{equation*}
for some $0 < \delta_{\;\!\!\circ} \ll 1$. 
 \begin{itemize} 
 \item If the cellular case holds, then it follows that
 \begin{equation*}
     \|Ef\|_{\BL{k,A}^p(B_{r} \cap N_{r^{1/2+\delta_{\;\!\!\circ}}}\bZ)}^p \lesssim d^{m}\|Ef\|^p_{\BL{k,A}^p(O)} \qquad \textrm{for all $O \in \O$}
 \end{equation*}
 where $\O$ is the collection of cells produced by the theorem.
 \item If the algebraic case holds, then it follows that
 \begin{equation*}
     \|Ef\|_{\BL{k,A}^p(B_{r} \cap N_{r^{1/2+\delta_{\;\!\!\circ}}}\bZ)}^p \lesssim 
     \|Ef\|_{\BL{k,A}^p(B_{r} \cap N_{\!r^{1/2+\delta_{\;\!\!\circ}}}\bY)}^p
\end{equation*}
where $\bY$ is the variety produced by the theorem.
\end{itemize}




\section{Wave packet decompositions}\label{wave packet decomposition section}




\subsection{Definition and basic properties} Let $r \gg 1$ and cover the domain $B^{n-1}$ by a family $\Theta_r$ of finitely-overlapping balls of radius $r^{-1/2}$. As noted in Section~\ref{Wolff axioms section}, these $(n-1)$-dimensional balls are referred to as \emph{$r^{-1/2}$-caps} and $\xi_{\theta}$ is used to denote the centre of $\theta$. Fix $(\psi_{\theta})_{\theta \in \Theta_r}$ a smooth partition of unity for $B^{n-1}$, subordinate to the cover $\Theta_r$, such that each function $\xi \mapsto \psi_{\theta}(\xi_{\theta} + r^{-1/2}\xi)$ is supported in $[-\pi,\pi]^{n-1}$ and
\begin{equation*}
    \|\partial_x^{\alpha} \psi_{\theta}\|_{L^{\infty}(\R^{n-1})} \lesssim_{\alpha} r^{|\alpha|/2} \qquad \textrm{for all $\alpha \in \N_0^{n-1}$.}
\end{equation*}

Given our smooth, bounded input function $f \colon B^{n-1} \to \C$, by performing a Fourier series decomposition, we have 
\begin{equation*}
    f\cdot \psi_{\theta}(\xi) = \Big(\frac{r^{1/2}}{2\pi}\Big)^{n-1}\sum_{v \in r^{1/2}\Z^{n-1}} e^{i \langle v, \xi \rangle } (f\cdot \psi_{\theta})^\wedge(v).
\end{equation*}
 Writing $\T[r] := \Theta_{r} \times r^{1/2}\Z^{n-1}$, this yields
\begin{equation}\label{wave packet decomposition}
f = \sum_{(\theta,v) \in \T[r]} f_{\theta,v}
\end{equation}
where 
\begin{equation*}
f_{\theta,v}(\xi) := \Big(\frac{r^{1/2}}{2\pi}\Big)^{n-1}e^{i \langle v, \xi \rangle } (f\cdot \psi_{\theta})^\wedge(v)\tilde{\psi}_{\theta}(\xi)
\end{equation*}
for $\tilde{\psi}_{\theta}$ a bump function which is also adapted to $\theta \in \Theta_{r}$ but which is equal to 1 on the support of $\psi_{\theta}$.  
The sum \eqref{wave packet decomposition} is referred to as the \emph{wave packet decomposition of~$f$ at scale $r$}. The functions $f_{\theta,v}$ and the pairs $(\theta,v) \in \T[r]$ will both be referred to as \emph{(scale $r$) wave packets}. 

The key properties of this decomposition are as follows:\\




\paragraph{\textbf{Orthogonality between the wave packets}} Recall that the $\psi_{\theta}$ have almost disjoint supports. Combining this observation with the Plancherel identity for Fourier series, one concludes that
\begin{equation*}
    \Big\|\sum_{(\theta,v) \in \mathbb{W}} f_{\theta,v}\Big\|_{2}^2 \sim \sum_{(\theta,v) \in \mathbb{W}} \|f_{\theta,v}\|_{2}^2
\end{equation*}
for any collection of wave packets $\mathbb{W} \subseteq \T[r]$. 

It is worth noting that there is a local version of this orthogonality relation. In particular, for $1 \leq \rho \leq r$ and  a $\rho^{-1/2}$-cap $\theta_*$, one may readily verify that
\begin{equation*}
    \Big\|\sum_{(\theta,v) \in \mathbb{W}} f_{\theta,v}\Big\|_{L^2(\theta_*)}^2 \lesssim \sum_{(\theta,v) \in \mathbb{W}} \|f_{\theta,v}\|_{L^2(3\theta_*)}^2
\end{equation*}
where the right-hand norm is over the cap $3\theta_*$ concentric to $\theta_*$ but with thrice the radius. A reverse form of this inequality also holds (with $\theta_*$ on the left and $3\theta_*$ on the right-hand side), and together they imply the more symmetric estimate
\begin{equation*}
    \max_{\theta_* : \rho^{-1/2}- \mathrm{cap}}\Big\|\sum_{(\theta,v) \in \mathbb{W}} f_{\theta,v}\Big\|_{L^2(\theta_*)}^2 \sim \max_{\theta_* : \rho^{-1/2}- \mathrm{cap}} \sum_{(\theta,v) \in \mathbb{W}} \|f_{\theta,v}\|_{L^2(\theta_*)}^2,
\end{equation*}
where the maximum is over all $\rho^{-1/2}$-caps. \\



\paragraph{\textbf{Spatial concentration}} Given any wave packet $(\theta,v) \in \T[r]$, on the ball $B(0,r)$ the function $Ef_{\theta,v}$ is essentially supported on the tube 
\begin{equation*}
\Big\{\, x \in B(0,r)\, :\, |x'+2x_n\xi_{\theta} + v| \leq r^{1/2}\, \Big\}
\end{equation*}
in the sense that $|Ef_{\theta,v}(x)|$ decays rapidly as $x \in B(0,r)$ moves away from this set. More precisely, a simple stationary phase analysis shows that
\begin{equation*}
    |Ef_{\theta,v}(x)| \lesssim_N r^{-\frac{n-1}{4}} (1 + r^{-1/2}|x'+2x_n\xi_{\theta} + v|)^{-N} \|f_{\theta,v}\|_{2}
\end{equation*}
for all $N \in \N$ and $x\in\R^n$ with $|x_n| < r$; see, for example, \cite[Lemma 4.1]{Tao2003}. In particular, given $0 < \delta \ll 1$, the function $|Ef_{\theta,v}|$ is very small away from the slightly fattened tube
\begin{equation*}
  T_{\theta,v} :=  \Big\{\, x \in B(0,r)\, :\, |x'+2x_n\xi_{\theta} + v| \leq r^{1/2+\delta}\, \Big\},
\end{equation*}
satisfying
\begin{equation}\label{concentration estimate}
    |Ef_{\theta,v}(x)\mathbf{1}_{B(0,r) \setminus T_{\theta,v}}(x)| \lesssim_{\delta,N} r^{-N} \|f_{\theta,v}\|_{2}
\end{equation}
for all $N \in \N$ and $x \in \R^n$ with $|x_n| < r$. Note that $T_{\theta,v}$ as defined above is a tube with direction $G(\xi_{\theta})$ (where $G$ is the Gauss map as defined in \eqref{Gauss map}) which passes through the point $(-v,0) \in \R^n$. 

Rapidly decaying terms of the kind seen in \eqref{concentration estimate} are a regular feature of the forthcoming analysis and it is convenient to introduce the notation $\mathrm{RapDec}(r)$ to denote a non-negative term which is rapidly decreasing in $r$: that is, 
\begin{equation*}
    \mathrm{RapDec}(r) \lesssim_{\delta,N} r^{-N} \qquad \textrm{for all $N \in \N$.}
\end{equation*}
Thus, with this definition, the estimate in \eqref{concentration estimate} can be succinctly written as
\begin{equation*}
    |Ef_{\theta,v}(x)\mathbf{1}_{B(0,r) \setminus T_{\theta,v}}(x)| = \mathrm{RapDec}(r)\|f_{\theta,v}\|_{2}
\end{equation*}
for all  $x \in \R^n$ with $|x_n| < r$.




\subsection{Comparing wave packet decompositions at different scales}\label{comparing wave packet decompositions section} 

For $r$ as above, consider a smaller scale $\rho$ satisfying $r^{1/2} \leq \rho \leq r$ and a ball $B(y,\rho)$ with centre $y \in B(0,r)$. We decompose $f$ into wave packets over the ball $B(y,\rho)$ at this smaller spatial scale. The first step is to apply a transformation to recentre~$B(y,\rho)$ at the origin. In particular, write $Ef(x) = E\tilde{f}(\tilde{x})$ where $x = y + \tilde{x}$ for some $\tilde{x} \in B(0,\rho)$ and
\begin{equation*}\label{tilde}
\tilde{f}(\xi) := e^{i(\langle y' \!,\, \xi \rangle + y_n|\xi|^2)}f(\xi). 
\end{equation*}
The function $\tilde{f}$ is now decomposed into scale $\rho$ wave packets;
\begin{equation}\label{second}
    \tilde{f} = \sum_{(\tilde{\theta}, \tilde{v}) \in \T[\rho]} \tilde{f}_{\tilde{\theta},\tilde{v}}.
\end{equation}

A basic question, studied in detail in \cite[Section 7]{Guth}, is to understand how the two wave packet decompositions \eqref{wave packet decomposition}  and \eqref{second} relate to one another. For instance, suppose the significant contributions to $f$ come from a subcollection $\mathbb{W}$ of the scale~$r$ wave packets; which scale $\rho$ wave packets contribute significantly to $f$? 
To make this question precise, we introduce the following definition.

\begin{definition}\label{conc} The function $f \colon B^{n-1} \to \C$ is said to be concentrated on wave packets from $\mathbb{W}$ if 
\begin{equation*}
    \big\|\sum_{(\theta,v)\notin \mathbb{W}} f_{\theta,v}\big\|_{\infty} = \mathrm{RapDec}(r)\|f\|_2.
\end{equation*}
\end{definition}

With this definition, the following lemma provides a relationship between wave packet concentration properties at distinct scales.

\begin{lemma}[\cite{Guth}]\label{relation between scales lemma} If $f$ is concentrated on scale $r$ wave packets $\mathbb{W} \subseteq \T[r]$, then $\tilde{f}$ is concentrated on a set of wave packets $\widetilde{\mathbb{W}} \subseteq \T[\rho]$ with the following property: for every $(\tilde{\theta}, \tilde{v}) \in \widetilde{\mathbb{W}}$ there exists a wave packet $(\theta, v) \in \mathbb{W}$ such that
\begin{enumerate}[i)]
\item $ \mathrm{dist}_H\big(T_{\tilde{\theta}, \tilde{v}} + y, T_{\theta,v} \cap B(y, \rho)\big) \lesssim r^{1/2 + \delta}$;
\item $  \angle(G(\xi_{\theta}),G(\xi_{\tilde{\theta}})) \lesssim \rho^{-1/2}$.
\end{enumerate}
Here $\mathrm{dist}_H$ denotes the Hausdorff distance. 
\end{lemma}

\begin{figure}
\begin{center}

\resizebox {0.7\textwidth} {!} {
\begin{tikzpicture}

 {
 \filldraw[yellow!50,cm={cos(3) ,-sin(3) ,sin(3) ,cos(3) ,(0 cm, 0 cm)}](0.6,5) -- (0.6,-4.5) -- (-0.6,-4.5) -- (-0.6,4.5) -- (0.6,4.5);
  \draw[yellow,thick, cm={cos(3) ,-sin(3) ,sin(3) ,cos(3) ,(0 cm, 0 cm)}](0.6,5) -- (0.6,-4.5) -- (-0.6,-4.5) -- (-0.6,4.5) -- (0.6,4.5);
  \draw[yellow, thick, dashed, cm={cos(3) ,-sin(3) ,sin(3) ,cos(3) ,(0 cm, 0 cm)}](1,4.5) -- (1,-4.5) -- (-1,-4.5) -- (-1,4.5) -- (1,4.5);
}

 \filldraw[blue!20,cm={cos(0-2) ,-sin(0-2) ,sin(0-2) ,cos(0-2) ,(-0.5 cm, 0.3 cm)}](0.075,5) -- (0.075,-5) -- (-0.075,-5) -- (-0.075,5) -- (0.075,5);

 \draw[blue,cm={cos(0-2) ,-sin(0-2) ,sin(0-2) ,cos(0-2) ,(-0.5 cm, 0.3 cm)}] (0.075,5) -- (0.075,-5);
 \draw[blue,cm={cos(0-2) ,-sin(0-2) ,sin(0-2) ,cos(0-2) ,(-0.5 cm, 0.3 cm)}] (-0.075,-5) -- (-0.075,5);

\fill[fill=white] (4cm,0) arc [white,radius=4cm, start angle=0, delta angle=180]
                  -- (-5.4cm,0) arc [white, radius=5.4cm, start angle=180, delta angle=-180]
                  -- cycle;
\fill[fill=white] (4cm,0) arc [white,radius=4cm, start angle=0, delta angle=-180]
                  -- (-5cm,0) arc [white, radius=5cm, start angle=180, delta angle=180]
                  -- cycle;

\draw[blue,thick,dashed] (0,0) circle (4cm);

 \node[scale=1.5] at  (-1.9cm,1.5) {$ T_{\tilde{\theta},\tilde{v}}+y$};
  \node[scale=1.5] at  (1.5cm,-0.5cm) {$ T_{\theta,v} \cap B(y, \rho)$};
\node[right, scale=1.5] at  (3.2cm,-3.2cm) {$B(y, \rho)$};

{
    \draw[black,line width=0.5mm, ->, cm={cos(3) ,-sin(3) ,sin(3) ,cos(3) ,(0 cm, 0 cm)}](0,1.45) -- (0,4.85) node [above, right, scale = 1.3] {$G(\xi_{\theta})$}; 
}

{
\draw[black,line width=0.5mm, ->, cm={cos(0-2) ,-sin(0-2) ,sin(0-2) ,cos(0-2) ,(-0.5 cm, 0.3 cm)}](0,1.2) -- (0,4.55) node [above, left, scale = 1.3] {$G(\xi_{\tilde{\theta}})$};
}

		\end{tikzpicture}}
		
\caption{For every $(\tilde{\theta}, \tilde{v}) \in \widetilde{\mathbb{W}}$ there exists a `parent' wave packet $(\theta,v) \in \mathbb{W}$ such that: i) $T_{\tilde{\theta}, \tilde{v}}+y$ (denoted here in \colorbox{blue!20}{blue}) is contained a fixed dilate of $T_{\theta,v} \cap B(y,\rho)$ (denoted here in \colorbox{yellow!50}{yellow}) and ii) the angle between the directions is $O(\rho^{-1/2})$.
}
\label{tube diagram}
\end{center}
\end{figure}
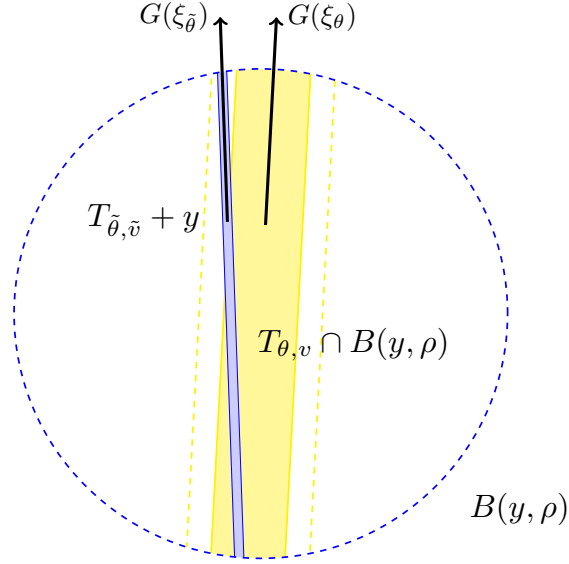

The lemma tells us that every small scale wave packet $(\tilde{\theta},\tilde{v}) \in \widetilde{\mathbb{W}}$ has a `parent' large scale wave packet $(\theta,v) \in \mathbb{W}$ such that $T_{\tilde{\theta},\tilde{v}}$ both lies close to $T_{\theta,v}$ and points in a similar direction to $T_{\theta,v}$. This behaviour is represented in Figure~\ref{tube diagram}. 

\section{Tangential wave packets}\label{tangential}

 We begin by giving the precise definition of what it means for a tube $T_{\theta,v}$ to be tangent to $\bZ$; throughout this section $\bZ\subset \R^n$ will denote an $m$-dimensional transverse complete intersection and $0 < \delta \ll \delta_m \ll 1$ are fixed small parameters, where~$\delta$ is as in the  previous section.

\begin{definition}\label{tangent definition} Letting $r \gg 1$ and $y \in B_{\>\!\!R}$,  a (translated) tube $T_{\theta,v} + y$ for $(\theta,v) \in \T[r]$ is said to be $r^{-1/2+\delta_m}$-tangent to $\bZ$ in $B(y,r)$ if:
\begin{enumerate}[i)]
\item $T_{\theta,v}+y  \subseteq N_{\!r^{1/2+\delta_m}}\bZ \cap B(y,r)$;
\item For any $x \in T_{\theta,v}+y$ and $z \in \bZ \cap B(y,r)$ with $|z-x| \lesssim r^{1/2 + \delta_m}$ one has
\begin{equation*}
    \angle(G(\theta), T_z\bZ) \lesssim r^{-1/2+\delta_m}. 
\end{equation*}
\end{enumerate}
\end{definition}

Throughout this section, we consider a function $g$  which is concentrated on tangential wave packets in the sense that
\begin{equation}\label{tangency hypothesis}
    g = \sum_{(\theta,v) \in \T_{\bZ}[r]} g_{\theta,v} + \mathrm{RapDec}(r)\|g\|_2
\end{equation}
where
\begin{equation*}
 \T_{\bZ}[r] := \Big\{\, (\theta,v) \in \T[r]\, :\, T_{\theta,v} \textrm{ is $r^{-1/2+\delta_m}$-tangent to $\bZ$ in $B(0,r)$}\, \Big\}.
\end{equation*}
An important ingredient in the proof of Theorem~\ref{main theorem} will be to understand what can be said about $Eg$ under this tangency hypothesis. Recall from the discussion in Section~\ref{Wolff axioms section} that there are three useful estimates at our disposal:
\begin{itemize}
     \item Vanishing property of the $k$-broad norms,
     \item Transverse equidistribution estimates,
     \item Bounds arising from the polynomial Wolff axioms.
\end{itemize}
The purpose of this section is to provide the precise details of all three of these estimates. The first two were observed and used by Guth~\cite{Guth} to prove restriction estimates in high dimensions. The polynomial Wolff axioms were also applied earlier by Guth~\cite{Guth2016} in the special case of 2-surfaces in $\R^3$ to study the restriction problem in 3-dimensions. 




\subsection{Vanishing property of the $k$-broad norms} The key advantage of working with $k$-broad norms rather than classical $L^p$ inequalities is that the former satisfy the following property. 

\begin{lemma}\label{key k-broad lemma} Let $r \gg 1$, let $1 \leq m < k \leq n$, and let $\bZ$ be $m$-dimensional. Suppose that $g$ is concentrated on wave packets from $\T_{\bZ}[r]$. Then
\begin{equation*}
    \|Eg\|_{\BL{k,A}^p(B_{r})} = \mathrm{RapDec}(r)\|g\|_2.
\end{equation*}
\end{lemma}

The lemma follows fairly directly from the definition of the $k$-broad norms and the basic properties of the wave packet decomposition. The simple argument can be readily extracted from the beginning of the proof of Proposition 8.1 in~\cite{Guth}.




\subsection{Comparing tangency properties at different scales}\label{tangency at different scales section} The description of the transverse equidistribution estimates is a little involved and will require some preliminary definitions. In Section~\ref{comparing wave packet decompositions section} we compared wave packet concentration properties at different spatial scales; we now pursue this investigation further in the tangential scenario.

As above, suppose $g$ is concentrated on wave packets from $\T_{\bZ}[r]$. Once again, let $r^{1/2} \leq \rho \leq r$ be a choice of smaller spatial scale and consider some $\rho$-ball $B(y, \rho)$ with centre $y \in B(0,r)$. Lemma~\ref{relation between scales lemma} can be used to analyse the tangency properties of the scale $\rho$ wave packets defined over the ball $B(y,\rho)$. To see this, first write 
\begin{equation*}
    \tilde{g}(\xi) := e^{i(\langle y', \xi \rangle + y_n|\xi|^2)}g(\xi),
\end{equation*}
as in Section~\ref{comparing wave packet decompositions section}. By Lemma~\ref{relation between scales lemma}, the function $\tilde{g}$ is concentrated on scale $\rho$ wave packets which each admit a `parent' wave packet in $\T_{\bZ}[r]$. The scale $\rho$ wave packets therefore inherit tangency properties from their parents. It turns out that the angle condition inherited by the scale $\rho$ wave packets is very strong, but the containment property is too weak to ensure that the scale $\rho$ wave packets are tangent to $\bZ$ itself. However, as shown in \cite[Section 7]{Guth}, the function $\tilde{g}$ \emph{is} concentrated on scale $\rho$ wave packets~$T_{\tilde{\theta},\tilde{v}}$ which are tangent to various \emph{translates} of $\bZ$. A schematic of this behaviour is provided in Figures~\ref{tangent figure 1} and~\ref{tangent figure 2} below.

To make the preceding discussion more precise, given $b \in \R^n$ let $\T_b[\rho]$ be the subcollection of $\T[\rho]$ consisting of those wave packets which are $\rho^{-1/2+\delta_m}$-tangent to $\bZ  + b-y$ in $B(0,\rho)$. At least heuristically, there is a finite set of translates $\fB \subseteq B(0,r^{1/2+\delta_m})$ such that the  $\{\T_b[\rho] : b \in \fB\}$ are pairwise disjoint and
\begin{equation}\label{g decomposition}
    \tilde{g} = \sum_{b \in \fB} \tilde{g}_b + \mathrm{RapDec}(r)\|g\|_2,\quad \text{where}\quad 
\tilde{g}_b := \sum_{(\tilde{\theta}, \tilde{v}) \in \T_b[\rho]} \tilde{g}_{\tilde{\theta}, \tilde{v}}.
\end{equation}

By the spatial concentration property of the wave packets, it follows that
\begin{equation*}
    E\tilde{g}_b(\tilde{x}) = \mathbf{1}_{N_{\!\rho^{1/2 + \delta_m}}(\bZ+b)}(x)Eg(x) + \mathrm{RapDec}(r)\|g\|_2 
\end{equation*}
whenever $x = \tilde{x} + y$ for some $\tilde{x} \in B(0,\rho)$. The  decomposition in \eqref{g decomposition} therefore breaks~$E\tilde{g}$ into pieces with the property that each piece is concentrated on a $\rho^{1/2+\delta_m}$-neighbourhood of some translate of $\bZ$.

Finding the set of translates $\fB$ involves some technicalities and the precise statements are perhaps not quite as clean as the above discussion suggests. A rigorous version of \eqref{g decomposition} is given by the following proposition, which is implicit in~\cite{Guth} and is described more explicitly in~\cite{GHI}.

\begin{proposition}\label{random cover proposition} Let $B(y,\rho) \cap N_{\!\rho^{1/2+\delta_m}}\bZ \neq \emptyset$ and let $g$ be concentrated on wave packets from  $\T_{\bZ}[r]$. Then there is a set of translates $\mathfrak{B} \subset B(0,r^{1/2 + \delta_m})$ such that 
\begin{equation*}
    \|Eg\|_{\BL{k,A}^p(B(y,\rho))}^p \lesssim \log^2 r \sum_{b \in \mathfrak{B}} \|E\tilde{g}_b\|_{\BL{k,A}^p(B(0,\rho) \cap N_{\!\rho^{1/2 + \delta_m}}(\bZ - y + b))}^p + \mathrm{RapDec}(r)\|g\|_2^2
\end{equation*}
 and
\begin{equation*}
    \sum_{b \in \mathfrak{B}}\|\tilde{g}_b\|_2^2 \lesssim \|g\|_2^2.
\end{equation*}
\end{proposition}

The lemma can be proved by independently selecting the translates $b$ at random, although this argument involves some technicalities. See the proof of \cite[Proposition 8.1]{Guth} or \cite[Lemma 10.5]{GHI} for further details.

\begin{figure}
    \centering
         \begin{tikzpicture}[scale=1]
\draw [black, dashed, domain=-4:4,name path=upper_A] plot (\x, {\x/4-\x/4*\x/4*\x/4 + 1}); 
\draw [black, dashed, domain=-4:4,name path=lower_A] plot (\x, {\x/4-\x/4*\x/4*\x/4 -1});

\tikzfillbetween[
    of=upper_A and  lower_A, on layer=bg] {pattern=dots, pattern color=black!50};

{
 \filldraw[yellow!50,cm={cos(85) ,-sin(85), sin(85), cos(85) ,(0 cm, 0 cm)}](0.6,4) -- (0.6,-4) -- (-0.6,-4) -- (-0.6,4) -- (0.6,4);
  \draw[yellow,thick, cm={cos(85) ,-sin(85), sin(85), cos(85) ,(0 cm, 0 cm)}](0.6,4) -- (0.6,-4);
 \draw[yellow,thick, cm={cos(85) ,-sin(85), sin(85), cos(85) ,(0 cm, 0 cm)}](-0.6,4) -- (-0.6,-4);
}

\draw [black, line width=0.5mm, domain=-4:4] plot (\x, {\x/4-\x/4*\x/4*\x/4});

\fill[fill=white] (3.7cm,0) arc [white,radius=3.7cm, start angle=0, delta angle=180]
                  -- (-4.2cm,0) arc [white, radius=4.2cm, start angle=178, delta angle=-180]
                  -- cycle;
\fill[fill=white] (3.7cm,0) arc [white,radius=3.7cm, start angle=0, delta angle=-180]
                  -- (-4.2cm,0) arc [white, radius=4.2cm, start angle=180, delta angle=180]
                  -- cycle;

\draw[blue,thick,dashed] (0,0) circle (3.7cm);

\node[right, scale=1.5] at  (4cm,0cm) {$\mathbf{Z}$};
\node[scale=1.5] at  (1cm, 1.7cm) {$N_{r^{1/2 + \delta_m}}\mathbf{Z}$};
\node[scale=1.5] at  (4cm, -2.5cm) {$B(y, \rho)$};

\node[scale=1.5, left] (A) at  (-4cm, 1.5cm) {$T_{\theta,v}$};
\node[scale=1.5] (B) at  (-2cm, 0cm) {};

\draw[->, line width=0.5mm](A) to [bend left,looseness=0.7] (B);
\end{tikzpicture}
    \caption{The large scale tube $T_{\theta,v}$ is $r^{-1/2+\delta_m}$-tangent to $\bZ$ in $B(0,r)$. Here we consider its intersection with $B(y,\rho)$.}
    \label{tangent figure 1}
\end{figure}
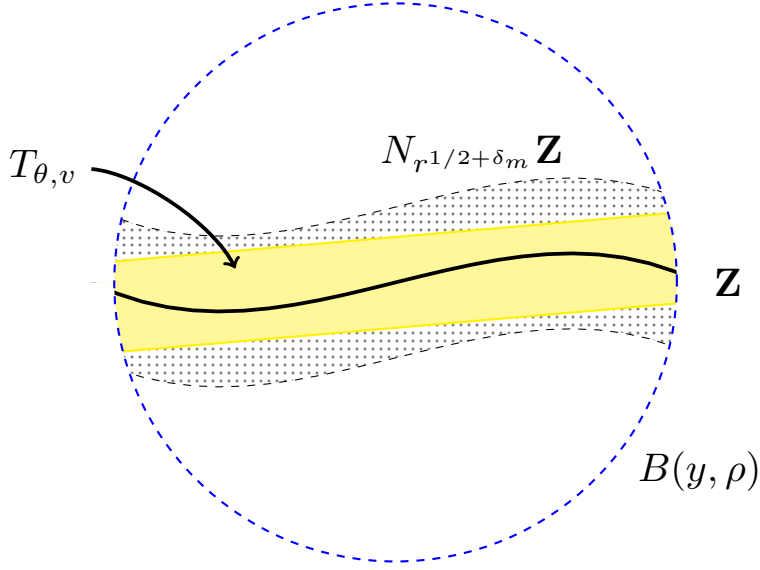

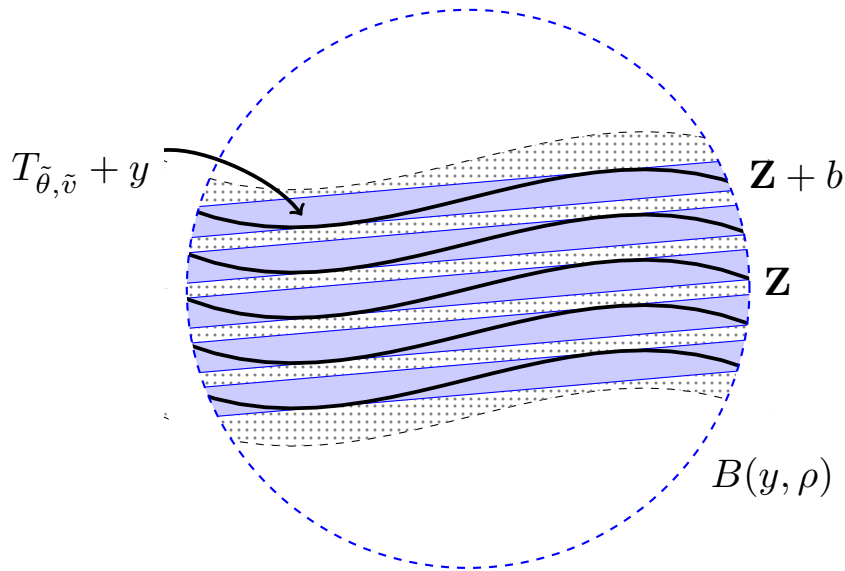
\begin{figure}
    \centering
    \begin{tikzpicture}


\draw [black, dashed, domain=-4:4,name path=lower_B] plot (\x, {\x/4-\x/4*\x/4*\x/4 -1.7}); 
\draw [black, dashed, domain=-4:4,name path=upper_B] plot (\x, {\x/4-\x/4*\x/4*\x/4 +1.7}); 

\tikzfillbetween[
    of=upper_B and  lower_B] {pattern=dots, pattern color=black!50};

\foreach \a in {-1.2, -0.6,..., 1.2}
    {
 \filldraw[blue!20,cm={cos(85) ,-sin(85) ,sin(85) ,cos(85) ,(0 cm, \a cm)}](0.2,3.8) -- (0.2,-3.8) -- (-0.2,-3.8) -- (-0.2,3.8) -- (0.2,3.8);
  \draw[blue,cm={cos(85) ,-sin(85) ,sin(85) ,cos(85) ,(0 cm, \a cm)}](0.2,3.8) -- (0.2,-3.8) -- (-0.2,-3.8) -- (-0.2,3.8) -- (0.2,3.8);
 
}
\foreach \a in {-1.2, -0.6,..., 1.2}
    {
\draw [black, line width=0.5mm, domain=-4:4] plot (\x, {\x/4-\x/4*\x/4*\x/4 + \a}); 
}

\fill[fill=white] (3.7cm,0cm) arc [white,radius=3.7cm, start angle=0, delta angle=180]
                  -- (-4.3cm,0cm) arc [white, radius=4.3cm, start angle=180, delta angle=-180]
                  -- cycle;
\fill[fill=white] (3.7cm,0cm) arc [white,radius=3.7cm, start angle=0, delta angle=-180]
                  -- (-4.3cm,0cm) arc [white, radius=4.3cm, start angle=180, delta angle=180]
                  -- cycle;

\draw[blue,thick,dashed] (0,0) circle (3.7cm);

\node[right, scale=1.5] at  (3.7cm,0.1cm) {$\mathbf{Z}$};

\node[right, scale=1.5] at  (3.5cm,1.5cm) {$\mathbf{Z}+b$};

\node[scale=1.5] at  (4cm, -2.5cm) {$B(y, \rho)$};

\node[scale=1.5, left] (A) at  (-4cm, 1.5cm) {$T_{\tilde{\theta},\tilde{v}} + y$};
\node[scale=1.5] (B) at  (-2cm, 0.8cm) {};

\draw[->, line width=0.5mm](A) to [bend left,looseness=0.7] (B);

\end{tikzpicture}
    \caption{The scale $\rho$ wave packets are partitioned into collections $\T_b[\rho]$. For each $(\tilde{\theta},\tilde{v}) \in \T_b[\rho]$ the corresponding tube $T_{\tilde{\theta},\tilde{v}}+y$ is tangent to the translate $\bZ + b$ in $B(y,\rho)$.}
    \label{tangent figure 2}
\end{figure}




\subsection{Transverse equidistribution estimates}\label{transverse equidistribution section} If $h \colon B^{n-1} \to \C$ is concentrated on wave packets from $\T_{\bZ}[r]$, then this property constrains the support of $h$ (since points of the support of $h$ roughly correspond to directions of the wave packets). This in turn influences the behaviour of $Eh$ via the uncertainty principle. In particular, it transpires that $Eh$ is essentially constant at scale $r^{1/2}$ in directions transverse to the variety $\bZ$. This phenomenon is encapsulated in the \emph{transverse equidistribution estimate} of \cite[Section~6]{Guth} which roughly states that
\begin{equation}\label{equidistribution heuristic}
    \frac{1}{|N_{\!\rho^{1/2}}\bZ \cap B_{r^{1/2}}|}\int_{N_{\!\rho^{1/2}}\bZ \cap B_{\!r^{1/2}}}|Eh|^2 \lesssim  \frac{1}{|B_{r^{1/2}}|}\int_{B_{\!r^{1/2}}}|Eh|^2 
\end{equation}
for any  $r^{1/2}$-ball $B_{r^{1/2}}$ and $1 \leq \rho \leq r$.
An informative case to have in mind is given by taking $\bZ$ to be a plane in the co-ordinate hyperplane perpendicular to $e_n$; in this situation, a rigorous version of the above inequality can be readily verified along the lines discussed above. For the general case, the reader is referred to Sections~2 and 6 of~\cite{Guth} for a more detailed discussion of the transverse equidistribution phenomenon, which plays a fundamental role in~\cite{Guth} and also here.

It is of particular interest to apply these observations to $h := \tilde{g}_b$, where $\tilde{g}_b$ is one of the functions introduced in the previous subsection. Indeed, by the discussion in Section~\ref{tangency at different scales section}, the operator $|E\tilde{g}_b|$ is concentrated in $N_{\rho^{1/2+\delta_m}}(\bZ - y + b)$ and so expressions of the form of the left-hand side of \eqref{equidistribution heuristic} naturally arise in this context. 

Estimates for $L^2$ quantities involving $E\tilde{g}_b$ can be related to $L^2$ estimates for the input function $\tilde{g}_b$ via Plancherel's theorem or, more precisely, the energy identity~\eqref{energy identity}. The following consequence of transverse equidistribution will be useful, which is established in Section 7 of~\cite{Guth}.

\begin{lemma}[Guth~\cite{Guth}] \label{transverse equidistribution lemma} Let $1 \leq \rho'\le  \rho \leq r$ and  $|b| \lesssim r^{1/2+\delta_m}$. Let $\bZ$ be $m$-dimensional and let $g$ be concentrated on wave packets from $\T_{\bZ}[r]$. Then
\begin{equation}\label{transverse equidistribution inequality}
\max_{\theta : (\rho')^{-1/2}-\mathrm{cap}}\|\tilde{g}_b\|_{L^2(\theta)}^2 \lesssim_{\Deg \bZ} \,\,\,\,r^{O(\delta_m)}\Big(\frac{r}{\rho}\Big)^{-\frac{n-m}{2}}\!\!\!\!\max_{\theta : (\rho')^{-1/2}-\mathrm{cap}}\|g\|_{L^2(\theta)}^2,
\end{equation}
where $\tilde{g}_b$ is defined with respect to scale $\rho$ wave packets as in \eqref{g decomposition}.
\end{lemma}
Note that the factor gained in \eqref{transverse equidistribution inequality} is the ratio of the volumes of the sets of integration in  \eqref{equidistribution heuristic}.
 The inequality \eqref{transverse equidistribution inequality} is explicitly stated in \cite[Lemma 7.6]{Guth} for the case $\rho' = 1$; the version for general  $1 \leq \rho'\leq \rho$ can be deduced via similar arguments (see also the equation (8.26) from~\cite{Guth}).




\subsection{Applying the polynomial Wolff axioms} Theorem~\ref{Katz--Rogers theorem} can be expressed in terms of wave packets. 

\begin{proposition}[\cite{KR}]\label{reformulated Katz--Rogers} Let $\delta>0$, $c,r\ge1$ and $\bZ \subseteq \R^n$ be $m$-dimensional. If $\mathbb{W} \subseteq \T[r]$ is such that $  T_{\theta,v} \subseteq N_{cr^{1/2}}\bZ$ for all $(\theta,v) \in \mathbb{W}$, then
\begin{equation*}
    \#\Big\{\, \theta\, :\, (\theta,v) \in \mathbb{W} \textrm{ for some } v \in r^{1/2}\mathbb{Z}^{n-1}\, \Big\} \lesssim_{\Deg \bZ} c^{n-m}r^{\frac{m-1}{2} + \delta}.
\end{equation*}
\end{proposition}

From this geometric bound, we deduce an  estimate involving the averaged norm
$$\|f\|_{L^2_{\mathrm{avg}}(\theta)} := \Big(\frac{1}{|\theta|}\int_\theta |f(\xi)|^2\ud \xi\Big)^{1/2},$$
which is a higher dimensional generalisation of an inequality that featured prominently in~\cite{Guth2016}. 

\begin{lemma}\label{nesting lemma} Let $\delta>0$, $c,r\ge1$ and $\bZ \subseteq \R^n$ be $m$-dimensional. If $g$ is concentrated on wave packets $(\theta,v) \in \T[r]$ satisfying $T_{\theta,v} \subseteq N_{cr^{1/2}}\bZ$, then
\begin{equation*}
    \|g\|_{L^2(B^{n-1})}^2 \lesssim_{\Deg \bZ} c^{n-m}r^{-\frac{n - m}{2} + \delta} \max_{\theta : r^{-1/2}-\mathrm{cap}} \|g\|_{L^2_{\mathrm{avg}}(\theta)}^2.
\end{equation*} 
\end{lemma}

\begin{proof} By the concentration hypothesis one may write
\begin{equation*}
    g = \sum_{(\theta,v) \in \mathbb{W}} g_{\theta,v} + \mathrm{RapDec}(r)\|g\|_2
\end{equation*}
where $\mathbb{W}$ are scale $r$ wave packets satisfying $  T_{\theta,v} \subseteq N_{cr^{1/2}}\bZ$. Given an $r^{-1/2}$-cap $\theta$, define 
\begin{equation*}
\T_{\bZ}(\theta) := \Big\{\, v \in r^{1/2}\Z^{n-1}\, :\, (\theta,v) \in \mathbb{W}\, \Big\} 
\end{equation*}
 and let $\Theta_{\bZ}$ denote the collection of all $r^{-1/2}$-caps $\theta$ for which $\T_{\bZ}(\theta) \neq \emptyset$. Thus, by the orthogonality and support properties of the wave packets,
\begin{equation*}
    \|g\|_{2}^2 \sim \sum_{\theta \in \Theta_{\bZ}} \Big\|\sum_{v \in \T_{\bZ}(\theta)} g_{\theta,v}\Big\|_{2}^2 \lesssim  \sum_{\theta \in \Theta_{\bZ}} \|g\|_{L^2(\theta)}^2.
\end{equation*}
To prove the lemma it therefore suffices to show that
\begin{equation*}
    \# \Theta_{\bZ} \lesssim_{\Deg \bZ} c^{n-m}r^{\frac{m-1}{2} + \delta},
\end{equation*}
but this immediately follows from Proposition~\ref{reformulated Katz--Rogers}.
\end{proof}



\section{Finding polynomial structure}\label{structure lemma section}

The purpose of this section is to reformulate the core of the (inductive) proofs in~\cite{Guth2016, Guth}  as a recursive process. The argument will in fact be presented as two separate algorithms:
\begin{itemize}
    \item \texttt{[alg 1]} is the more involved of the two and is presented in the current section. It effects a dimensional reduction, essentially passing from an $m$-dimensional to an $(m-1)$-dimensional situation. 
    \item  \texttt{[alg 2]} is described in Section~\ref{proof section} below. It consists of repeated application of the first algorithm to reduce to a minimal dimensional case. 
\end{itemize}
Comparing the present analysis with the original induction arguments of Guth, \texttt{[alg 1]} corresponds to the induction on the radius in the proof of Proposition 8.1 of~\cite{Guth}, whilst \texttt{[alg 2]} corresponds to the induction on dimension.

\subsection*{The first algorithm}  Throughout this section let $p \geq 2$,  $0< \varepsilon \ll 1$ be fixed and
\begin{equation}\label{small parameters}
    \varepsilon^{C} \leq \delta \ll \delta_n \ll \delta_{n-1} \ll \dots \ll \delta_1 \ll \delta_0 \ll \varepsilon
\end{equation}
be a family of small parameters. Taking, for instance, $\delta_0 := \varepsilon^{10}$, $\delta_j := \delta_{j-1}^{10}$ for $1 \leq j \leq n$ and $\delta := \delta_n^{10}$ suffices.  These parameters play a rather technical role\footnote{They are essentially used to compensate for certain $r^{\bar{C}\delta_m}$-losses arising from the transverse equidistribution lemma.} and are chosen so as to satisfy the requirements of the forthcoming proof.\\ 

\paragraph{\underline{\texttt{Input}}} \texttt{[alg 1]} will take as its input:
\begin{itemize}
    \item An $r$-ball $B_{r} \subset \R^n$ for some choice of large scale $r \gg 1$.
    \item A transverse complete intersection $\bZ$ of dimension $m \geq 2$.
  
    \item A function $f \in B^{n-1} \to \C$ concentrated on wave packets which are $r^{-1/2+\delta_m}$-tangent to $\bZ$ in $B_{r}$.
    
    \item An admissible large integer $A \in \N$. 
\end{itemize}

\begin{remark} The integer $A$ corresponds to the $A$ parameter featured in the definition of the broad norm. It is chosen large enough to facilitate repeated application of Lemma \ref{triangle inequality lemma} and Lemma \ref{logarithmic convexity inequality lemma}. These lemmas will be used no more than $\delta^{-2}$ times and so it suffices to take $A \geq 2^{\delta^{-2}}$: see the discussion following \eqref{letter count} below.
\end{remark}

\vspace{1em}

\paragraph{}The description of the output of the algorithm is, unfortunately, far more involved.

\vspace{1em}

%
%
%
\paragraph{\underline{\texttt{Output}}} \texttt{[alg 1]} will output a finite sequence of sets $(\sE_j)_{j=0}^J$, which are constructed via a recursive process. Each $\sE_j$ is referred to as an \emph{ensemble} and contains all the relevant information coming from the $j$th step of the algorithm. In particular, the ensemble $\sE_j$ consists of:
\begin{itemize}
    \item A word $\fh_j$ of length $j$ in the alphabet $\{\ta, \tc\}$, is referred to as a {\it history}. The rationale behind this notation is that $\ta$ is an abbreviation of `algebraic' and $\tc$ `cellular'. The words $\fh_j$ are recursively defined by successively adjoining a single letter. Each $\fh_j$ records how the cells $O_j \in \O_j$ were constructed via repeated application of the polynomial partitioning theorem and, in particular, whether the algebraic or cellular case held in successive stages of the process.
    \item A choice of spatial scale $\rho_j \geq 1$. The $\rho_j$ will in fact be completely determined by the initial scale $r$ and the history $\fh_j$. In particular, define an auxiliary exponent $\tilde{\delta}_{m-1}$ by 
    \begin{equation}\label{auxiliary delta}
        (1-\tilde{\delta}_{m-1})(1/2 + \delta_{m-1}) = (1/2 + \delta_m),
    \end{equation}    
noting that $\delta_{m-1}/2 \leq \tilde{\delta}_{m-1} \leq 2\delta_{m-1}$.  
Let $\sigma_{k} \colon [1,\infty) \to [0,\infty)$ be given by
    \begin{equation*}
        \sigma_{k}(\rho) := \left\{ \begin{array}{ll}
        \displaystyle\frac{\rho}{2} &  \textrm {if the $k$th letter of $\fh_j$ is $\tc$} \\[6pt]
        \rho^{1-\tilde{\delta}_{m-1}} & \textrm{if the $k$th letter of $\fh_j$ is $\ta$}
        \end{array}\right. 
    \end{equation*}
for each $1 \leq k \leq j$. With these definitions, take 
\begin{equation*}
\rho_j := \sigma_{j} \circ \cdots \circ \sigma_{1}(r);
\end{equation*}
this sequence of scales is represented pictorially by the tree in Figure~\ref{scales figure}.

Note that each $\sigma_{k}$ is a decreasing function and therefore 
\begin{equation}\label{radius bounds}
\rho_j \leq r^{(1-\tilde{\delta}_{m-1})^{\#_{\sta}(j)}}   \quad \textrm{and} \quad \rho_j \leq \frac{r}{2^{\#_{\stc}(j)}}  
\end{equation}
  where $\#_{\bta}(j)$ and  $\#_{\btc}(j)$ denote the number of occurrences of $\ta$ and $\tc$ in the history $\fh_j$, respectively. 
   \end{itemize}
\begin{remark} It is perhaps useful to give some justification for the introduction of the auxiliary exponent $\tilde{\delta}_{m-1}$. In the algebraic case one passes from some scale $\rho_j$ to a new scale $\rho_{j+1} := \rho_j^{1-\tilde{\delta}_{m-1}}$; this is encoded in the definition of $\sigma_k(\rho)$ above. In the last step of the current algorithm the analysis passes to a lower dimensional variety (or we end up in a trivial small scale case). For this step, one wishes to analyse tangency properties of the wave packets at the new scale $\rho_{j+1}$ with respect to an $(m-1)$-dimensional variety. Looking at the second condition in Definition~\ref{tangent definition}, this involves analysis at the scale $\rho_{j+1}^{1/2+\delta_{m-1}}$. The formula \eqref{auxiliary delta} ensures that $\rho_{j+1}^{1/2+\delta_{m-1}} = \rho_{j}^{1/2+\delta_{m}}$; this allows certain tangency properties to be inherited at the new scale. 
\end{remark}
 \begin{itemize} 
  \item A family of subsets $\O_j$ of $\R^n$ which will be referred to as \emph{cells}. Each cell $O_j \in \O_j$ will have diameter at most $\rho_j$. 
\item A collection of functions $(f_{O_j})_{O_j \in \O_j}$. Each $f_{O_j}$ is concentrated on wave packets in $\T[\rho_j]$ which are $\rho_j^{-1/2+\delta_m}$-tangent to some translate of $\bZ$ on (a ball of radius $\rho_j$ containing) $O_j$. 
  \item A large integer $d \in \N$ which depends only on the admissible parameters and $\Deg \bZ$.
  \end{itemize}
  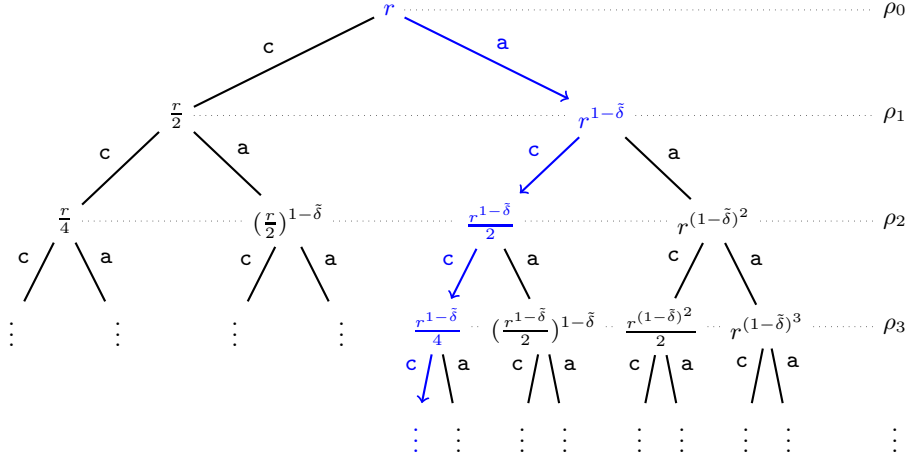
\begin{figure}
 \small
 \begin{center}
    \begin{tikzpicture}[scale=0.7, thick, level 1/.style={sibling distance=80mm},
      level 2/.style={sibling distance=42mm},
      level 3/.style={sibling distance=20mm},
       level 4/.style={sibling distance=8mm}, level distance=20mm]
      
      \node (root) {${\color{blue}r}$}
        child { node (c) {$\frac{r}{2}$}
          child { 
            node (cc) {$\frac{r}{4}$}
             child {node {$\vdots$} edge from parent
              node[above left] {\texttt{c}}}
             child {node {$\vdots$} edge from parent
              node[above right] {\texttt{a}}}
            edge from parent
              node[above left] {\texttt{c}}
              }
          child { 
            node (ca)  {$(\frac{r}{2})^{1-\tilde{\delta}}$}
            child {node {$\vdots$} edge from parent
              node[above left] {\texttt{c}}}
             child {node {$\vdots$} edge from parent
              node[above right] {\texttt{a}}}
            edge from parent
              node[above right] {\texttt{a}}}
          edge from parent
            node[above left] {\texttt{c}}}
        child { node (a) {${\color{blue}r^{1-\tilde{\delta}}}$}
          child { 
            node (ac) {${\color{blue}\frac{r^{1-\tilde{\delta}}}{2}}$}
              child {
                node (acc) {${\color{blue}\frac{r^{1-\tilde{\delta}}}{4}}$}
                child {node {${\color{blue}\vdots}$} edge from parent[blue, ->]
              node[above left] {\texttt{c}}}
             child {node {$\vdots$}edge from parent[black]
              node[above right] {\texttt{a}}}
                edge from parent[blue, ->]
                  node[above left] {\texttt{c}}}
              child {
                 node (aca) {$(\frac{r^{1-\tilde{\delta}}}{2})^{1-\tilde{\delta}}$}
                 child {node {$\vdots$} edge from parent
              node[above left] {\texttt{c}}}
             child {node {$\vdots$} edge from parent
              node[above right] {\texttt{a}}}
                 edge from parent[black]
                   node[above right] {\texttt{a}}}
              edge from parent[blue, ->]
                node[above left] {\texttt{c}}}
          child { 
            node (aa) {$r^{(1-\tilde{\delta})^2}$ }
              child {
                node (aac) {$\frac{r^{(1-\tilde{\delta})^2}}{2}$}
                child {node {$\vdots$} edge from parent
              node[above left] {\texttt{c}}}
             child {node {$\vdots$} edge from parent
              node[above right] {\texttt{a}}}
                edge from parent
                  node[above left] {\texttt{c}}}
              child {
                 node (aaa){$r^{(1-\tilde{\delta})^3}$}
                 child {node {$\vdots$} edge from parent
              node[above left] {\texttt{c}}}
                 child {node {$\vdots$}
                 child [grow=right] {node {$\vdots$} edge from parent[draw=none]
                 child [grow=up] {node (r3)  {$\rho_3$}  edge from parent[draw=none]
                 child [grow=up] {node (r2) {$\rho_2$}edge from parent[draw=none]
                 child [grow=up] {node (r1) {$\rho_1$} edge from parent[draw=none]
                 child [grow=up] {node (r0) {$\rho_0$} edge from parent[draw=none]}} }}}
                 edge from parent node [above right] {\texttt{a}} 
              }
                 edge from parent[black]
                   node[above right] {\texttt{a}}}
              edge from parent[black]
                node[above right] {\texttt{a}}}
           edge from parent[blue, ->]
             node[above right] {\texttt{a}}};
             
             \draw[very thin, dotted](root)->(r0);
             \draw[very thin, dotted](c)->(a);
             \draw[very thin, dotted](a)->(r1);
             \draw[very thin, dotted](cc)->(ca);
             \draw[very thin, dotted](ca)->(ac);
             \draw[very thin, dotted](ac)->(aa);
             \draw[very thin, dotted](aa)->(r2);
             \draw[very thin, dotted](acc)->(aca);
             \draw[very thin, dotted](aca)->(aac);
             \draw[very thin, dotted](aac)->(aaa);
             \draw[very thin, dotted](aaa)->(r3);

    \end{tikzpicture}
    \end{center} 
    \caption{The chain of scales $\rho_j$. Here we drop the subscript by writing $\tilde{\delta} := \tilde{\delta}_{m-1}$. The {\color{blue}blue} path marked by arrows corresponds to the situation where the word $\mathfrak{h}_j$ of length $j$ begins with the letters $\mathrm{\texttt{accc}}\dots$ for $4\leq j \leq J$. In this case the sequence $\rho_j$ is given by $\rho_0 = r$, $\rho_1 = r^{1-\tilde{\delta}}$, $\rho_2 = r^{1-\tilde{\delta}}/2$, $\rho_3 = r^{1-\tilde{\delta}}/4$, $\dots$.}
    \label{scales figure}
  \end{figure}

Moreover, the components of the ensemble are defined so as to ensure that, for certain coefficients\footnote{The quantity $d^{\#_{\stc}(j)\delta}$ may be large (and non-admissible). Nevertheless, these $d^{\#_{\stc}(j)\delta}$ losses will be compensated for by other gains in the argument: see Remark 10.1 below.}
\begin{equation}\label{small coefficients}
   C^{\mathrm{I}}_{j,\delta}(d,r), \; C^{\mathrm{II}}_{j,\delta}(d), \; C^{\mathrm{III}}_{j,\delta}(d,r) \lesssim_{d,\delta} r^{\delta_0} d^{\#_{\stc}(j)\delta}
\end{equation}
and $A_j := 2^{-\#_{\sta}(j)}A \in \N$, the following properties hold:\\

\paragraph{\underline{Property I}} Most of the mass of $\|Ef\|_{\BL{k,A}^p(B_{r})}^p$ is concentrated over the $O_j \in \O_j$:
\begin{equation} \tag*{$(\mathrm{I})_j$}
      \|Ef\|_{\BL{k,A}^p(B_{r})}^p \leq C^{\mathrm{I}}_{j,\delta}(d,r) \sum_{O_{j} \in \O_{j}} \|Ef_{O_j}\|_{\BL{k,A_j}^p(O_{j})}^p + \mathrm{err}(j)
\end{equation}
where $\mathrm{err}(j) := j r^{-N} \|f\|_{2}^p$ for some fixed $N \in \N$ is a harmless `error' term.\\

\paragraph{\underline{Property II}} The functions $f_{O_j}$ satisfy
\begin{equation}\tag*{$(\mathrm{II})_j$}
    \sum_{O_{j} \in \O_{j}} \|f_{O_j}\|_{2}^2 \leq C^{\mathrm{II}}_{j,\delta}(d)d^{\#_{\stc}(j)}  \|f\|_{2}^2.
\end{equation}
\paragraph{\underline{Property III}} Furthermore, each individual $f_{O_j}$ satisfies
\begin{equation}\tag*{$(\mathrm{III})_j$}
    \|f_{O_{j}}\|_{2}^2 \leq C_{j,\delta}^{\mathrm{III}}(d, r) \Big(\frac{r}{\rho_{j}}\Big)^{-\frac{n-m}{2}} d^{-\#_{\stc}(j)(m-1)}  \|f\|_{2}^2
\end{equation}
 and 
\begin{equation}\tag*{$(\mathrm{III}_{\mathrm{loc}})_j$}
 \max_{\theta : \rho^{-1/2}-\mathrm{cap}} \|f_{O_{j}}\|_{L_{\mathrm{avg}}^2(\theta)}^2 \leq C_{j,\delta}^{\mathrm{III}}(d, r) \Big(\frac{r}{\rho_{j}}\Big)^{-\frac{n-m}{2}}\!\!\!\!\!\max_{\theta : \rho^{-1/2}-\mathrm{cap}}\|f\|_{L_{\mathrm{avg}}^2(\theta)}^2
\end{equation}
for all $1 \leq \rho \leq \rho_j$.\\

The factors $C^{\mathrm{I}}_{j,\delta}(d,r)$, $C^{\mathrm{II}}_{j,\delta}(d)$ and $C^{\mathrm{III}}_{j,\delta}(d,r)$ play a minor technical role in the analysis but, nevertheless, it is useful to work with explicit formul\ae\, for these coefficients. In particular, they are defined by
\begin{align*}
C_{j,\delta}^{\mathrm{I}}(d, r) &:= d^{\#_{\stc}(j)\delta} (\log r)^{2\#_{\sta}(j)(1 + \delta)}, \\
C^{\mathrm{II}}_{j,\delta}(d) &:= d^{\#_{\stc}(j)\delta+n\#_{\sta}(j)(1+\delta)}, \\
C_{j,\delta}^{\mathrm{III}}(d, r) &:=  d^{\#_{\stc}(j)\delta+\#_{\sta}(j)\delta}r^{\bar{C}\#_{\sta}(j)\delta_m},
\end{align*}
where  $\bar{C}$ is some suitably chosen large constant. 

%
%
%

\subsection*{The initial step}

The initial ensemble $\sE_0$ is defined by taking:
\begin{itemize}
    \item $\fh := \emptyset$ to be the empty word;
    \item $\rho_0 := r$;
    \item $\O_0$ the collection consisting of a single cell $O_0 := N_{r^{1/2+\delta_m}}\bZ \cap B_{r}$;
    \item $f_{O_0} := f$.
\end{itemize} 
At this point it is convenient also to fix $d \in \N$ to be some large integer, to be determined later, which depends only on admissible parameters and $\Deg\bZ$. 

With these definitions, Property I holds due to the hypothesis on $f$ and the spatial concentration property of the wave packets, whilst Properties II and III both hold vacuously.

%
%
%

\subsection*{The recursive step} Assume the ensembles $\sE_0, \dots, \sE_j$ have all been constructed for some $j \in \N_0$ and that they all satisfy the desired properties. \\

\paragraph{\underline{\texttt{Stopping conditions}}} The algorithm has two stopping conditions which are labelled \texttt{[tiny]} and \texttt{[tang]}. 
\begin{itemize}
\item[\texttt{Stop:[tiny]}] The algorithm terminates if $\rho_j \leq r^{\tilde{\delta}_{m-1}}$.
\end{itemize}

\begin{itemize}
\item[\texttt{Stop:[tang]}]Let $C_{\textrm{\texttt{tang}}}$ and $C_{\alg}$ be fixed admissible constants, chosen large enough to satisfy the forthcoming requirements of the proof, and $\tilde{\rho} := \rho_j^{1 - \tilde{\delta}_m}$. The algorithm terminates if the inequalities
\begin{equation*}
    \sum_{O_j \in \O_j} \|Ef_{O_j}\|_{\BL{k,A_j}^p(O_{j})}^p \leq C_{\textrm{\texttt{tang}}}  \sum_{S \in \Sc} \|Ef_S\|_{\BL{k,A_j/2}^p(B_{\tilde{\rho}}[S])}^p
\end{equation*}
and
\begin{align}\label{tangent stopping conditions}
   \sum_{S \in \Sc}\|f_{S}\|_2^2  &\leq C_{\textrm{\texttt{tang}}}r^{n\tilde{\delta}_m}\sum_{O_j \in \O_j}\|f_{O_j}\|_2^2; \\ 
  \nonumber
  \max_{\substack{S \in \Sc\\\theta : \rho^{-1/2}-\mathrm{cap}}}\|f_{S}\|_{L^2_{\mathrm{avg}}(\theta)}^2 &\leq C_{\textrm{\texttt{tang}}}\!\!\! \max_{\substack{O_{j} \in \O_{j} \\\theta : \rho^{-1/2}-\mathrm{cap}}}\|f_{O_{j}}\|_{L^2_{\mathrm{avg}}(\theta)}^2
\end{align}
hold for all $1 \leq \rho \leq \tilde{\rho}$ for some choice of:
\end{itemize}
\begin{itemize}
    \item $\Sc$ a collection of transverse complete intersections in $\R^n$ all of equal dimension $m-1$ and degree at most $C_{\alg}d$;
    \item $B_{\tilde{\rho}}[S]$ an assignment of a $\tilde{\rho}$-ball to each $S \in \Sc$;
    \item $f_S$ an assignment of a function to each $S \in \Sc$ which is concentrated on wave packets $\tilde{\rho}^{-1/2+\delta_{m-1}}$-tangent to $S$ on $B_{\tilde{\rho}}$ in the sense of Definition~\ref{tangent definition}.
\end{itemize}

The stopping condition \texttt{[tang]} is somewhat involved, but it can be roughly interpreted as forcing the algorithm to terminate if one can pass to a lower dimensional situation.

If either of the above conditions hold, then the stopping time is defined to be $J := j$. Recalling \eqref{radius bounds}, the stopping condition \texttt{[tiny]} implies that the algorithm must terminate after finitely many steps and, moreover, 
\begin{equation}\label{letter count}
   \#_{\bta}(J) \lesssim \delta_{m-1}^{-1} \log(\delta_{m-1}^{-1}) \quad \textrm{and} \quad  \#_{\btc}(J) \lesssim \log r.
\end{equation}
These estimates can be combined with the explicit formul\ae\, for $C_{j,\delta}^{\mathrm{I}}(d, r)$, $C_{j,\delta}^{\mathrm{II}}(d)$ and $C_{j,\delta}^{\mathrm{III}}(d, r)$ to show that the bound \eqref{small coefficients} always holds, provided $\delta_m$ is chosen to be sufficiently small relative to $\delta_{m-1}$. Furthermore, by choosing $A \geq 2^{\delta^{-2}}$, say, one can ensure that the $A_j$ defined above are indeed integers.\\

\paragraph{\underline{\texttt{Recursive step}}}

Suppose that neither stopping condition \texttt{[tiny]} nor \texttt{[tang]} is met. One proceeds to construct the ensemble $\sE_{j+1}$ as follows. 

Given $O_j \in \O_j$, apply the polynomial partitioning theorem, Theorem \ref{partitioning theorem},  with degree $d$ to 
\begin{equation*}
    \|Ef_{O_j}\|_{\BL{k,A_j}^p(O_j\cap N_{\!\!\rho_j^{1/2 + \delta_m}}(\bZ + x_{O_j}))}^p = \|Ef_{O_j}\|_{\BL{k,A_j}^p(O_j)}^p + \Dec^p,
\end{equation*}
where $x_{O_j} \in \R^n$ is a choice of translate such that $f_{O_j}$ is concentrated on wave packets $\rho_j^{-1/2+\delta_m}$-tangent to $\bZ + x_{O_j}$ in (a $\rho_j$-ball containing) $O_j$. For each $O_j \in \O_j$ either the cellular or the algebraic case holds, as defined in Theorem~\ref{partitioning theorem}. Let $\O_{j,\cell}$ denote the subcollection of $\O_j$ consisting of all cells for which the cellular case holds and $\O_{j,\alg} := \O_j \setminus \O_{j,\cell}$. Thus, by $(\mathrm{I})_j$, one may bound $ \|Ef\|_{\BL{k,A}^p(B_{r})}^p$ by
\begin{equation*}
    C_{j,\delta}^{\mathrm{I}}(d, r)  \Big[ \sum_{O_j \in \O_{j,\cell}} \|Ef_{O_j}\|_{\BL{k,A_j}^p(O_j)}^p + \sum_{O_j \in \O_{j,\alg}} \|Ef_{O_j}\|_{\BL{k,A_j}^p(O_j)}^p \Big] + \mathrm{err}(j);
\end{equation*}
the analysis splits into two cases depending on which term in this sum dominates.




\subsection*{$\blacktriangleright$ Cellular-dominant case} Suppose that the inequality
\begin{equation}\label{cellular dominant comparison}
     \sum_{O_j \in \O_{j,\alg}} \|Ef_{O_j}\|_{\BL{k,A_j}^p(O_j)}^p \leq \sum_{O_j \in \O_{j,\cell}} \|Ef_{O_j}\|_{\BL{k,A_j}^p(O_j)}^p 
\end{equation}
holds so that
\begin{equation}\label{cellular dominant case}
   \|Ef\|_{\BL{k,A}^p(B_{r})}^p \leq 2C_{j,\delta}^{\mathrm{I}}(d, r)  \sum_{O_j \in \O_{j,\cell}} \|Ef_{O_j}\|_{\BL{k,A_j}^p(O_j)}^p + \mathrm{err}(j).
\end{equation} 

\paragraph{\underline{Definition of $\sE_{j+1}$}} Define $\fh_{j+1}$ by adjoining the letter $\tc$ to the word $\fh_j$. Thus, it follows from the definitions that 
\begin{equation}\label{cellular word}
    \rho_{j+1} =  \frac{\rho_j}{2}, \quad \#_{\btc}(j+1) = \#_{\btc}(j)+1 \quad \textrm{and} \quad \#_{\bta}(j+1) = \#_{\bta}(j).
\end{equation}

The next generation of cells $\O_{j+1}$ will arise from the cellular decomposition of Theorem~\ref{partitioning theorem}. Fix $O_j \in \O_{j, \cell}$ so that there exists some polynomial $P \colon \R^n \to \R$ of degree $O(d)$ with the following properties:
\begin{enumerate}[i)]
    \item $\#\cell(P) \sim d^m$ and each $O \in \cell(P)$ has diameter at most $\rho_{j+1}$. 
    \item One may pass to a refinement of $\cell(P)$ such that if 
    \begin{equation}\label{shrunken cells recall}
        \O_{j+1}(O_j) := \Big\{\, O\setminus N_{\!\!\rho_{j}^{1/2 + \delta_m}}Z(P)\, :\, O \in \cell(P)\,\Big\}
    \end{equation}
    denotes the corresponding collection of $\rho_{j}^{1/2 + \delta_m}$-shrunken cells, then 
    \begin{equation*}
    \|Ef_{O_j}\|_{\BL{k,A_j}^p(O_{j})}^p \lesssim d^{m}\|Ef_{O_j}\|_{\BL{k,A_j}^p(O_{j+1})}^p \qquad \textrm{for all $O_{j+1} \in \O_{j+1}(O_j)$.}
\end{equation*}
    \end{enumerate}
 Given $O_{j+1} \in \O_{j+1}(O_j)$, define
 \begin{equation*}
 f_{O_{j+1}} := \sum_{\substack{(\theta,v) \in \T[\rho_j] \\ T_{\theta,v} \cap O_{j+1} \neq \emptyset}} (f_{O_j})_{\theta,v}.
\end{equation*}
It is a simple consequence of the fundamental theorem of algebra (or B\'ezout's theorem) that any tube $T_{\theta,v}$ for $(\theta,v) \in \T[\rho_j]$ can enter at most $O(d)$ cells $O_{j+1} \in \O_{j+1}(O_j)$ (it is for this reason that one works with the collection of \emph{shrunken} cells as defined in \eqref{shrunken cells recall}). Consequently, by the basic orthogonality between the wave packets,
\begin{equation}\label{cellular 1}
    \sum_{O_{j+1} \in \O_{j+1}(O_j)} \|f_{O_{j+1}} \|_{2}^2 \lesssim d \|f_{O_j} \|_{2}^2. 
\end{equation}
By the pigeonhole principle, one may therefore pass to a refinement of $\O_{j+1}(O_j)$ such that
 \begin{equation}\label{cellular 2}
    \|f_{O_{j+1}} \|_{2}^2 \lesssim d^{-(m-1)} \|f_{O_j} \|_{2}^2 \qquad \textrm{for all $O_{j+1} \in \O_{j+1}(O_j)$.}
\end{equation}
 Finally, define 
 \begin{equation*}
 \O_{j+1} := \bigcup_{O_j \in \O_{j, \cell}} \O_{j+1}(O_j).
\end{equation*}
This completes the construction of $\sE_{j+1}$ and 
it remains to check that the new ensemble satisfies the desired properties.\footnote{There is a slight technical issue here as the $f_{O_{j+1}}$ are required to satisfy the tangency hypothesis at scale $\rho_{j+1}$; this is not quite directly inherited from the parent $f_{O_j}$ functions since they only satisfy a tangency hypothesis at scale  $\rho_{j}$. Although the scales differ by only a factor of 2, the construction is applied repeatedly as part of the recursive process and therefore such factors can build up and potentially threaten the argument. One may deal with this problem by performing a further decomposition of the cells $O_{j+1}$ and functions $f_{O_{j+1}}$ using Proposition~\ref{random cover proposition}: the details are omitted since the argument is similar (but significantly simpler) to that used to treat the algebraic case below. See also Lemma 10.2 of~\cite{GHI}.} In view of this, it is useful to note that
\begin{equation}\label{cellular coefficients}
C^{N}_{j, \delta}(d,r) = d^{-\delta} C^{N}_{j+1,\delta}(d, r)\quad \textrm{for $N \in \{\mathrm{I},\mathrm{II},\mathrm{III}\}$} \quad \textrm{and}\quad A_j = A_{j+1},
\end{equation}
which follows immediately from \eqref{cellular word} and the definition of the $C^{N}_{j, \delta}(d,r)$ and $A_j$.\\



%
%
%
\paragraph{\underline{Property I}} Fixing $O_j \in \O_{j,\cell}$, observe that $\#\O_{j+1}(O_j) \sim d^m$ and
\begin{equation*}
    \|Ef_{O_j}\|_{\BL{k,A_j}^p(O_{j})}^p  \lesssim d^{m}\|Ef_{O_j}\|_{\BL{k,A_j}^p(O_{j+1})}^p \qquad \textrm{for all $O_{j+1} \in \O_{j+1}(O_j)$}
\end{equation*}
by the properties i) and ii) from the polynomial partitioning theorem and the fact that $\O_{j+1}(O_j)$ is obtained by twice refining a set of cardinality comparable to that of $\cell(P)$. Thus,
\begin{equation*}
    \|Ef_{O_j}\|_{\BL{k,A_j}^p(O_j)}^p \lesssim \sum_{O_{j+1} \in \O_{j+1}(O_j)}  \|Ef_{O_j}\|_{\BL{k,A_j}^p(O_{j+1})}^p
\end{equation*}
and, recalling \eqref{cellular dominant case} and \eqref{cellular coefficients}, one deduces that
\begin{equation*}
    \|Ef\|_{\BL{k,A}^p(B_{r})}^p \leq C d^{-\delta} C^{\mathrm{I}}_{j+1,\delta}(d, r) \sum_{O_{j+1} \in \O_{j+1}}  \|Ef_{O_j}\|_{\BL{k,A_{j+1}}^p(O_{j+1})}^p  + \mathrm{err}(j).
\end{equation*}
By the definition of $f_{O_{j+1}}$ and the spatial concentration property of the wave packets, it follows that
\begin{equation*}
    Ef_{O_j}(x) = Ef_{O_{j+1}}(x) +  \Dec \qquad \textrm{for all $x \in O_{j+1}$}.
\end{equation*}
This inequality relies on the fact that $\rho_{j+1} \gtrsim r^{\delta}$, which is valid since it is assumed that the stopping condition \texttt{[tiny]} fails. If $r$ is sufficiently large, then one concludes that
\begin{equation*}
    \|Ef\|_{\BL{k,A}^p(B_{r})}^p \leq C d^{-\delta} C^{\mathrm{I}}_{j+1,\delta}(d, r) \!\!\! \sum_{O_{j+1} \in \O_{j+1}}  \|Ef_{O_{j+1}}\|_{\BL{k,A_{j+1}}^p(O_{j+1})}^p + \mathrm{err}(j+1).
\end{equation*}
Thus, provided $d$ is chosen large enough so as to ensure that the additional $d^{-\delta}$ factor absorbs the unwanted constant $C$, one deduces $(\mathrm{I})_{j+1}$.  This should be compared with Solymosi and Tao's approach to polynomial partitioning~\cite{ST}.
\\

%
%
%
\paragraph{\underline{Property II}} By the construction, 
\begin{align*}
    \sum_{O_{j+1} \in \O_{j+1}} \| f_{O_{j+1}}\|_{2}^2 &= \sum_{O_{j} \in \O_{j,\mathrm{cell}}} \sum_{O_{j+1} \in \O_{j+1}(O_j)} \| f_{O_{j+1}}\|_{2}^2  \\
    &\lesssim d \sum_{O_j \in \O_j} \| f_{O_j}\|_{2}^2, 
\end{align*}
where the inequality follows from a term-wise application of \eqref{cellular 1}. Thus, $(\mathrm{II})_j$ and~\eqref{cellular coefficients} imply that
\begin{equation*}
    \sum_{O_{j+1} \in \O_{j+1}} \| f_{O_{j+1}}\|_{2}^2 \lesssim d^{-\delta} C^{\mathrm{II}}_{j+1,\delta}(d)  d^{\#_{\stc}(j+1)} \| f\|_{2}^2 
\end{equation*}
and, provided $d$ is chosen sufficiently large, one deduces $(\mathrm{II})_{j+1}$.\\

%
%
%
\paragraph{\underline{Property III}} Fix $O_j \in \O_{j, \mathrm{cell}}$, $O_{j+1} \in \O_{j+1}(O_j)$ and recall from \eqref{cellular 2} that
\begin{equation}\label{gain in d}
    \|f_{O_{j+1}}\|_{2}^2 \lesssim d^{-(m-1)}\|f_{O_{j}}\|_{2}^2.
\end{equation}
Thus, $(\mathrm{III})_j$ and \eqref{cellular coefficients} imply that
\begin{equation*}
    \|f_{O_{j+1}}\|_{2}^2 \lesssim d^{-\delta} C^{\mathrm{III}}_{j+1,\delta}(d, r) \Big(\frac{r}{\rho_j}\Big)^{-\frac{n-m}{2}} d^{-(\#_{\stc}(j)+1)(m-1)}  \|f\|_{2}^2.
\end{equation*}
Since $\rho_j \sim \rho_{j+1}$ and $\#_{\stc}(j) + 1 = \#_{\stc}(j+1)$, provided $d$ is chosen sufficiently large, one deduces $(\mathrm{III})_{j+1}$. 

The local inequality $(\mathrm{III}_{\mathrm{loc}})_{j+1}$ follows in a similar manner but with one key difference: the inequality \eqref{gain in d} is no longer available due to the localisation in the $L^2$-norms. Instead, one uses simple orthogonality between the wave packets to prove that
\begin{equation*}
    \max_{\theta : \rho^{-1/2}-\mathrm{cap}}\|f_{O_{j+1}}\|_{L^2_{\mathrm{avg}}(\theta)}^2 \lesssim \max_{\theta : \rho^{-1/2}-\mathrm{cap}}\|f_{O_{j}}\|_{L^2_{\mathrm{avg}}(\theta)}^2 
\end{equation*}
for $1 \le \rho \leq \rho_{j+1} \leq \rho_j$.




\subsection*{$\blacktriangleright$ Algebraic-dominant case} Suppose that the hypothesis \eqref{cellular dominant comparison} of the cellular-dominant case fails so that
\begin{equation}\label{algebraic dominant comparison}
     \sum_{O_j \in \O_{j,\cell}} \|Ef_{O_j}\|_{\BL{k,A_j}^p(O_j)}^p  < \sum_{O_j \in \O_{j,\alg}} \|Ef_{O_j}\|_{\BL{k,A_j}^p(O_j)}^p
\end{equation}
and, consequently,
\begin{equation}\label{algebraic dominant case}
   \|Ef\|_{\BL{k,A}^p(B_{r})}^p \leq 2C_{j,\delta}^{\mathrm{I}}(d, r)  \sum_{O_j \in \O_{j,\alg}} \|Ef_{O_j}\|_{\BL{k,A_j}^p(O_j)}^p + \mathrm{err}(j).
\end{equation} 
Each cell in $\O_{j,\alg}$ satisfies the condition of the algebraic case of Theorem~\ref{partitioning theorem}; this information is used to construct the $(j+1)$-generation ensemble. \\

\paragraph{\underline{Definition of $\sE_{j+1}$}} Define $\fh_{j+1}$ by adjoining the letter $\ta$ to the word $\fh_j$. Thus, it follows from the definitions that 
\begin{equation*}
    \rho_{j+1} =  \rho_j^{1-\tilde{\delta}_{m-1}}, \quad \#_{\btc}(j+1) = \#_{\btc}(j) \quad \textrm{and} \quad \#_{\bta}(j+1) = \#_{\bta}(j)+1.
\end{equation*} 
 The next generation of cells is constructed from the varieties  which arise from the algebraic case in Theorem~\ref{partitioning theorem}. Fix $O_j \in \O_{j,\alg}$ so that there exists a transverse complete intersection $\bY$ of dimension $m-1$ and $\Deg \bY \leq C_{\mathrm{alg}} d$ such that
\begin{equation*}
    \|Ef_{O_j}\|_{\BL{k,A_j}^p(O_j)}^p \lesssim \|Ef_{O_j}\|_{\BL{k,A_{j}}^p(O_j \cap N_{\!\!\rho_j^{1/2 + \delta_m}}\bY)}^p.
\end{equation*}
Let $\cB(O_j)$ be a  cover of $O_j \cap N_{\!\!\rho_j^{1/2 + \delta_m}}\bY$ by finitely-overlapping balls of radius $\rho_{j+1}$. For each $B \in \cB(O_j)$ let $\T_{B}$ denote the collection of all $(\theta,v) \in \T[\rho_j]$ for which $T_{\theta,v} \cap B \cap N_{\rho_j^{1/2 + \delta_m}}\bY  \neq \emptyset$. This set is partitioned into subsets $\T_{B,\tang}$ and $\T_{B, \trans}$ consisting of wave packets in $\T_{B}$ which are tangential and transverse to $\bY$ on $B$, respectively, by defining 
\begin{equation*}
    \T_{B, \tang} := \Big\{\, (\theta,v) \in \T_{B}\, :\, \textrm{$T_{\theta,v}$ is $\rho_{j+1}^{-1/2+\delta_{m-1}}$-tangent to $\bY$ on $B$}\, \Big\}
\end{equation*}
and $\T_{B, \trans} := \T_{B} \setminus \T_{B, \tang}$. This setup is slightly inconsistent with the definition of tangent from Definition~\ref{tangent definition} (since the wave packets in $\T_B$ are at the large scale~$\rho_j$ rather than $\rho_{j+1}$) and therefore some clarification is necessary. 

\begin{definition}\label{new tangent definition} In this context, the tangency condition means that the following conditions hold:
\begin{enumerate}[i)]
\item $T_{\theta,v} \cap 2B \subseteq N_{2\rho_{j+1}^{1/2+\delta_{m-1}}}\bY  = N_{2\rho_{j}^{1/2+\delta_{m}}}\bY$ ;
\item If $x \in T_{\theta,v}$ and $y \in \bY \cap 2B$ satisfy $|y-x| \lesssim \rho_{j+1}^{1/2 + \delta_{m-1}}= \rho_{j}^{1/2 + \delta_m}$, then
\begin{equation*}
    \angle(G(\theta), T_y\bY) \lesssim \rho_{j+1}^{-1/2+\delta_{m-1}}. 
\end{equation*}
\end{enumerate}

\end{definition}

By the basic concentration property of the wave packets, one may decompose the function $Ef_{O_j}$ on $B$ as
\begin{equation*}
    Ef_{O_j}(x) = Ef_{B, \trans}(x) + Ef_{B, \tang}(x) + \Dec \qquad \textrm{for all $x \in B$}
\end{equation*}
where 
\begin{equation*}
    f_{B, \tang} := \sum_{(\theta,v) \in \T_{B, \tang}} (f_{O_j})_{\theta,v} \qquad \textrm{and} \qquad  f_{B, \trans} := \sum_{(\theta,v) \in \T_{B, \trans}} (f_{O_j})_{\theta,v}.
\end{equation*}

 The functions $f_{B,\tang}$ are in fact concentrated on scale $\rho_{j+1}$ wave packets which are $\rho_{j+1}^{-1/2+\delta_{m-1}}$-tangent to $\bY$ in $B$ in precisely the sense of Definition~\ref{tangent definition}. This can be seen by a direct application of Lemma~\ref{relation between scales lemma}. In addition, by the local version of the basic orthogonality between wave packets,
 \begin{equation*}
     \max_{\theta : \rho^{-1/2}-\mathrm{cap}}\|f_{B, \tang}\|_{L^2_{\mathrm{avg}}(\theta)}^2 \lesssim \max_{\theta : \rho^{-1/2}-\mathrm{cap}}\|f_{O_j}\|_{L^2_{\mathrm{avg}}(\theta)}^2 
 \end{equation*}
 whenever $B \in \cB(O_j)$ and $1 \leq \rho \leq \rho_{j}$. Provided that the constant $C_{\mathrm{tang}}$ is suitably chosen, these observations imply that the functions $f_{B,\tang}$ satisfy the conditions stated in \eqref{tangent stopping conditions} from the definition of the stopping condition \texttt{[tang]}. 
 
 By hypothesis, \texttt{[tang]} fails and, consequently, one may deduce that
\begin{equation}\label{algebraic 1}
    \sum_{O_j \in \O_{j,\mathrm{alg}}}\|Ef_{O_j}\|_{\BL{k,A_j}^p(O_j)}^p \lesssim  \sum_{O_j \in \O_{j,\mathrm{alg}}} \sum_{B \in \cB(O_j)}\|Ef_{B, \trans}\|_{\BL{k,A_{j+1}}^p(B)}^p,
\end{equation}
where this inequality holds up to the inclusion of a rapidly decaying error term on the right-hand side.
Indeed, by the triangle inequality for broad norms (Lemma~\ref{triangle inequality lemma}) and since $A_{j+1} = A_j/2$, one may dominate the left-hand side of \eqref{algebraic 1} by
\begin{equation*}
 \sum_{O_j \in \O_{j,\mathrm{alg}}}\sum_{B \in \cB(O_j)} \Big[\|Ef_{B, \tang}\|_{\BL{k,A_{j+1}}^p(B)}^p + \|Ef_{B, \trans}\|_{\BL{k,A_{j+1}}^p(B)}^p\Big] +\Dec^p.
\end{equation*}
 By the preceding observations, the failure of the stopping condition \texttt{[tang]} forces
\begin{equation*}
    \sum_{O_j \in \O_{j,\mathrm{alg}}}\sum_{B \in \cB(O_j)} \|Ef_{B, \tang}\|_{\BL{k,A_{j+1}}^p(B)}^p < \frac{1}{C_{\mathrm{tang}}}  \sum_{O_j \in \O_j}\|Ef_{O_j}\|_{\BL{k,A_j}^p(O_j)}^p 
\end{equation*}
(since it has been shown that all other conditions for \texttt{[tang]} are met). Recalling~\eqref{algebraic dominant comparison}, for a suitable choice of constant $C_{\mathrm{tang}}$, this implies \eqref{algebraic 1}.

The functions $f_{B, \trans}$ and sets $B$ are further decomposed so as to ensure favourable tangency properties with respect to translates of the variety $\bZ$ at the new scale~$\rho_{j+1}$. Let $\bZ_B := \bZ + x_{O_j} - x_B$ where $x_B$ denotes the centre of $B \in \cB(O_j)$. Proposition~\ref{random cover proposition} implies that for each $B \in \cB(O_j)$ there exists a finite set of translates $\fB \subseteq B(0,\rho_j^{1/2 + \delta_m})$ such that
\begin{equation}\label{algebraic 2}
    \|Ef_{B, \trans}\|_{\BL{k,A_{j+1}}^p(B)}^p \lesssim \log^2r  \sum_{B \in \fB} \|E\tilde{f}_{B, \trans,b}\|_{\BL{k,A_{j+1}}^p(B(0,\rho_{j+1}) \cap N_{\!\!\rho_{j+1}^{1/2 + \delta_m}} (\bZ_B + b))}^p
\end{equation}
holds up to the inclusion of a rapidly decaying error term, whilst
\begin{equation}\label{algebraic 3}
    \sum_{b \in \fB} \|\tilde{f}_{B, \trans,b}\|_{2}^2 \lesssim \|f_{B, \trans}\|_{2}^2.
\end{equation}
Here the functions $\tilde{f}_{B, \trans, b}$ are defined as in Section~\ref{comparing wave packet decompositions section}. In particular, each $\tilde{f}_{B, \trans, b}$ is concentrated on wave packets which are $\rho_{j+1}^{-1/2+\delta_m}$-tangent to $\bZ_B + b$ in $B(0, \rho_{j+1})$. Finally, define
\begin{equation*}
    \O_{j+1}(O_j) := \Big\{\, B \cap N_{\!\!\rho_{j+1}^{1/2 + \delta_m}} (\bZ + x_{O_j} + b)\, :\, B \in \cB(O_j) \textrm{ and } b \in \fB \,\Big\}
\end{equation*} 
and for any $O_{j+1} = B \cap N_{\!\!\rho_{j+1}^{1/2 + \delta_m}} (\bZ + x_{O_j} + b) \in  \O_{j+1}(O_j)$ let $f_{O_{j+1}}$ satisfy $\tilde{f}_{O_{j+1}} := \tilde{f}_{B, \trans, b}$; once again, we are using the definition of the map $g \mapsto \tilde{g}$ from Section~\ref{comparing wave packet decompositions section}. Thus, each $f_{O_{j+1}}$ is concentrated on wave packets which are $\rho_{j+1}^{-1/2+\delta_m}$-tangent to $\bZ + x_{O_{j+1}}$ in $B \in \cB(O_j)$, where $x_{O_{j+1}}=x_{O_{j}}+b$. The collection of cells $\O_{j+1}$ is then given by
\begin{equation*}
    \O_{j+1} := \bigcup_{O_j \in \O_{j, \alg}} \O_{j+1}(O_j). 
\end{equation*}
It remains to verify that the ensemble $\sE_{j+1}$ satisfies the desired properties. In view of this, it is useful to note that
\begin{align}
\nonumber
C^{\mathrm{I}}_{j, \delta}(d,r) &= (\log r)^{-2(1+\delta)}C^{\mathrm{I}}_{j+1,\delta}(d, r), \\
\label{algebriac coefficients}
C^{\mathrm{II}}_{j, \delta}(d) &= d^{-n(1+\delta)}C^{\mathrm{II}}_{j+1,\delta}(d),\\
\nonumber
C^{\mathrm{III}}_{j, \delta}(d,r) &= r^{-\bar{C}\delta_m} d^{-\delta}C^{\mathrm{III}}_{j+1,\delta}(d, r),
\end{align}
which can be verified directly from the definitions.\\

%
%
%
\paragraph{\underline{Property I}} By combining \eqref{algebraic 1} and \eqref{algebraic 2} together with the various definitions one obtains 
\begin{equation*}
  \sum_{O_{j} \in \O_{j,\mathrm{alg}}} \|Ef_{O_j}\|_{\BL{k,A_j}^p(O_{j})}^p \lesssim \log^2 r\sum_{O_{j+1} \in \O_{j+1}} \|Ef_{O_{j+1}}\|_{\BL{k,A_{j+1}}^p(O_{j+1})}^p,
\end{equation*}
where this inequality holds up to the inclusion of a rapidly decaying error term on the right-hand side.
Recalling \eqref{algebraic dominant case} and \eqref{algebriac coefficients}, it follows that 
\begin{equation*}
  \|Ef\|_{\BL{k,A}^p(B_{r})}^p \leq \frac{C\cdot C^{\mathrm{I}}_{j+1,\delta}(d, r)}{(\log r)^{2\delta}}  \sum_{O_{j+1} \in \O_{j+1}}\|Ef_{O_{j+1}}\|_{\BL{k,A_{j+1}}^p(O_{j+1})}^p+\mathrm{err}(j+1).
\end{equation*}
Provided $r$ is chosen to be sufficiently large, one may absorb the unwanted constant~$C$ by the additional $(\log r)^{-2\delta}$ factor and thereby deduce $(\mathrm{I})_{j+1}$. 
\\
\paragraph{\underline{Property II}} Fix $O_j \in \O_{j,\mathrm{alg}}$ and note that
\begin{align}
\nonumber
    \sum_{O_{j+1} \in \O_{j+1}(O_j)} \|f_{O_{j+1}}\|_{2}^2 &= \sum_{B \in \cB(O_j)} \sum_{b \in \fB} \|f_{B, \trans,b}\|_{2}^2 \\
    \label{algbraic case property II}
    &\lesssim \sum_{B \in \cB(O_j)}  \|f_{B, \trans}\|_{2}^2
\end{align}
by the definition of $f_{O_{j+1}}$ and \eqref{algebraic 3}. To estimate the latter sum one exploits the transversal property of the wave packets of the $f_{B, \trans}$. The key observation is the following algebraic-geometric result of Guth, which appears in Lemma 5.7 of~\cite{Guth} and can be roughly thought of as a continuum version of the fundamental theorem of algebra (or B\'ezout's theorem). 

\begin{lemma}[\cite{Guth}]\label{transversal intersection lemma}
Suppose $T$ is an infinite cylinder in $\R^n$ of radius $\rho$ and central axis $\ell$ and $\bY$ is a transverse complete intersection. For $\alpha > 0$ let
\begin{equation*}
    \bY_{> \alpha} := \{ y \in \bY : \angle(T_y\bY, \ell) > \alpha \}.
\end{equation*}
The set $\bY_{> \alpha} \cap T$ is contained in a union of $O\big((\Deg \bY)^n\big)$ balls of radius $\rho\alpha^{-1}$.
\end{lemma}

By choosing the implicit constants correctly in Definition~\ref{new tangent definition}, a wave packet $(\theta,v) \in \T_B$ belongs to $\T_{B,\trans}$ if and only if the angle condition ii) fails to be satisfied. Indeed, if ii) holds, then since $T_{\theta,v} \cap B \cap N_{\rho_j^{1/2 + \delta_m}}\bY \neq \emptyset$ by the definition of $\T_B$, the containment property i) automatically follows and therefore $(\theta,v) \in \T_{B,\tang}$ (see, for instance, \cite[Proposition 9.2]{GHI} for details of an argument of this type). Thus, given any $(\theta,v) \in \bigcup_{B \in \cB(O_j)} \T_{B, \trans}$, it follows from the definitions that 
\begin{equation}\label{angle lower bound}
    \angle(G(\theta), T_y \bY) \gtrsim \rho_{j+1}^{-1/2 + \delta_{m-1}}
\end{equation}
for some $y \in \bY\cap 2B$ with $|y - x| \lesssim \rho_{j+1}^{1/2 + \delta_{m-1}}$ for some $x\in T_{\theta,v}$. This implies that 
\begin{equation*}
    T \cap B \cap \bY_{\gtrsim \rho_{j+1}^{-1/2 + \delta_{m-1}}} \neq \emptyset
\end{equation*}
where $T$ is the infinite cylinder that shares the core line of $T_{\theta,v}$ but has radius $\sim \rho_{j+1}^{1/2 + \delta_{m-1}}$. Observe that
\begin{equation*}
   \underbrace{\rho_{j+1}^{1/2 + \delta_{m-1}}}_{\substack{\sim \textrm{\tiny Radius} \\ \textrm{\tiny of $T$}}}  \big[\underbrace{\rho_{j+1}^{-1/2 + \delta_{m-1}}}_{\substack{\sim \textrm{\tiny Angle } \\ \textrm{\tiny from \eqref{angle lower bound}}}} \big]^{-1} = \underbrace{\rho_{j+1}}_{\substack{\textrm{\tiny Radius} \\ \textrm{\tiny of $B \in \cB(O_j)$}}}.
\end{equation*} 
Thus, by Lemma~\ref{transversal intersection lemma}, any $(\theta,v) \in \bigcup_{B \in \cB(O_j)} \T_{B, \trans}$ lies in at most $O(d^n)$ of the sets~$\T_{B, \trans}$ and, consequently, by the basic orthogonality between the wave packets,
\begin{equation*}
    \sum_{B \in \cB(O_j)}  \|f_{B, \trans}\|_{2}^2 \lesssim d^n \|f_{O_j}\|_{2}^2.
\end{equation*}
Combining this inequality with \eqref{algbraic case property II} and summing over all $O_j \in \O_{j,\mathrm{alg}}$,
\begin{equation*}
    \sum_{O_{j+1} \in \O_{j+1}}  \|f_{O_{j+1}}\|_{2}^2 \lesssim d^n \sum_{O_{j} \in \O_{j}}\|f_{O_j}\|_{2}^2.
\end{equation*}
Applying $(\mathrm{II})_j$ and \eqref{algebriac coefficients}, one concludes that
\begin{equation*}
    \sum_{O_{j+1} \in \O_{j+1}} \|f_{O_{j+1}}\|_{2}^2 \lesssim d^{-n\delta}C^{\mathrm{II}}_{j+1,\delta}(d, r) \|f\|_{2}^2.
\end{equation*}
Thus, provided $d$ is chosen sufficiently large, one deduces $(\mathrm{II})_{j+1}$. 
\\

\paragraph{\underline{Property III}} Fix $O_j \in \O_{j,\mathrm{alg}}$ and $O_{j+1} \in \O_{j+1}(O_j)$ and suppose that $\tilde{f}_{O_{j+1}} = \tilde{f}_{B, \trans, b}$. Recall that the function $f_{B, \trans}$ is concentrated on scale $\rho_{j}$ wave packets which are $\rho_{j}^{-1/2+\delta_m}$-tangent to some translate of $\bZ$ on some $\rho_{j}$-ball. It therefore follows from the transverse equidistribution estimate \eqref{transverse equidistribution inequality} of Lemma~\ref{transverse equidistribution lemma} with $\rho := 1$ that
\begin{equation*}
    \|\tilde{f}_{B, \trans, b}\|_{2}^2 \lesssim_{\Deg\bZ} r^{\bar{C}\delta_m} \Big(\frac{\rho_j}{\rho_{j-1}}\Big)^{-\frac{n-m}{2}} \|f_{B, \trans}\|_{2}^2.
\end{equation*}
On the other hand, by the basic orthogonality between the wave packets,
\begin{equation*}
   \|f_{B, \trans}\|_{2}^2 \lesssim \|f_{O_j}\|_{2}^2.
\end{equation*}
Applying $(\mathrm{III})_j$ and \eqref{algebriac coefficients}, one concludes that
\begin{equation*}
    \|f_{O_{j+1}}\|_{2}^2 \lesssim_{\Deg\bZ} d^{-\delta}C^{\mathrm{III}}_{j+1,\delta}(d, r) \Big(\frac{r}{\rho_{j+1}}\Big)^{-\frac{n-m}{2}}d^{-\#_{\stc}(j+1)(m-1)}\|f\|_{2}^2.
\end{equation*}
Thus, provided $d$ is chosen sufficiently large, one deduces $(\mathrm{III})_{j+1}$. The local version, $(\mathrm{III}_{\mathrm{loc}})_{j+1}$, follows in a similar manner, using the local transverse equidistribution estimate \eqref{transverse equidistribution inequality} for general values of $1 \leq \rho \leq \rho_{j+1}$.




\section{Proof of Theorem~\ref{main theorem}}\label{proof section}

\subsection{The second algorithm} Theorem~\ref{main theorem} is established by repeated application of the algorithm \texttt{[alg 1]} from the previous section. This process forms part of a second algorithm which is referred to as \texttt{[alg 2]} and is described presently.

Throughout this section, let $p_{\ell}$ denote Lebesgue exponents, to be fixed later, defined for $k \leq \ell \leq n$ and satisfying $$p_k \geq p_{k+1} \geq \dots \geq p_n =: p \geq 2.$$ The numbers $0 \leq \alpha_{\ell}, \beta_{\ell} \leq 1$ for $k \leq \ell \leq n$ are then defined in terms of the $p_{\ell}$ by
\begin{equation*}
    \frac{1}{p_{\ell}} =: \frac{1-\alpha_{\ell-1}}{2} + \frac{\alpha_{\ell-1}}{p_{\ell-1}} \quad \textrm{and} \quad \beta_{\ell} := \prod_{i = \ell}^{n-1} \alpha_{i} \qquad \textrm{for $k+1 \leq \ell \leq n-1$}
\end{equation*}
and $\alpha_n :=: \beta_n := 1$. Also fix $\varepsilon > 0$ and define the small parameters $\delta_{\ell}$ as in the previous section so that the inequalities in \eqref{small parameters} hold. 

There are two stages to \texttt{[alg 2]}, which can roughly be described as follows:
\begin{itemize}
    \item \textbf{The recursive stage}: $Ef$ is repeatedly decomposed into pieces with favourable tangency properties with respect to varieties of progressively lower dimension.  
    \item \textbf{The final stage}: $Ef$ is further decomposed into very small scale pieces.
\end{itemize}
To begin, the recursive stage of \texttt{[alg 2]} is described. \\

\paragraph{\underline{\texttt{Input}}}
Fix $R \gg 1$ and let  $f \colon B^{n-1} \to \C$ be smooth and bounded and, without loss of generality, assume that $f$ satisfies the \emph{non-degeneracy hypothesis} 
\begin{equation}\label{non degeneracy hypothesis}
\|Ef\|_{\BL{k,A}^{p}(B_{\>\!\!R})} \geq C_{\textrm{hyp}} R^{\varepsilon}\|f\|_2
\end{equation}
where $C_{\textrm{hyp}}$ and $A \in \N$ are constants which are chosen sufficiently large to satisfy the forthcoming requirements of the proof.\\ 

\paragraph{\underline{\texttt{Output}}} The $(n+1-\ell)$th step of the recursion will produce:
\begin{itemize}
    \item An $(n+1-\ell)$-tuple of:
    \begin{itemize}
        \item scales $\vec{r}_{\ell} = (r_n, \dots, r_{\ell})$ satisfying $R = r_n > r_{n-1} > \dots > r_{\ell}$;
        \item large and (in general) non-admissible parameters $\vec{D}_{\ell} = (D_n, \dots, D_{\ell})$;
        \item integers $\vec{A} = (A_n, \dots, A_{\ell})$ satisfying $A = A_n > A_{n-1} > \dots > A_{\ell}$.
    \end{itemize}
    Each of these $(n+1-\ell)$-tuples is formed by adjoining a component to the corresponding $(n-\ell)$-tuple from the previous stage.
    \item A family $\vec{\Sc}_{\ell}$ of $(n+1-\ell)$-tuples of transverse complete intersections $\vec{S}_{\ell} = (S_n, \dots, S_{\ell})$ satisfying $\dim S_i = i$ and $\Deg S_i = O(1)$ for $\ell \leq i \leq n$. 
    \item An assignment of a function $f_{\vec{S}_{\ell}}$ and a ball $B_{r_{\ell}}[\vec{S}_{\ell}]$ to each $\vec{S}_{\ell} \in \vec{\Sc}_{\ell}$ with the property that $f_{\vec{S}_{\ell}}$ is concentrated on scale $r_{\ell}$ wave packets which are $r_\ell^{-1/2+\delta_{\ell}}$-tangent to $S_{\ell}$ in $B_{r_{\ell}}[\vec{S}_{\ell}]$ (here $S_{\ell}$ is the final component of $\vec{S}_{\ell}$).   For notational convenience, the dependence on $\vec{S}_{\ell}$ will often be suppressed in the $B_{r_\ell}[\vec{S}_{\ell}]$ notation by simply writing $B_{r_{\ell}}$.
\end{itemize}
This data is chosen so that the following properties hold:

\begin{notation} Throughout this section a large number of harmless $R^{C\delta_0}$ factors appear in the inequalities, where $C$ is a constant depending on $p$ and $n$. By choosing $\delta_0$ sufficiently small relative to $\varepsilon$, at the end of the argument one may dominate any $R^{C\delta_0}$ by $R^{\varepsilon}$, say, which constitutes an acceptable loss in the inequality. Thus, for notational convenience, given $A, B \geq 0$ let $A \lessapprox B$ or $B \gtrapprox A$ denote $A \lesssim R^{C\delta_0} B$.
\end{notation}

\paragraph{\underline{Property 1}} The inequality
\begin{equation}\label{l step}
      \|Ef\|_{\BL{k,A}^p(B_{\>\!\!R})} \lessapprox M(\vec{r}_{\ell}, \vec{D}_{\ell}) \|f\|_2^{1-\beta_{\ell}} \Big( \sum_{\vec{S}_{\ell} \in \vec{\Sc}_{\ell}} \|Ef_{\vec{S}_{\ell}}\|_{\BL{k,A_{\ell}}^{p_{\ell}}(B_{r_{\ell}})}^{p_{\ell}}\Big)^{\frac{\beta_{\ell}}{p_{\ell}}}
\end{equation}
holds for 
\begin{equation*}
    M(\vec{r}_{\ell}, \vec{D}_{\ell}) := \Big(\prod_{i=\ell}^{n-1}D_i\Big)^{(n-\ell)\delta}\Big(\prod_{i=\ell}^{n-1} r_i^{\frac{1}{2}(\beta_{i+1}-\beta_i)}D_i^{\frac{1}{2}(\beta_{i+1} - \beta_{\ell})}\Big).
\end{equation*}
\paragraph{\underline{Property 2}}  For $\ell \leq n-1$ the inequality
\begin{equation*}
    \sum_{\vec{S}_{\ell} \in \vec{\Sc}_{\ell}} \|f_{\vec{S}_{\ell}}\|_2^2 \lessapprox D_{\ell}^{1 + \delta} \sum_{\vec{S}_{\ell+1} \in \vec{\Sc}_{\ell+1}} \|f_{\vec{S}_{\ell+1}}\|_2^2 
\end{equation*}
holds.\\

\paragraph{\underline{Property 3}}    For $\ell \leq n-1$ the inequalities
\begin{equation*}
    \max_{\vec{S}_{\ell} \in \vec{\Sc}_{\ell}} \|f_{\vec{S}_{\ell}}\|_2^2 \lessapprox\Big(\frac{r_{\ell+1}}{r_{\ell}}\Big)^{-\frac{n-\ell-1}{2}}  D_{\ell}^{-\ell + \delta} \max_{\vec{S}_{\ell+1} \in \vec{\Sc}_{\ell+1}} \|f_{\vec{S}_{\ell+1}}\|_2^2 
\end{equation*}
and
\begin{equation*}
    \max_{\substack{\vec{S}_{\ell} \in \vec{\Sc}_{\ell} \\ \theta : \rho^{-1/2}-\mathrm{cap}}} \|f_{\vec{S}_{\ell}}\|_{L^2_{\mathrm{avg}}(\theta)}^2 \lessapprox \Big(\frac{r_{\ell+1}}{r_{\ell}}\Big)^{-\frac{n-\ell-1}{2}}D_{\ell}^{\delta}  \max_{\substack{\vec{S}_{\ell+1} \in \vec{\Sc}_{\ell+1} \\ \theta : \rho^{-1/2}-\mathrm{cap}}} \|f_{\vec{S}_{\ell+1}}\|_{L^2_{\mathrm{avg}}(\theta)}^2 
\end{equation*}
hold for $1 \leq \rho \leq r_{\ell}$.\\

\paragraph{\underline{\texttt{First step}}} Vacuously, the function $f$ is concentrated on scale $R$ wave packets which are $R^{-1/2+\delta_n}$-tangent to the $n$-dimensional variety $\R^n$ on $B_{\>\!\!R}$. Thus, one may define
\begin{itemize}
        \item $r_n := R$; $D_n := 1$ and $A_n:= A$.
        \item $\Sc_n$ is the collection consisting of a single 1-tuple $\vec{S}_n = (S_n)$ where $S_n := \R^n$.
        \item $f_{\vec{S}_n} := f$ and $B_{r_n}[\vec{S}_n] := B_{\>\!\!R}$. 
\end{itemize}
With these definitions, all the desired properties vacuously hold.\\

\paragraph{\underline{\texttt{($n+2-\ell$)th step}}} Let $\ell \geq 1$ and suppose that the recursive algorithm has ran through $n+1-\ell$ steps. Since each function $f_{\vec{S}_{\ell}}$ is concentrated on wave packets $r_{\ell}^{-1/2+\delta_{\ell}}$-tangent to $S_{\ell}$ on $B_{r_{\ell}}[S_{\ell}]$, one may apply \texttt{[alg 1]} to bound the $k$-broad norm $\|Ef_{\vec{S}_{\ell}}\|_{\BL{k,A_{\ell}}^{p_{\ell}}(B_{r_{\ell}})}$. One of two things can happen: either \texttt{[alg 1]} terminates due to the stopping condition \texttt{[tiny]} or it terminates due to the stopping condition \texttt{[tang]}. The current recursive process terminates if the contributions from terms of the former type dominate:\\

\paragraph{\underline{\texttt{Stopping condition}}} The recursive stage of \texttt{[alg 2]} has a single stopping condition, which is denoted by \texttt{[tiny-dom]}. 
\begin{itemize}
    \item[\texttt{Stop:[tiny-dom]}] Suppose that the inequality
  \begin{equation}\label{tiny case 1}
      \sum_{\vec{S}_{\ell} \in \vec{\Sc}_{\ell}} \|Ef_{\vec{S}_{\ell}}\|_{\BL{k,A_{\ell}}^{p_{\ell}}(B_{r_{\ell}})}^{p_{\ell}} \leq\, 2\!\!\!  \sum_{\vec{S}_{\ell} \in \vec{\Sc}_{\ell,\textrm{\texttt{tiny}}}} \|Ef_{\vec{S}_\ell}\|_{\BL{k,A_{\ell}}^{p_{\ell}}(B_{r_{\ell}})}^{p_{\ell}}
\end{equation}
holds, where the right-hand summation is restricted to those $S_{\ell} \in \vec{\Sc}_{\ell}$ for which \texttt{[alg 1]} terminates owing to the stopping condition \texttt{[tiny]}. Then \texttt{[alg 2]} terminates.
\end{itemize}

Assume that the condition \texttt{[tiny-dom]} is not met. Necessarily,
  \begin{equation}\label{tangent case 1}
      \sum_{\vec{S}_{\ell} \in \vec{\Sc}_{\ell}} \|Ef_{\vec{S}_{\ell}}\|_{\BL{k,A_{\ell}}^{p_{\ell}}(B_{r_{\ell}})}^{p_{\ell}} \leq\, 2\!\!\! \sum_{\vec{S}_{\ell} \in \vec{\Sc}_{\ell,\textrm{\texttt{tang}}}} \|Ef_{\vec{S}_{\ell}}\|_{\BL{k,A_{\ell}}^{p_{\ell}}(B_{r_{\ell}})}^{p_{\ell}},
\end{equation}
where the right-hand summation is restricted to those $S_{\ell} \in \vec{\Sc}_{\ell}$ for which \texttt{[alg 1]} does not terminate owing to \texttt{[tiny]} and therefore terminates owing to \texttt{[tang]}. Consequently, for each $\vec{S}_{\ell} \in \vec{\Sc}_{\ell,\textrm{\texttt{tang}}}$ the inequalities
  \begin{equation}\label{recursive step property 1}
      \|Ef_{\vec{S}_{\ell}}\|_{\BL{k,A_{\ell}}^{p_{\ell}}(B_{r_{\ell}})}^{p_{\ell}} \lessapprox\, D_{\ell-1}^{\delta}\!\!\!\!\! \sum_{S_{\ell-1} \in \Sc_{\ell-1}[\vec{S}_{\ell}]} \|Ef_{\vec{S}_{\ell-1}}\|_{\BL{k,2A_{\ell-1}}^{p_{\ell}}(B_{r_{\ell-1}})}^{p_{\ell}},
\end{equation}
and
\begin{align}
\label{recursive step property 2}
    \sum_{S_{\ell-1} \in \Sc_{\ell-1}[\vec{S}_{\ell}]} \|f_{\vec{S}_{\ell-1}}\|_2^2 &\lessapprox\, D_{\ell-1}^{1+ \delta}\|f_{\vec{S}_{\ell}}\|_2^2; \\
\label{recursive step property 3}
\max_{S_{\ell-1} \in \Sc_{\ell-1}[\vec{S}_{\ell}]} \|f_{\vec{S}_{\ell-1}}\|_2^2 &\lessapprox \Big(\frac{r_{\ell}}{r_{\ell-1}}\Big)^{-\frac{n-\ell}{2}} D_{\ell-1}^{-(\ell-1)+\delta}\|f_{\vec{S}_{\ell}}\|_2^2;\\
\label{recursive step property 3loc}
\max_{\substack{S_{\ell-1} \in \Sc_{\ell-1}[\vec{S}_{\ell}]\\ \theta : \rho^{-1/2}-\mathrm{cap}}} \|f_{\vec{S}_{\ell-1}}\|_{L^2_{\mathrm{avg}}(\theta_{\ell})}^2 &\lessapprox \Big(\frac{r_{\ell}}{r_{\ell-1}}\Big)^{-\frac{n-\ell}{2}}D_{\ell-1}^{\delta} \max_{ \theta: \rho^{-1/2}-\mathrm{cap}}\|f_{\vec{S}_{\ell}}\|_{L^2_{\mathrm{avg}}(\theta)}^2
\end{align}
hold for $1 \leq \rho \leq r_{\ell-1}$ for some choice of:
\begin{itemize}
    \item Scale $R^{\delta} < r_{\ell-1} < r_{\ell}$, an (in general) non-admissible number $D_{\ell-1}$ and some large integer $A_{\ell-1}$ satisfying $A_{\ell-1} \sim A_{\ell}$;
    \item Family $\Sc_{\ell-1}[\vec{S}_{\ell-1}]$ of $(\ell-1)$-dimensional transverse complete intersections of degree $O(1)$;
    \item Assignment $B_{r_{\ell-1}}[\vec{S}_{\ell-1}]$ of an $r_{\ell-1}$-ball to every $S_{\ell-1} \in \Sc_{\ell-1}[\vec{S}_{\ell}]$; 
    \item Assignment $f_{\vec{S}_{\ell-1}} = (f_{\vec{S}_{\ell}})_{S_{\ell-1}}$ of a function to every $S_{\ell-1} \in \Sc_{\ell-1}[\vec{S}_{\ell}]$ which is concentrated on wave packets which are $r_{\ell-1}^{-1/2+\delta_{\ell-1}}$-tangent to $S_{\ell-1}$ on $B_{r_{\ell-1}}[S_{\ell-1}]$.
\end{itemize}
Each inequality \eqref{recursive step property 1}, \eqref{recursive step property 2}, \eqref{recursive step property 3} and \eqref{recursive step property 3loc} is obtained by combining the definition of the stopping condition \texttt{[tang]} with Properties I, II and both the global and local variants of Property III from \texttt{[alg 1]}, respectively.\footnote{Here the `error terms' $\mathrm{err}(j) := jr^{-N}\|f\|_2^p$ in Property I of \texttt{[alg 1]} can be ignored owing to the non-degeneracy hypothesis \eqref{non degeneracy hypothesis}.} Indeed, using the notation from \texttt{[alg 1]}, we take $$r := r_{\ell},\quad r_{\ell-1} := \rho_{\!J}^{1-\tilde{\delta}_{\ell-1}},\quad \text{and}\quad D_{\ell-1} := d^{\#_{\stc}(J)}.$$
Note that the $R^{O(\delta_0)}D_{\ell-1}^{\delta}$ factors arise in the above inequalities owing to~\eqref{small coefficients}.

The $r_{\ell-1}$, $D_{\ell-1}$ and $A_{\ell-1}$ can depend on the choice of $\vec{S}_{\ell}$, but this dependence can be essentially removed by pigeonholing. Indeed, recalling that $\#_{\stc}(J)= O(\log R)$, one may find a subset of the~$\Sc_{\ell,\textrm{\texttt{tang}}}$ over which the $D_{\ell-1}$ all have a common value and, moreover, the inequality \eqref{tiny case 1} still holds except that the constant $\frac{1}{2}$ is now replaced with, say,~$R^{\delta_0}$. A brief inspection of \texttt{[alg 1]} shows that, once we have pigeonholed in the parameter $N$ above, both $r_{\ell-1}$ and $A_{\ell-1}$ immediately inherit the desired uniformity.  

Letting $\vec{\Sc}_{\ell-1}$ denote the structured set 
\begin{equation*}
    \vec{\Sc}_{\ell-1} := \Big\{\, (\vec{S}_{\ell}, S_{\ell-1})\, :\, \vec{S}_{\ell} \in \vec{\Sc}_{\ell,\textrm{\texttt{tang}}} \textrm{ and } S_{\ell-1} \in \Sc_{\ell-1}[\vec{S}_{\ell}] \,\Big\},
\end{equation*}
where $\vec{\Sc}_{\ell,\textrm{\texttt{tang}}}$ is understood to be the refined collection described in the previous paragraph, it remains to verify the desired properties for the newly constructed data. Property 2 follows immediately from \eqref{recursive step property 2} and Property 3 from \eqref{recursive step property 3} and \eqref{recursive step property 3loc}, so it remains only to verify Property 1. 

By combining the inequality \eqref{l step} from the previous stage of the algorithm with~\eqref{tangent case 1} and \eqref{recursive step property 1}, one deduces that
   \begin{equation*}
      \|Ef\|_{\BL{k,A}^p(B_{\>\!\!R})} \lessapprox D_{\ell-1}^{\delta} M(\vec{r}_{\ell}, \vec{D}_{\ell}) \|f\|_2^{1-\beta_{\ell}} \Big(\!\!\sum_{\vec{S}_{\ell-1}\in \vec{\Sc}_{\ell-1}}\!\|Ef_{\vec{S}_{\ell-1}}\|_{\BL{k,2A_{\ell-1}}^{p_{\ell}}(B_{r_{\ell-1}})}^{p_{\ell}}\Big)^{\frac{\beta_{\ell}}{p_{\ell}}}.
\end{equation*}
Writing
\begin{equation*}
    \Big(\sum_{\vec{S}_{\ell-1}\in \vec{\Sc}_{\ell-1}}\|Ef_{\vec{S}_{\ell-1}}\|_{\BL{k,2A_{\ell-1}}^{p_{\ell}}(B_{r_{\ell-1}})}^{p_{\ell}}\Big)^{\frac{1}{p_{\ell}}} =  \Big\| \|Ef_{\vec{S}_{\ell-1}}\|_{\BL{k,2A_{\ell-1}}^{p_{\ell}}(B_{r_{\ell-1}})}\Big\|_{\ell^{p_{\ell}}(\vec{\Sc}_{\ell-1})},
\end{equation*}
one may apply the logarithmic convexity inequality from Lemma~\ref{logarithmic convexity inequality lemma} to dominate this expression by
\begin{equation*}
  \Big\| \|Ef_{\vec{S}_{\ell-1}}\|_{\BL{k,A_{\ell-1}}^{2}(B_{r_{\ell-1}})}\Big\|_{\ell^2(\vec{\Sc}_{\ell-1})}^{1-\alpha_{\ell-1}}  \Big\| \|Ef_{\vec{S}_{\ell-1}}\|_{\BL{k,A_{\ell-1}}^{p_{\ell-1}}(B_{r_{\ell-1}})}\Big\|_{\ell^{p_{\ell-1}}(\vec{\Sc}_{\ell-1})}^{\alpha_{\ell-1}}.
\end{equation*}
By the standard $L^2$ estimate \eqref{broad energy identity} applied to broad norms, 
\begin{equation*}
 \Big\| \|Ef_{\vec{S}_{\ell-1}}\|_{\BL{k,A_{\ell-1}}^{2}(B_{r_{\ell-1}})}\Big\|_{\ell^2(\vec{\Sc}_{\ell-1})} \lesssim r_{\ell-1}^{1/2}\Big( \sum_{\vec{S}_{\ell-1} \in \vec{\Sc}_{\ell-1}}  \|f_{\vec{S}_{\ell-1}}\|_2^2 \Big)^{1/2}
\end{equation*}
and, by Property 2 for the the families of functions $(f_{\vec{S}_{i}})_{\vec{S}_i \in \vec{\Sc}_i}$ for $\ell - 1 \leq i \leq n-1$, it follows that
\begin{equation*}
   \Big\| \|Ef_{\vec{S}_{\ell-1}}\|_{\BL{k,A_{\ell-1}}^{2}(B_{r_{\ell-1}})}\Big\|_{\ell^2(\vec{\Sc}_{\ell-1})} \lessapprox \Big(r_{\ell-1} \prod_{i=\ell-1}^{n-1}D_{i}^{1+\delta}\Big)^{1/2} \|f\|_2.
\end{equation*}
One may readily verify that
\begin{equation*}
    D_{\ell-1}^{\delta}\cdot M(\vec{r}_{\ell}, \vec{D}_{\ell})\cdot  \Big(r_{\ell-1} \prod_{i=\ell-1}^{n-1}D_{i}^{1+\delta}\Big)^{\frac{1}{2}(1-\alpha_{\ell-1})\beta_{\ell}} \leq M(\vec{r}_{\ell-1}, \vec{D}_{\ell-1})
\end{equation*}
and so combining the above estimates yields
   \begin{equation*}
      \|Ef\|_{\BL{k,A}^p(B_{\>\!\!R})} \lessapprox M(\vec{r}_{\ell-1}, \vec{D}_{\ell-1}) \|f\|_2^{1-\beta_{\ell-1}} \Big\| \|Ef_{\vec{S}_{\ell-1}}\|_{\BL{k,A_{\ell-1}}^{p_{\ell-1}}(B_{r_{\ell-1}})}\Big\|_{\ell^{p_{\ell-1}}(\vec{\Sc}_{\ell-1})}^{\beta_{\ell-1}},
\end{equation*}
which is Property 1 in this case.
\\ 

\paragraph{\underline{\texttt{The final stage}}} If the algorithm has not stopped by the $k$th step, then it necessarily terminates at the $k$th step. Indeed, otherwise \eqref{l step} would hold for $\ell = k-1$ and functions $f_{\vec{S}_{{k-1}}}$ concentrated on wave packets which are tangent to some transverse complete intersection of dimension $k-1$. By the vanishing property of the $k$-broad norms as described in Lemma~\ref{key k-broad lemma}, one would then have
\begin{equation*}
    \|Ef_{\vec{S}_{k-1}}\|_{\BL{k,A_{k-1}}^{p_{k-1}}(B_{r_{k-1}})} = \mathrm{RapDec}(R)\|f_{\vec{S}_{k-1}}\|_2
\end{equation*}
and it would easily follow from \eqref{l step} that $\|Ef\|_{\BL{k,A}^p(B_{\>\!\!R})} = \mathrm{RapDec}(R)\|f\|_2$. If $R$ is sufficiently large, then this would contradict the non-degeneracy hypothesis \eqref{non degeneracy hypothesis}. 

Suppose the recursive process terminates at step $m$, so that $m \geq k$. For each $\vec{S}_m \in \vec{\Sc}_{m,\textrm{\texttt{tiny}}}$ let $\O[\vec{S}_m]$ denote the final collection of cells output by \texttt{[alg 1]} (that is, the collection denoted by $\O_J$ in the notation of Section~\ref{structure lemma section}) when applied to estimate the broad norm $\|Ef_{\vec{S}_m}\|_{\BL{k,A_m}^{p_m}(B_{r_m})}$. Each $O \in \O[\vec{S}_m]$ has diameter at most $R^{\delta_0}$ by the definition of the stopping condition \texttt{[tiny]}. By Properties I, II and III of \texttt{[alg 1]} one has 
   \begin{equation*}
      \|Ef_{\vec{S}_m}\|_{\BL{k,A_m}^{p_m}(B_{r_m})}^{p_m} \lessapprox\ D_{m-1}^{\delta}\! \sum_{O \in \O[\vec{S}_m]} \|Ef_O \|_{\BL{k,A_{m-1}}^{p_m}(O)}^{p_m},
\end{equation*}
for some $A_{m-1} \sim A_m$ where the functions $f_O$ satisfy
\begin{equation}\label{tiny property 2}
    \sum_{O \in \O[\vec{S}_m]} \|f_O\|_2^2 \lessapprox D_{m-1}^{1+\delta}\|f_{\vec{S}_m}\|_2^2
    \end{equation}
and
\begin{equation}\label{tiny property 3}
    \max_{O \in \O[\vec{S}_m]} \|f_O\|_2^2 \lessapprox \Big(\frac{r_m}{r_{m-1}}\Big)^{-\frac{n-m}{2}}D_{m-1}^{-(m-1)+\delta}\|f_{\vec{S}_m}\|_2^2
\end{equation}
for $D_{m-1}$ a large non-admissible parameter. In particular, $D_{m-1} := d^{\#_{\stc}(J)}$ where~$J$ is the stopping time for this final application of \texttt{[alg 1]}. Once again, by pigeonholing, one may pass to a subcollection of $\Sc_{m, \textrm{\texttt{tiny}}}$ and thereby assume that the $D_{m-1}$ (and also the $A_{m-1}$) all share a common value. 

If $\O$ denotes the union of the $\O[\vec{S}_m]$ over all $\vec{S}_m$ belonging to subcollection of $\Sc_{m, \textrm{\texttt{tiny}}}$ described above, then 
   \begin{equation}\label{final decomposition}
      \|Ef\|_{\BL{k,A}^p(B_{\>\!\!R})} \lessapprox D_{m-1}^{\delta} M(\vec{r}_m, \vec{D}_m) \|f\|_2^{1-\beta_m} \Big( \sum_{O \in \O} \|Ef_{O}\|_{\BL{k,A_{m-1}}^{p_m}(O)}^{p_m}\Big)^{\frac{\beta_m}{p_m}}\!\!\!.
\end{equation}
This concludes the description of \texttt{[alg 2]}.

\subsection{Applying \texttt{[alg 2]} to prove $k$-broad estimates}\label{applying the algorithm section} Having arrived at the final decomposition of the broad norm given by \eqref{final decomposition}, the task is now to apply the properties guaranteed by the algorithm in order to prove the desired estimates. In particular, one wishes to show that the quantity 
\begin{equation*}
    M(\vec{r}_m, \vec{D}_m) \Big( \sum_{O \in \O} \|Ef_{O}\|_{\BL{k,A_{m-1}}^{p_m}(O)}^{p_m}\Big)^{\frac{\beta_m}{p_m}}
\end{equation*}
featured in \eqref{final decomposition} can be effectively bounded, provided that the exponents $p_k, \dots, p_n$ are suitably chosen. 
Since each $O \in \O$ has diameter at most $R^{\delta_0}$, trivially one may bound
\begin{equation*}
    \|Ef_{O}\|_{\BL{k,A_{m-1}}^{p_m}(O)} \lessapprox\|f_O\|_2
\end{equation*}
and, thus, it follows that
\begin{equation*}
\Big( \sum_{O \in \O} \|Ef_{O}\|_{\BL{k,A_{m-1}}^{p_m}(O)}^{p_m}\Big)^{\frac{\beta_m}{p_m}} \lessapprox \Big( \sum_{O \in \O} \|f_{O}\|_{2}^{2}\Big)^{\frac{\beta_m}{p_m}}\!\!\max_{O \in \O} \|f_O\|_2^{2(\frac{1}{2}-\frac{1}{p_m})\beta_m}.
\end{equation*}
The definition of the $\beta_{m}$ ensures that
\begin{equation*}
    \Big(\frac{1}{2}-\frac{1}{p_{m}}\Big)\beta_{m} = \frac{1}{2}-\frac{1}{p_n}
\end{equation*}
whilst \eqref{tiny property 2} and repeated application of Property 2 from \texttt{[alg 2]} imply that
\begin{equation*}
 \sum_{O \in \O} \|f_{O}\|_{2}^{2} \lessapprox \Big(\prod_{i = m-1}^{n-1} D_{i}^{1 + \delta}\Big)\|f\|_{2}^{2}.  
\end{equation*}
 Combining this estimate with \eqref{final decomposition} and the definition of $M(\vec{r}_m, \vec{D}_m)$, one concludes that
  \begin{equation}\label{reduction to max bound}
      \|Ef\|_{\BL{k,A}^{p.}(B_{\>\!\!R})} \lessapprox \prod_{i = m-1}^{n-1} r_{i}^{\frac{\beta_{i+1}-\beta_{i}}{2}}D_{i}^{\frac{\beta_{i+1}}{2} - (\frac{1}{2}-\frac{1}{p_n})+O(\delta)}\|f\|_{2}^{\frac{2}{p_n}}\max_{O \in \O} \|f_O\|_2^{1-\frac{2}{p_n}}
\end{equation}
where $r_{m-1} := 1$. The problem is now to bound the maximum appearing on the right-hand side of this expression. 

By \eqref{tiny property 3} and repeated application of Property 3 of \texttt{[alg 2]}, it follows that for any $m \leq \ell \leq n$ the inequality
\begin{align}
\nonumber
 \max_{O \in \O}\|f_{O}\|_{2}^2 &\lessapprox\prod_{i=m-1}^{\ell-1}\Big(\frac{r_{i+1}}{r_{i}}\Big)^{-\frac{n-i-1}{2}} D_{i}^{-i + \delta} \max_{\vec{S}_{\ell} \in \vec{\Sc}_{\ell}} \|f_{\vec{S}_{\ell}}\|_{2}^2\\
  \label{truncated estimate}
 &= r_\ell^{-\frac{n-\ell}{2}}\prod_{i=m-1}^{\ell-1}r_{i}^{-1/2} D_{i}^{-i + \delta} \max_{\vec{S}_{\ell} \in \vec{\Sc}_{\ell}} \|f_{\vec{S}_{\ell}}\|_{2}^2
\end{align}
holds. This bound will be exploited in different ways.




\subsection{Guth's estimate revisited}\label{Guth revisited section} As a warm up exercise for the more involved computation to follow, here Guth's $k$-broad estimate from~\cite{Guth} is recovered using the above inequalities. In particular, taking $\ell=n$, the inequality \eqref{truncated estimate} simplifies to give:\\

\begin{mdframed}[style=MyFrame]
\begin{key estimate}
\begin{equation*}
   \max_{O \in \O}\|f_{O}\|_{2}^2 \lessapprox \prod_{i=m-1}^{n-1}r_{i}^{-1/2}D_{i}^{-i + \delta}  \|f\|_{2}^2.
\end{equation*}
\end{key estimate}
\end{mdframed}
Combining this with \eqref{reduction to max bound}, one concludes that    
\begin{equation}\label{Guth inequality}
 \|Ef\|_{\BL{k,A}^{p}(B_{\>\!\!R})} \lessapprox \prod_{i=m-1}^{n-1} r_{i}^{X_{i}} D_{i}^{Y_{i} + O(\delta)}\|f\|_{2}
\end{equation}
where
\begin{equation*}
    X_{i} := \frac{\beta_{i+1}-\beta_{i}}{2}- \frac{1}{2}\Big(\frac{1}{2} - \frac{1}{p_n}\Big); \qquad
    Y_{i} := \frac{\beta_{i+1}}{2}-(i+1)\Big(\frac{1}{2}-\frac{1}{p_n}\Big).
\end{equation*}

In order to ensure that there is an acceptable dependence on $R$ in \eqref{Guth inequality}, the parameters must be chosen so as to ensure that $X_{i}, Y_{i} \leq 0$ for $m \leq i \leq n-1$ and $Y_{m-1}\leq 0$. 

\begin{remark} Ostensibly, the above conditions on the $Y_{i}$ do not take into account the additional $O(\delta)$-powers of the $D_{i}$ in \eqref{Guth inequality}. By perturbing the exponents which result under these conditions and choosing $\delta$ sufficiently small depending on the choice of perturbation, the $O(\delta)$-powers may nevertheless be safely handled. This perturbative argument yields an open range of $k$-broad estimates, which can be trivially extended to a closed range via interpolation through logarithmic convexity (the interpolation argument relies on the fact that one is permitted an $R^{\varepsilon}$-loss in the constants in the $k$-broad inequalities).
\end{remark}

Recalling from the definitions that
\begin{equation*}
    \beta_i = \Big(\frac{1}{2} - \frac{1}{p_{n}}\Big)\Big(\frac{1}{2} - \frac{1}{p_{i}}\Big)^{-1},
\end{equation*}
the condition $X_{i} \leq 0$ is equivalent to 
\begin{equation}\label{Guth r constraint}
\Big(\frac{1}{2} - \frac{1}{p_{i + 1}}\Big)^{-1}-\Big(\frac{1}{2} - \frac{1}{p_{i}}\Big)^{-1}\le 1
\end{equation}
whilst the condition $Y_{i-1} \leq 0$ is equivalent to
\begin{equation}\label{Guth D constraint}
    \Big(\frac{1}{2}-\frac{1}{p_{i}}\Big)^{-1}-2i \leq 0 . 
\end{equation}
Choose $p_m := \frac{2m}{m -1}$ so that the exponent satisfies $(\frac{1}{2} - \frac{1}{p_m})^{-1}=2m$ and therefore~\eqref{Guth D constraint} is saturated in the $i = m$ case. The remaining $p_{i}$ are then chosen so as to satisfy 
\begin{equation*}
    \Big(\frac{1}{2} - \frac{1}{p_{i}}\Big)^{-1}= m + i  
\end{equation*}
so that \eqref{Guth r constraint} is saturated for every value of $i$. With this choice, \eqref{Guth D constraint} automatically holds for all the remaining indices $m+1 \leq i \leq n$. The worst situation occurs when $m=k$, in which case one deduces that the inequality 
\begin{equation}\label{Guth estimate}
    \|Ef\|_{\BL{k,A}^p(B_{\>\!\!R})} \lesssim R^{\varepsilon}\|f\|_2
\end{equation}
holds for all $p \geq 2 + \tfrac{4}{n+k-2}$; this exponent agrees with that featured in \cite[Proposition 8.1]{Guth} and, indeed, the above argument is simply a reformulation of the proof appearing in~\cite{Guth}. 




\subsection{Improvement using the polynomial Wolff axioms} To prove Theorem~\ref{main theorem}, the argument of the previous subsection is augmented with the bounds coming from the polynomial Wolff axiom theorem. This follows the strategy of~\cite{Guth2016}, which established the $n=3$ case of the theorem. The goal is to improve the range of $p$ at the expense of weakening the $L^2$-type estimate \eqref{Guth estimate} to an $L^{\infty}$-type estimate
\begin{equation*}
    \|Ef\|_{\BL{k,A}^p(B_{\>\!\!R})} \lesssim R^{\varepsilon}\|f\|_{\infty}.
\end{equation*}
One key observation is that the choice of exponents in the previous subsection does not saturate the constraint \eqref{Guth D constraint} coming from the $D_{i}$ exponents for $m \leq i \leq n-1$. This provides some leeway, and the polynomial Wolff axiom theorem allows one to trade an acceptable loss in the $D_{i}$ exponents for a gain in the $r_{i}$ exponents, and thereby leads to an improvement in the $p$ range.

Fix $m \leq \ell \leq n$ and apply Lemma~\ref{nesting lemma} to deduce that 
\begin{equation}\label{1st}
   \max_{\vec{S}_\ell \in \vec{\Sc}_\ell} \|f_{\vec{S}_{\ell}}\|_2^2 \lessapprox r_{\ell}^{-\frac{n-\ell}{2}}\!\!\!\!\!\max_{\substack{ \vec{S}_{\ell} \in \vec{\Sc}_{\ell} \\ \theta_{\ell} : r_{\ell}^{-1/2}-\mathrm{cap}}} \|f_{\vec{S}_{\ell}}\|_{L^2_{\mathrm{avg}}(\theta_{\ell})}^2.
\end{equation}
Let $\ell \leq i \leq n-1$ and note that, by Property 3 of \texttt{[alg 2]}, 
\begin{equation}\label{2nd}
   \max_{\substack{ \vec{S}_{i} \in \vec{\Sc}_{i} \\ \theta_{i} : r_{i}^{-1/2}-\mathrm{cap}}} \|f_{\vec{S}_{i}}\|_{L^2_{\mathrm{avg}}(\theta_{i})}^2 \lessapprox \Big(\frac{r_{i+1}}{r_{i}}\Big)^{-\frac{n-i-1}{2}}  D_{i}^{\delta}\max_{\substack{ \vec{S}_{i+1} \in \vec{\Sc}_{i+1} \\ \theta_{i+1} : r_{i+1}^{-1/2}-\mathrm{cap}}}\|f_{\vec{S}_{i+1}}\|_{L^2_{\mathrm{avg}}(\theta_{i+1})}^2
\end{equation}
Combining \eqref{1st} with $n-\ell$ applications of \eqref{2nd}, we obtain
\begin{align*}
 \max_{\vec{S}_\ell \in \vec{\Sc}_\ell} \|f_{\vec{S}_{\ell}}\|_2^2 &\lessapprox r_{\ell}^{-\frac{n-\ell}{2}} \prod_{i=\ell}^{n-1}\Big(\frac{r_{i+1}}{r_{i}}\Big)^{-\frac{n-i-1}{2}}D_{i}^{\delta} \max_{ \theta : R^{-1/2}-\mathrm{cap}} \|f\|_{L^2_{\mathrm{avg}}(\theta)}^2 \\
 &\leq  \Big( \prod_{i=\ell}^{n-1}r_{i}^{-1/2}\Big) \Big(\prod_{i=\ell}^{n-1}D_{i}^{\delta}\Big) \|f\|_{\infty}^2.
\end{align*}
Substituting this estimate into \eqref{truncated estimate}, one concludes that
\begin{equation*}
    \max_{O \in \O}\|f_{O}\|_{2}^2 \lessapprox r_\ell^{-\frac{n-\ell}{2}}\Big(\prod_{i=m}^{n-1}r_{i}^{-1/2} D_{i}^{\delta}\Big)\Big( \prod_{i = m-1}^{\ell-1}D_{i}^{-i}\Big)  \|f\|_{\infty}^2
\end{equation*}
for all $m \leq \ell \leq n$. Finally, these $n-m+1$ different estimates are combined into a single inequality by taking a weighted geometric mean, yielding:\\

\begin{mdframed}[style=MyFrame]
\begin{key estimate} Let $0 \leq \gamma_m, \dots, \gamma_n \leq 1$ satisfy $\sum_{j = m}^{n} \gamma_{j} = 1$. Then
\begin{equation*}
    \max_{O \in \O}\|f_{O}\|_{2}^2 \lessapprox \prod_{i=m-1}^{n-1}r_{i}^{-\frac{1+(n-i)\gamma_{i}}{2}} D_{i}^{-i(1-\sum_{j=m}^{i} \gamma_j) + O(\delta)}  \|f\|_{\infty}^2.
\end{equation*}
\end{key estimate}
\end{mdframed}

Thus, the parameters $\gamma_j$ allow a loss in the $D_{i}$ exponents to be traded for a gain in the $r_{i}$ exponents.

The key estimate may be combined with the inequality \eqref{reduction to max bound} from Section~\ref{applying the algorithm section} to yield the bound
  \begin{equation*}
      \|Ef\|_{\BL{k,A}^{p_n}(B_{\>\!\!R})} \lessapprox\prod_{i = m-1}^{n-1} r_{i}^{X_{i}}D_{i}^{Y_{i}+O(\delta)}  \|f\|_{\infty}
\end{equation*}
where
\begin{align*}
    X_{i}&:= \frac{\beta_{i+1}-\beta_{i}}{2}-\frac{1+(n-i)\gamma_{i}}{2}\Big(\frac{1}{2}-\frac{1}{p_n}\Big); \\
    Y_{i}&:= \frac{\beta_{i+1}}{2} - \Big(1+i\big(1-\sum_{j=m}^{i} \gamma_j\big)\Big)\Big(\frac{1}{2}-\frac{1}{p_n}\Big).
\end{align*}
As in Section~\ref{Guth revisited section}, one chooses the various exponents so as to ensure $X_{i}, Y_{i} \leq 0$ for all $m \leq i \leq n-1$ and $Y_{m-1} \leq 0$. Owing to the extra degrees of freedom offered by the $\gamma_j$ parameters, in this case one may in fact saturate all the conditions: that is, the parameters may be chosen so as to ensure that $X_{i} = Y_{i} = 0$. Indeed, the condition $X_{i} = 0$ is equivalent to
\begin{equation}\label{our r constraint}
\Big(\frac{1}{2} - \frac{1}{p_{i+1}}\Big)^{-1}-\Big(\frac{1}{2} - \frac{1}{p_{i}}\Big)^{-1} = 1 + (n-i)\gamma_i
\end{equation}
whilst the condition $Y_{i-1} = 0$ is equivalent to 
\begin{equation}\label{our D constraint}
    \Big(\frac{1}{2} - \frac{1}{p_{i}}\Big)^{-1} = 2i-2(i-1)\sum_{j=m}^{i-1} \gamma_j
\end{equation}
Once again, choose $p_m := \frac{2m}{m -1}$ so that \eqref{our D constraint} holds in the $i = m$ case. The remaining~$p_{i}$ are then defined in terms of the $\gamma_j$ by the equation
\begin{equation}\label{linear system 1} \Big(\frac{1}{2} - \frac{1}{p_{i}}\Big)^{-1}=m + i + \sum_{j=m}^{i-1}(n-j)\gamma_j
\end{equation} 
so that each of the $n-m$ constraints in \eqref{our r constraint} is met. 

It remains to solve for the $n-m+1$ variables $\gamma_m, \dots, \gamma_n$; note that there are $n-m+1$ remaining constraints (in particular, there are $n-m$ constraints left over from \eqref{our D constraint} together with the condition that the $\gamma_{j}$ must sum to 1) and so the number of equations in our system equals the number of variables. By comparing the right-hand sides of \eqref{our D constraint} and \eqref{linear system 1}, it follows that
\begin{equation}\label{linear system 2}
  \sum_{j = m}^{i-1} (n - j +2i-2)\gamma_j = i - m  \qquad \textrm{for $m+1 \leq i \leq n$,}
\end{equation}
from which we read off that $\gamma_m=(n+m)^{-1}$.
To solve this linear system, let $\kappa_{i}$ denote the left-hand side of the equation in the above display and observe that 
\begin{equation*}
    \kappa_{i+1} -2\kappa_{i} + \kappa_{i-1} = (n+i)\gamma_{i}-(n+i-3)\gamma_{i-1} \qquad\textrm{ for $m+1 \leq i \leq n-1$,}
\end{equation*}
where $\kappa_m := 0$. On the other hand, by considering the right-hand side of \eqref{linear system 2}, it is clear that $\kappa_{i+1} -2\kappa_{i} +\kappa_{i-1} = 0$. Combining these observations gives a recursive relation for the $\gamma_j$ and from this one deduces that 
\begin{equation*}
    \gamma_j = \frac{1}{n+m} \prod_{i=m}^{j-1} \frac{n + i -2}{n + i + 1} = \frac{(n+m-1)(n+m-2)}{(n+j)(n+j-1)(n+j-2)}
\end{equation*}
for $m+1 \leq j \leq n-1$.
The remaining parameter $\gamma_n$ is then given by\footnote{To ensure this is a valid solution, one must verify that $\gamma_n \geq 0$ (so that $0 \leq \gamma_j \leq 1$ for all $m \leq j \leq n$). This property follows directly from the identity \eqref{telescoping sum identity} below.} 
\begin{equation*}
\gamma_n = 1 - \sum_{j=m}^{n-1}\gamma_j,
\end{equation*}
so that the $\gamma_j$ sum to 1.

It remains to check that these parameter values give the correct value of $p_n$, corresponding to the exponent $p_n(k)$ stated in Theorem~\ref{main theorem}. It follows from \eqref{our D constraint} that
\begin{equation}\label{computing p}
    \Big(\frac{1}{2} - \frac{1}{p_n}\Big)^{-1} = 2n -2(n-1)\sum_{j=m}^{n-1}\frac{(n+m-1)(n+m-2)}{(n+j)(n+j-1)(n+j-2)}.
\end{equation}
The expression on the right-hand side can be simplified by first writing the denominator in each summand as
\begin{equation*}
\frac{1}{(n+j)(n+j-1)(n+j-2)} = \frac{1}{2} \Big( \frac{1}{n+j-2} - \frac{1}{n+j-1}\Big) - \frac{1}{2} \Big( \frac{1}{n+j-1} - \frac{1}{n+j}\Big)    
\end{equation*}
and then using the resulting telescoping property of the sum. This yields the identity
\begin{equation}\label{telescoping sum identity}
\sum_{j=m}^{n-1}\frac{(n+m-1)(n+m-2)}{(n+j)(n+j-1)(n+j-2)}= \frac{1}{2}\Big(1 - \frac{(n+m-1)(n+m-2)}{(2n-1)2(n-1)}\Big).
\end{equation}
Plugging this into \eqref{computing p} and performing some simple algebraic manipulations, one concludes that
\begin{align*}
p_n
&=2+\frac{8(2n-1)}{n(5n+2m-9)+m(m-3) +4} \leq p_n(k)
\end{align*}
for $m \geq k$,  which completes the proof.\hfill $\Box$




\section{Final remarks}\label{final remarks section}

\begin{remark}\label{Demeter remark} One direction by which the argument could be improved would be to develop a more efficient mechanism for converting $k$-broad estimates into linear estimates than Proposition~\ref{9.1}. One such mechanism does indeed already exist and is described in the work of Bourgain--Guth (see the fourth section of~\cite{BG2011} or~\cite{Luca, Temur2014} for an alternative presentation of this method). In particular, Bourgain--Guth~\cite{BG2011} use Kakeya-type estimates to prove a stronger version of Proposition~\ref{9.1} in which the constraint $p \geq 2 + \frac{4}{2n-k}$ is slightly relaxed. Demeter~\cite{Demeter} used this approach (combined with recent advances on the Kakeya conjecture~\cite{GZ, Zahl2018}) to give the previous best range for the restriction problem in $\R^4$ (namely, $p > 2 + \frac{66642}{83303}$). In fact, using Theorem~\ref{main theorem} (and, in particular, the $3$-broad estimate in four dimensions with $p_4(3)=2+\frac{7}{9}$) one can slightly improve Demeter's result to $p > 2 + \frac{1407}{1759}$ via the same method. For other low dimensions the use of the more efficient Bourgain--Guth mechanism is limited due to the lack of understanding of the Kakeya problem in this regime. In high dimensions, however, stronger Kakeya maximal and $X$-ray transform estimates are available owing to the sum-difference approach to Kakeya, which was pioneered by Bourgain~\cite{Bourgain1999} and later honed by Katz--Tao~\cite{KT1999, Katz2002} and Oberlin \cite{Oberlin2010}. Potentially, improvements could be obtained in high dimensional cases using the more efficient Bourgain--Guth mechanism and the Kakeya-type estimates arising from sum-difference theory; however, since the computation of the various exponents is rather involved and any gain is likely to be very small, this has not been pursued here. 

\end{remark}

\begin{remark} An alternative approach to improving the range of restriction estimates would be to attempt to establish a stronger version of Theorem~\ref{main theorem}. This has been achieved for $n=3$ in the work of Wang~\cite{Wang} who showed that \eqref{broad} holds in the wider range $p > 3 + \frac{3}{13}$ in this case (this in turn implies the best-known result on the restriction problem in $\R^3$; see Figure~\ref{exponent table}). The proof of Wang's theorem relies on a careful analysis and exploitation of certain underlying geometric features of the restriction problem; it would be of interest to extend and incorporate this analysis into the study of higher dimensional situations. 
\end{remark}

\begin{remark}\label{general hypersurfaces remark}  It is not difficult to extend the methods of this article to treat the class of (compact pieces of) hypersurfaces with strictly positive principal curvatures, which includes the unit sphere $S^{n-1}$. To do this, one applies a standard argument to reduce considerations to hypersurfaces of \emph{elliptic-type}, as defined in~\cite{Moyua1999, TVV1998} (see also~\cite{Tao2003, Guth2016}). One may then appeal to the more general transverse equidistribution results of~\cite{GHI} in place of Lemma~\ref{transverse equidistribution lemma}. A more involved version of the Bourgain--Guth method for passing from $k$-broad to linear estimates is also required, but this already essentially appears in~\cite{BG2011} (see also~\cite{GHI}). For this it is useful to work with the class of elliptic-type hypersurfaces (rather than specific examples such as $S^{n-1}$), since this class is closed under parabolic rescaling.

On the other hand, the method breaks down when one considers general (compact pieces of) hypersurfaces of non-vanishing Gaussian curvature. For instance, for the prototypical example of a graph of a non-degenerate quadratic form, the transverse equidistribution estimate from Lemma~\ref{transverse equidistribution lemma} fails to hold in mixed signature cases (see~\cite{GHI} for further discussion of such phenomena). 
\end{remark}

\begin{remark} Another possible direction in which to strengthen the results would be to establish analogous estimates for Bochner--Riesz multipliers. An obvious approach to this would be to follow the classical Carleson--Sj\"olin argument~\cite{CS} (see also~\cite{Hormander1973}) which reduces the problem to establishing certain $L^p$ estimates for oscillatory integrals of the form
\begin{equation}\label{Carleson Sjolin operator}
    T^{\lambda}f(x) := \int_{\R^n} e^{i\lambda|x-y|} a(x,y)f(y)\,\ud y,
\end{equation}
where $a$ is some smooth, compactly supported amplitude. Here the key difficulty is to obtain a favourable dependence in the inequality on the parameter $\lambda \gg 1$. After fixing one of the components of $y$ and scaling, one obtains an operator which can be thought of as a perturbed version of $Ef$. The problem is then to show that the arguments used to study $Ef$ are stable under perturbation; see~\cite{Lee2006, BG2011, GHI} for recent examples of this approach, producing the current best-known results for the Bochner--Riesz problem.  

Again it is useful to work with a class of oscillatory integral operators which is closed under rescaling, rather than just the specific example arising from the Bochner--Riesz problem.  Here some care is needed, however: for a natural class of variable coefficient operators which extends the family of extension operators associated to positively-curved hypersurfaces, the desired $L^p$ estimates are \emph{false} for the range of $p$ featured in this article. Counterexamples of this kind first appeared in work of Bourgain~\cite{Bourgain1991a} and were further studied in~\cite{Bourgain1995, MS, Wisewell2005, BG2011} (see also~\cite{GHI}). For instance, Minicozzi and Sogge~\cite{MS} considered the analogue of \eqref{Carleson Sjolin operator} defined over a compact Riemannian manifold $(M,g)$ given by
\begin{equation}\label{Carleson Sjolin manifold}
    T^{\lambda}f(x) := \int_{M} e^{i\lambda\mathrm{dist}_g(x,y)} a(x,y) f(y)\,\ud y,
\end{equation}
where $\mathrm{dist}_g$ is the Riemannian distance function on $M$. These operators arise naturally in the study of Bochner--Riesz multipliers on compact manifolds, defined with respect to the spectral decomposition of the Laplace--Beltrami operator (see, for instance, \cite[Chapter 5]{Sogge2017}). In~\cite{MS} examples of $(M,g)$ were found for which the desired $L^p$ estimates for \eqref{Carleson Sjolin manifold} could only hold for a relatively small range of $p$. Sharp inequalities for such examples were later established in the work of Guth, Iliopoulou and the first author~\cite{GHI}. The problematic behaviour for certain $M$ can be attributed to the fact that analogues of the polynomial Wolff axioms can fail to hold for families of geodesic tubes relevant to the study of $T^{\lambda}$.
\end{remark}

\begin{remark} 
It is well-known that $L^p$-estimates for the extension operator imply bounds for the Kakeya maximal function. Let $\mathbf{T}$ be a collection of direction-separated $R$-tubes in $\R^n$, with angle at least $R^{-1/2}$ between each pair of tubes. If the estimate 
\begin{equation}\label{extension estimate}
    \|Ef\|_{L^p(\R^n)} \lesssim \|f\|_{L^p(\R^{n-1})}
\end{equation}
is valid for some $p > 2$, then
\begin{equation}\label{Kakeya estimate}
\Big\|\sum_{T\in \mathbf{T}} \mathbf{1}_{T}\Big\|_{L^{p/2}(\R^n)}\lesssim R^{n-1-\frac{2n}{p}} \Big(\sum_{T\in \mathbf{T}} |T|\Big)^{2/p};
\end{equation}
see, for example, ~\cite{Wolff1999} for a proof of this fact.
New estimates for the Kakeya maximal operator with $n=9$ are obtained by plugging in our estimates for the extension operator. For other values of $n$ the maximal function estimates that arise in this way are strictly weaker than those previously obtained by Wolff~\cite{Wolff1995} or Katz--Tao~\cite{Katz2002}.

Maximal inequalities such as \eqref{Kakeya estimate} imply lower bounds on the dimensions of Kakeya sets. Recall that a set $K\subset \R^n$ is \emph{Kakeya} if it is compact and it contains a unit line segment in every direction. Let $d(n)$ denote the infimum of the Hausdorff dimensions of Kakeya sets in $\R^n$; explicitly,
\begin{equation*}
    d(n) := \inf\{\dim K : K \subset \R^n \textrm{ Kakeya}\}.
\end{equation*}
The Kakeya conjecture then asserts that $d(n) = n$. As is well-known, the inequality~\eqref{Kakeya estimate}  implies that
\begin{equation*}
    d(n) \geq \frac{2p}{p-2} - n.
\end{equation*}
However, the aforementioned maximal inequality is not strong enough to improve over the existing lower bounds of Katz--Tao~\cite{Katz2002} for the Hausdorff dimension of Kakeya sets, obtained via the sum-difference method.

In terms of the asymptotic perspective espoused in this article, if \eqref{extension estimate} holds for $p = 2 + \lambda n^{-1} + O(n^{-2})$, then 
\begin{equation*}
    d(n) \geq \frac{4 - \lambda}{\lambda} n + O(1).
\end{equation*}
Taking $\lambda$ to be the value given by Theorem~\ref{asymptotic theorem}, it follows that 
\begin{equation*}
    \frac{4 - \lambda}{\lambda} = 4-2\sqrt{3}=0.535...,
\end{equation*}
which provides a high dimensional improvement over the classical $d(n) \geq \frac{n+2}{2}$ bound of Wolff~\cite{Wolff1995}. Once again, this does not improve the results of Katz--Tao~\cite{Katz2002}. Nevertheless, it seems of interest that one can go beyond the $d(n) \geq \frac{n}{2} + O(1)$ range for the Kakeya problem using a different approach than the sum-difference method, and that oscillatory methods are becoming more effective in the Kakeya problem.

We have since obtained further bounds for the Kakeya conjecture by applying similar arguments to those of this article directly in that context~\cite{HR}. 
\end{remark}





\bibliographystyle{amsplain}

\end{document}